\documentclass[12pt,a4paper,leqno,fleqn,twoside]{article}
\usepackage{helvet,times,amssymb}
\usepackage[latin1]{inputenc}
\usepackage{flafter}
\usepackage{latexsym}
\usepackage{amsthm}
\usepackage{amsmath, amssymb, ifthen, enumerate}

\begin{document}

\newtheorem{notation}{Notation}[section]
\newtheorem{definition}[notation]{Definition}
\newtheorem{satz}[notation]{}
\newtheorem{lemma}[notation]{Lemma}
\newtheorem{corollary}[notation]{Corollary}
\newtheorem{theorem}[notation]{Theorem}
\newtheorem{remark}[notation]{Remark}
\newtheorem{hyp}[notation]{Hypothesis}
\newtheorem{mdefinition}{Definition}
\newtheorem{mtheorem}{Theorem}

\def \<{\langle }
\def \>{\rangle }
\newcommand{\N}{\mathbb{N}}
\newcommand{\sy}{\textrm{Syl}}
\newcommand{\ov}{\overline}
\newcommand{\m}{\mathcal}
\newcommand{\F}{\mathcal{F}}
\newcommand{\ch}{\;char_{\mathcal{F}}\;}
\newcommand{\Ac}{A^\circ}
\newcommand{\Vc}{V^\circ}
\newcommand{\FN}{\F_\m{N}}
\newcommand{\FNT}{\F_\m{N}^*}
\newcommand{\Q}{\mathcal{Q}}

\title{Minimal Fusion Systems with a Unique Maximal Parabolic}
\author{Ellen Henke\footnote{The author was supported by EPSRC.}
}
\date{
Department of Mathematical Sciences\\
University of Copenhagen\\
Universitetsparken 5\\ 
2100 K{\o}benhavn {\O}\\
Denmark\\
Email: henke@math.ku.dk}

\maketitle

\thispagestyle{empty}

\thispagestyle{empty}

\textbf{\Large Abstract}

\bigskip

We define minimal fusion systems in a way that every non-solvable fusion system has a section which is minimal. Minimal fusion systems can also be seen as analogs of Thompson's N-groups. 
In this paper, we consider a minimal fusion system $\F$ on a finite $p$-group $S$ that has a unique maximal $p$-local subsystem containing $N_\F(S)$. For an arbitrary prime $p$, we determine the structure of a certain (explicitly described) $p$-local subsystem of $\F$. If $p=2$, this leads to a complete classification of the fusion system $\F$.

\section{Introduction}

A pattern for the classification of finite simple groups was set by Thompson in \cite{T}, where he gave a classification of all finite simple \textit{N-groups}. These are non-abelian finite simple groups with the property that every $p$-local subgroup is solvable, for every prime $p$.
Recall that a \textit{$p$-local subgroup} of a finite group $G$ is the normalizer of a non-trivial $p$-subgroup of $G$. Thompson's work was generalized by Gorenstein and Lyons, Janko and Smith to \textit{(N2)-groups}, that is to non-abelian finite simple groups all of whose $2$-local subgroups are solvable. Recall here that, by the Theorem of Feit-Thompson, every non-solvable group has even order.

\bigskip

N-groups play an important role, as every minimal non-solvable finite group is an N-group. Furthermore, every non-solvable group has a section which is an N-group. The respective properties hold also for (N2)-groups.

\bigskip

A new proof for the classification of (N2)-groups was given by Stellmacher in \cite{St2}. It uses the amalgam method, which is a completely local method. 
Currently, Aschbacher is working on another new proof for the classification of (N2)-groups. His approach uses \textit{saturated fusion systems} that were first introduced by Puig under the name of \textit{full Frobenius categories}. Aschbacher's plan is to classify all \textit{N-systems}, i.e. all saturated fusion systems $\F$ of characteristic $2$-type such that the group $Mor_\F(P,P)$ is solvable, for every subgroup $P$ of $\F$. Here the use of the group theoretical concept of solvability fits with the definition of solvable fusion systems as introduced by Puig. However, this concept seems not general enough to ensure that N-systems play the same role in saturated fusion systems as N-groups in groups. Therefore, in our notion of \textit{minimal fusion systems} introduced below, we find it necessary to use a concept of solvable fusion systems as defined by Aschbacher \cite[15.1]{A2}.

\bigskip

\textbf{For the remainder of the introduction let $\mathbf{p}$ be a prime and $\mathbf{\F}$ be a saturated fusion system on a finite $\mathbf{p}$-group $\mathbf{S}$.} We adapt the standard terminology regarding fusion systems as introduced by Broto, Levi and Oliver \cite{BLO}.
For further basic definitions and notation we refer the reader to Section~\ref{Fs}. Generic examples of saturated fusion systems are the fusion systems $\F_S(G)$, where $G$ is a finite group containing $S$ as a Sylow $p$-subgroup, the objects of $\F_S(G)$ are all subgroups of $S$, and the morphisms in $\F_S(G)$ between two objects are the injective group homomorphisms obtained by conjugation with elements of $G$.

\begin{definition}\label{MinDef}
The fusion system $\F$ is called minimal if $O_p(\F)=1$ and $N_\F(U)$ is solvable for every fully normalized subgroup $U\neq 1$ of $\F$.
\end{definition}

Here the fusion system $\F$ is \textit{solvable}, if and only if $O_p(\F/R)\neq 1$, for every strongly closed subgroup $R\neq S$ of $\F$. This implies that indeed every minimal non-solvable fusion system is minimal in the sense defined above. Furthermore, every non-solvable fusion system has a section which is minimal. Therefore, minimal fusion systems play a similar role in saturated fusion systems as N-groups in groups. However, a classification of minimal fusion systems seems a difficult generalization of the original N-group problem. One reason is that in fusion systems the prime $2$ does not play such a distinguished role as in groups. Therefore, we would like to treat minimal fusion systems also for odd primes as far as possible. Secondly, the notion of solvability in fusion systems is more general than the group theoretical notion. More precisely, although it turns out that every solvable fusion system is constrained and therefore the fusion system of a finite group, such a group can have certain composition factors that are non-abelian finite simple groups. Aschbacher showed in \cite{A2} that these are all finite simple groups in which fusion is controlled in the normalizer of a Sylow $p$-subgroup. Furthermore, Aschbacher gives a list of these groups. Generic examples are the finite simple groups of Lie type in characteristic $p$ of Lie rank $1$. For odd primes, Aschbacher's proof of these facts requires the complete classification of finite simple groups. For $p=2$ they follow already from Goldschmidt's Theorem on groups with a strongly closed abelian subgroup (see \cite{Gold}).

\bigskip

In this paper, we use a concept which is an analog to the (abstract) concept of parabolics in finite group theory, where a \textit{parabolic subgroup} is defined to be a $p$-local subgroup containing a Sylow $p$-subgroup. This generalizes the definition of parabolics in finite groups of Lie type in characteristic $p$. Suppose $S$ is a Sylow $p$-subgroup of a finite group $G$. It is a common strategy in the classification of finite simple groups and related problems to treat separately the case of a unique maximal (with respect to inclusion) parabolic containing $S$. In this case, one classifies as a first step a $p$-local subgroup of $G$ which has the pushing up property as defined in Section~\ref{PUGr}. In the remaining case, two distinct maximal parabolics containing $S$ form an amalgam of two groups that do not have a common normal $p$-subgroup. This usually allows an elegant treatment using the coset graph, and leads in the generic cases to a group of Lie type and Lie rank at least $2$. The main result of this paper handles the fusion system configuration which loosely corresponds to the pushing up case in the N-group investigation. We next introduce the concept of a parabolic in fusion systems.

\begin{definition}\label{ParDef}
\begin{itemize}
\item A subsystem of $\F$ of the form $N_\F(R)$ for some non-trivial normal subgroup $R$ of $S$ is called a \textbf{parabolic subsystem} of $\F$, or in short, a parabolic.
\item A \textbf{full parabolic} is a parabolic containing $N_\F(S)$. It is called a \textbf{full maximal parabolic}, if it is not properly contained in any other parabolic subsystem of $\F$.
\end{itemize}
\end{definition}

Thus, in this paper, we treat the case of a minimal fusion system having a unique full maximal parabolic. Note that this assumption is slightly more general than just supposing that a minimal fusion system has a unique maximal parabolic. Even more generally, we will in fact assume only that there is a proper saturated subsystem containing every full maximal parabolic. We will use the following notation.

\begin{notation}\label{ENDef}
Let $\m{N}$ be a subsystem of $\F$ on $S$. We write $\F_\m{N}$ for the set of centric subgroups $Q$ of $\F$ for which there exists an element of $Mor_\F(Q,Q)$ that is not a morphism in $\m{N}$.
\end{notation}

Note here that, if $\m{N}$ is a proper subsystem of $\F$, we get as a consequence of Alperin's Fusion Theorem that the set $\FN$ is non-empty. In our investigation we focus on members of $\F_\m{N}$ that are maximal in the sense defined next.

\begin{definition}\label{ThMax}
\begin{itemize}
\item For every subgroup $P$ of $S$ write $m(P)$ for the $p$-rank of $P$, i.e. for the largest integer $m$ such that $P$ contains an elementary abelian subgroup of order $p^m$.. 
\item Let $\m{E}$ be a set of subgroups of $S$. An element $Q$ of $\m{E}$ is called \textbf{Thompson-maximal} in $\m{E}$ if, for every $P\in\m{E}$, $m(Q)\geq m(P)$ and, if $m(Q)= m(P)$, then $|J(Q)|\geq |J(P)|$. 
\end{itemize}
\end{definition}

Here, for a finite group $G$, the \textit{Thompson subgroup} $J(G)$ (for the prime $p$) is the subgroup of $G$ generated by the elementary abelian $p$-subgroups of $G$ of maximal order. As a first step in our investigation we show the existence of \textit{Thompson-restricted subgroups}. These are subgroups of $S$ whose normalizer in $\F$ has a very restricted structure and involves $SL_2(q)$ acting on a natural module. More precisely, Thompson-restricted subgroups are defined as follows.

\begin{definition}\label{ThRestrictDef}
Let $Q\in \F$ be centric and fully normalized. Set $T:=N_S(Q)$ and let $G$ be a model for $N_\F(Q)$. We call such a subgroup $Q$ \textbf{Thompson-restricted} if, for every normal subgroup $V$ of $J(G)T$ with $\Omega(Z(T))\leq V\leq \Omega(Z(Q))$, the following hold:
\begin{itemize}
\item[(i)]  $N_S(J(Q))=T$ and $J(Q)$ is fully normalized..
\item[(ii)] $C_S(V)=Q$ and $C_{G}(V)/Q$ is a $p^\prime$-group.
\item[(iii)] $J(G)/C_{J(G)}(V)\cong SL_2(q)$ for some power $q$ of $p$, and $V/C_V(J(G))$ is a natural $SL_2(q)$-module for $J(G)/C_{J(G)}(V)$.
\item[(iv)] $C_{T}(J(G)/C_{J(G)}(V))\leq Q$.
\end{itemize} 
\end{definition}

Here a \textit{model} for $\F$ is a finite group $G$ containing $S$ as a Sylow $p$-subgroup such that $C_G(O_p(G))\leq O_p(G)$ and $\F=\F_S(G)$. By a Theorem of Broto, Castellana, Grodal, Levi and Oliver \cite{BCGLO}, there exists a (uniquely determined up to isomorphism) model for $\F$ provided $\F$ is constrained. Here $\F$ is called \textit{constrained} if $\F$ has a normal $p$-subgroup containing its centralizer in $S$. For every fully normalized, centric subgroup $Q$ of $\F$, the normalizer $N_\F(Q)$ is a constrained saturated subsystem of $\F$. This makes it possible in Definition~\ref{ThRestrictDef}  to choose a model for $N_\F(Q)$. For the definition of a \textit{natural $SL_2(q)$-module} see Definition~\ref{NatSL2qM}. 

\bigskip

Crucial in our proof is the following theorem that requires neither the minimality of $\F$, nor the existence of a proper saturated subsystem containing every full maximal parabolic.

\begin{mtheorem}\label{mainSL2qThm}
Let $\m{N}$ be a proper saturated subsystem of $\F$ containing $C_\F(\Omega(Z(S)))$ and $N_\F(J(S))$. Then there exists a Thompson-maximal subgroup $Q$ of $\F_\m{N}$ such that $Q$ is Thompson-restricted.
\end{mtheorem}

The proof of Theorem~\ref{mainSL2qThm} can be found in Section~\ref{chapterAut}. It uses FF-module results of Bundy, Hebbinghaus, Stellmacher \cite{BHS}. Apart from that, the proof is self-contained. In particular, it is  possible to avoid the use of the classification of finite simple groups or any kind of $\m{K}$-group hypothesis in the proof of Theorem~\ref{mainSL2qThm} and, in fact, in the proof of all the theorems in this paper.

\bigskip

Note that $N_\F(\Omega(Z(S)))$ and $N_\F(J(S))$ are full parabolics of $\F$, as $\Omega(Z(S))$ and $J(S)$ are characteristic in $S$. In particular, if $\m{N}$ is a proper saturated subsystem of $\F$ containing every full parabolic, then $\m{N}$ fulfills the hypothesis of Theorem~\ref{mainSL2qThm}. Hence, there exists a Thompson-maximal subgroup $Q$ of $\FN$ such that $Q$ is Thompson-restricted. As we show in the next theorem, the fusion system $\F$ being minimal implies for each such $Q$ that $N_\F(Q)$ has a very simple structure.

\newcounter{repeat}
\setcounter{repeat}{\value{mtheorem}}
\begin{mtheorem}\label{CasesL}
Let $\F$ be minimal and let $\m{N}$ be a proper saturated subsystem of $\F$ containing every full parabolic. Let $Q\in\FN$ such that $Q$ is Thompson-restricted and Thompson-maximal in $\FN$. Let $G$ be a model for $N_\F(Q)$ and $M=J(G)$. Then $N_S(X)=N_S(Q)$, for every non-trivial normal $p$-subgroup $X$ of $MN_S(Q)$. Moreover, $Q\leq M$, $M/Q\cong SL_2(q)$, and one of the following holds:
\begin{itemize}
\item[(I)] $Q$ is elementary abelian, and $Q/C_Q(M)$ is a natural $SL_2(q)$-module for $M/Q$, or
\item[(II)] $p=3$, $S=N_S(Q)$ and $|Q|= q^5$. Moreover, $Q/Z(Q)$ and $Z(Q)/\Phi(Q)$ are natural $SL_2(q)$-modules for $M/Q$, and $\Phi(Q)=C_Q(M)$.
\end{itemize}
\end{mtheorem}

The proof of Theorem~\ref{CasesL} can be found in Chapter~\ref{PUFs} and is self-contained. For  $p=2$, Theorem~\ref{mainSL2qThm} and Theorem~\ref{CasesL} lead to a complete classification of the fusion system $\F$. This is a direct consequence of a more general result (Theorem~\ref{ClassifyG}) on fusion systems of characteristic $2$-type that we prove in Chapter~\ref{ClassifyGSection}. This proof relies on a  group theoretical result (Theorem~\ref{MainAmalgamThm}) from Chapter~\ref{AmalgamSection}. It uses a special case of the classification of weak BN-pairs of rank $2$ from \cite{DGS} (see Theorem~\ref{ThmA}), and is apart from that self-contained. However, many of our arguments are similar to the ones in \cite{A3}. In fact, using the Odd Order Theorem of Feit--Thompson and the above mentioned theorem of Goldschmidt on groups with a strongly closed $2$-group, the following  classification for $p=2$ could also be obtained as a consequence of \cite{A3}. However, we prefer in this paper to give a proof that does not rely on these theorems. In particular, our proof needs only methods and results from local group theory, whereas Goldschmidt's theorem relies heavily on Glaubermann's $Z^*$-Theorem, whose proof uses modular representation theory.

\begin{mtheorem}\label{Classify}
Assume $p=2$, $\F$ is minimal, and there is a proper saturated subsystem of $\F$ containing every full parabolic of $\F$. Then there is a finite group $G$ containing $S$ as a Sylow $2$-subgroup such that $\F\cong\F_S(G)$ and one of the following holds:
\begin{itemize}
\item[(a)] $S$ is dihedral of order at least $16$, and $G\cong L_2(r)$ or $PGL_2(r)$ for some odd prime power $r$.
\item[(b)] $S$ is semidihedral, and $G$ is an extension of $L_2(r^2)$ by an automorphism of order $2$, for some odd prime power $r$.
\item[(c)] $S$ is semidihedral of order $16$, and $G\cong L_3(3)$.
\item[(d)] $|S|=32$, and $G\cong Aut(A_6)$ or $Aut(L_3(3))$.
\item[(e)] $|S|=2^7$ and $G\cong J_3$.
\item[(f)] $F^*(G)\cong L_3(q)$ or $Sp_4(q)$, $|O^2(G):F^*(G)|$ is odd and $|G:O^2(G)|=2$. Moreover, if $F^*(G)\cong Sp_4(q)$ then $q=2^e$ where $e$ is odd.
\end{itemize}
\end{mtheorem}

Throughout this paper, we write mappings on the right side. By $p$ we will always denote a prime. In our notation and terminology regarding fusion systems we mostly follow \cite{BLO}. The reader can find further basic definitions in Section \ref{Fs}. We adapt the group theoretic notions from \cite{KS}.
In particular, we define a finite group to be $p$-closed if it has a normal Sylow $p$-subgroup.
Moreover, for a normal subgroup $N$ of $G$, we will often make use of the so called ``bar''-notation. This means that, after setting $\ov{G}=G/N$, we write $\ov{U}$ (respectively $\ov{g}$) for the image of a subgroup $U$ of $G$ (respectively, an element $g\in G$) in $\ov{G}$.

\section{Saturated fusion systems}\label{Fs}

Let $G$ be a group. Write $Inn(G)$ for the group of inner automorphisms of $G$. For $g\in G$, denote by $c_g:G\rightarrow G$ the inner automorphism of $G$ determined by $g$. Let $P$ and $Q$ be subgroups of $G$. 
For any map $\phi:P\rightarrow Q$, $A\leq P$ and $A\phi\leq B\leq G$, write $\phi_{|A,B}$ for the map with domain $A$ and range $B$ mapping each element of $A$ to its image under $\phi$.  We will frequently use the following notation.

\begin{notation}\label{SP}
For subgroups $P$ and $R$ of $G$ set
$$R_P:=Aut_R(P):=\{{c_g}_{|P,P}:g\in N_R(P)\}.$$
\end{notation}

We adapt the basic definitions and notation related to fusion systems from \cite{BLO}. \textbf{From now on let $S$ be a finite $p$-group and $\F$ be a fusion system on $S$.} If $G$ contains $S$ as a subgroup, then we write $\F_S(G)$ for the fusion system on $S$ whose morphisms are the conjugation maps ${c_g}_{|P,Q}$ with $P,Q\leq S$ and $g\in G$ such that $P^g\leq Q$.
By an abuse of notation we denote by $\mathcal{F}$ also the set of all objects of $\mathcal{F}$. In particular, we write $Q\in \mathcal{F}$ instead of $Q\leq S$. By $P^\mathcal{F}$ we denote the $\F$-conjugacy class of $P$. We will refer to the following elementary property:

\begin{remark}\label{FullNormQ}
Let $Q\in \F$ and let $U$ be a characteristic subgroup of $Q$. Assume $U$ is fully normalized in $\F$ and $N_S(U)=N_S(Q)$. Then $Q$ is fully normalized.
\end{remark}

\begin{proof}
Let $P\in Q^\F$ and $\phi\in Mor_\F(Q,P)$. Then $U\phi\in U^\F$ and $U\phi$ is characteristic in $P=Q\phi$. So, as $U$ is fully normalized, we get
$$ |N_S(P)|\leq |N_S(U\phi)|\leq |N_S(U)|=|N_S(Q)|.$$
Hence, $Q$ is fully normalized.
\end{proof}

For a fusion system $\m{E}$ on a subgroup $T$ of $S$, we write $\m{E}\leq \F$ if $\m{E}$ is a subsystem of $\F$, i.e. if $Mor_{\m{E}}(P,Q)\subseteq Mor_\F(P,Q)$ for all $P,Q\leq T$. For $\m{E}\leq\F$, $P\in \m{E}$ and $L\leq Aut_\F(P)$, we write $L\leq \m{E}$ to indicate that $L\leq Aut_{\m{E}}(P)$, and $L\not\leq \m{E}$ if the converse holds. We will also use the following notation:

\begin{notation}\label{phi*}
\begin{itemize}
\item For every $P\in\F$ set
$$Aut_{\mathcal{F}}(P)=Mor_{\mathcal{F}}(P,P).$$
\item For $P,Q\in \F$ and an isomorphism $\phi\in Mor_\F(P,Q)$ we write $\phi^*$ for the map
$$\phi^*: Aut_\F(P)\rightarrow Aut_\F(Q)\mbox{ defined by }\alpha\mapsto \phi^{-1}\alpha\phi.$$
\item If $P\leq A\leq S$, $Q\leq B\leq S$ and $\phi\in Mor_\F(A,B)$ such that $\phi_{|P,Q}$ is an isomorphism, then we sometimes write $\phi^*$ instead of $(\phi_{|P,Q})^*$.
\end{itemize}
\end{notation}

\textbf{For the remainder of this section assume that $\mathbf{\F}$ is saturated.} To ease notation we set 
$$A(P):=Aut_\F(P),\mbox{ for every }P\in\F.$$

\begin{remark}\label{SatAut}
Let $\phi\in A(U)$ and $U\leq X\leq N_S(U)$.
\begin{itemize}
\item[(a)] If $\phi$ extends to a member of $A(X)$ then $X_U\phi^*=X_U$.
\item[(b)] Assume $C_S(U)\leq X$ and $U$ is fully centralized. Then $X_U\phi^*=X_U$ if and only if $\phi$ extends to a member of $A(X)$.
\end{itemize} 
\end{remark}

\begin{proof}
An elementary calculation shows (a), and (b) is a consequence of (a) and the saturation axioms.
\end{proof}

\begin{lemma}\label{fullNormChar}
Let $Q\in\F$. Then $Q$ is fully normalized if and only if, for each $P\in Q^{\mathcal{F}}$, there exists a morphism
$\phi\in Mor_{\mathcal{F}}(N_S(P),N_S(Q))$ such that $P\phi=Q$.
\end{lemma}

\begin{proof}
See for example \cite[2.6]{L}.
\end{proof}

A subgroup $P\in \F$ is called \textbf{normal} in $\F$ if $\F=N_\F(P)$. Observe that the product of two normal subgroups of $\F$ is again normal in $\F$. Hence, there is a largest normal subgroup of $\F$ which we denote by $O_p(\F)$. The fusion system $\F$ is called \textbf{constrained} if $O_p(\F)$ is centric. 
A \textbf{model} for $\F$ is a finite group $G$ containing $S$ as a Sylow $p$-subgroup such that $G$ has characteristic $p$ (i.e. $C_G(O_p(G))\leq O_p(G)$), and $\F=\F_S(G)$.

\begin{theorem}[Broto, Castellana, Grodal, Levi and Oliver]\label{ModExUnique}
A fusion system $\F$ is constrained if and only if there is a model for $\F$. Furthermore, if $G$ and $H$ are models for $\F$ then there exists an isomorphism $\phi: G\rightarrow H$ such that $\phi$ is the identity on $S$.
\end{theorem}

\begin{proof}
This is Proposition~C in \cite{BCGLO}.
\end{proof}

Observe that, if $G$ is a model for $\F$, then $O_p(\F)=O_p(G)$. In the notation we introduce next we follow Aschbacher \cite{A3}.

\begin{notation}\label{ModNot}
Let $P\in \F$ such that $P$ is fully normalized and $N_\F(P)$ is constrained. Then by $G(P)$ we denote a model for $N_\F(P)$. 
\end{notation}

Note here that, by Theorem~\ref{ModExUnique}, $G(P)$ exists and is uniquely determined up isomorphism.

\bigskip

Define $P$ to be \textbf{strongly closed} in $\F$ if, for all $A,B\leq S$ and every morphism $\phi\in Mor_\F(A,B)$, $(A\cap P)\phi\leq P$.
Note that every normal subgroup of $\F$ is strongly closed in $\F$, and every strongly closed subgroup of $\F$ is normal in $S$.
Given a strongly closed subgroup $R$ of $\F$, Puig defined a \textbf{factor system} $\F/R$ which is a fusion system on $S/R$. Here for subgroups $A,B$ of $S$ containing $R$, the morphisms in $Mor_{\F/R}(A/R,B/R)$ are just the maps induced by the elements of $Mor_\F(A,B)$.

\begin{theorem}[Puig]
Let $R$ be strongly closed in $\F$. Then $\F/R$ is a saturated fusion system on $S/R$.
\end{theorem}

\begin{proof}
This follows from \cite[6.3]{Pu}.
\end{proof}

In our definition of solvable fusion systems we follow Aschbacher \cite{A2}, who defined $\F$ to be solvable if every composition factor of $\F$ is the fusion system of the group of order $p$. However, as we have not defined normal subsystems and composition factors in this paper, we prefer to give the definition in the language we introduced. Aschbacher \cite[15.2,15.3]{A2} has shown his definition to be equivalent to the following.

\begin{definition}\label{SolvProp}
The fusion system $\F$ is \textbf{solvable} if and only if $O_p(\F/R)\neq 1$ for every strongly closed subgroup $R\neq S$ of $\F$.
\end{definition}

We will use the following properties of solvable fusion systems.

\begin{theorem}[Aschbacher]\label{SubSolv}
Let $\F$ be solvable. 
\begin{itemize}
\item[(a)] Every saturated subsystem of $\F$ is solvable.
\item[(b)] $\F$ is constrained.
\end{itemize}
\end{theorem}

Following Aschbacher \cite{A3}, we define $\F$ to be of \textbf{characteristic $p$-type} if $N_\F(P)$ is constrained, for every fully normalized subgroup $P\in\F$.
Recall from the introduction that we call $\F$ \textbf{minimal} if $N_\F(P)$ is solvable, for every fully normalized subgroup $P\in\F$. 
As a consequence of Theorem~\ref{SubSolv}(b) we get:

\begin{corollary}\label{MinCharpType}
If $\F$ is minimal then $\F$ is of characteristic $p$-type.
\end{corollary}

Next we state Alperin--Goldschmidt Fusion Theorem and some of its consequences. The following definition is crucial.

\begin{definition}
A subgroup $Q\in\F$ is called \textbf{essential} if $Q$ is centric and $A(Q)/Inn(Q)$ has a strongly $p$-embedded subgroup.
\end{definition}

Recall that a proper subgroup $H$ of a finite group $G$ is called \textbf{strongly $p$-embedded} if $p$ divides the order of $H$, and the order of $H\cap H^g$ is not divisible by $p$ for every $g\in G\backslash H$. It is elementary to check that every $\F$-conjugate of an essential subgroup is again essential. This allows to refer to \textbf{essential classes} meaning the $\F$-conjugacy classes of essential subgroups.

\begin{theorem}[The Alperin--Goldschmidt Fusion Theorem, Puig]\label{AlpGold1}
Let $\m{C}$ be a set of subgroups of $S$ such that $S\in \m{C}$ and $\m{C}$ intersects non-trivially with every essential class. Then, for all $P,Q\leq S$ and every isomorphism $\phi\in Mor_{\mathcal{F}}(P,Q)$, there exist sequences of subgroups of $S$ 
$$P=P_0,P_1,\dots ,P_n=Q\;\mbox{in}\;\mathcal{F},\;\mbox{and}\;Q_1,\dots,Q_n\;\mbox{in}\;\mathcal{C}$$
and elements $\alpha_i\in A(Q_i)$ for $i=1,\dots,n$ such that $P_{i-1},P_i\leq Q_i$, $P_{i-1}\alpha_i=P_i$ and 
$$\phi=({\alpha_1}_{|P_0,P_1})({\alpha_2}_{|P_1,P_2})\dots ({\alpha_n}_{|P_{n-1},P_n}).$$
\end{theorem}

\begin{proof}
This is a direct consequence of \cite[5.2]{L} and \cite[2.10]{DGMP}.
\end{proof}

The proof of the following Lemma uses Theorem \ref{AlpGold1}.

\begin{lemma}\label{NormalEssential}
Let $N\in\F$, and let $\m{D}$ be a set of representatives of the essential classes of $\F$. Set $\m{C}=\{S\}\cup\m{D}$. 
Then $N$ is normal in $\F$ if and only if, for every $P\in \m{C}$, $N\leq P$ and $N$ is $A(P)$-invariant.
\end{lemma}

\begin{proof}
See \cite[2.17]{H}.
\end{proof}

\begin{lemma}\label{fullNormChar3}
Let $U\in \F$ such that $U$ is not fully normalized. Then $N_S(U)$ is contained in an essential subgroup of $\F$.
\end{lemma}

\begin{proof}
By Lemma~\ref{fullNormChar}, there is $\phi\in Mor_\F(N_S(U),S)$ such that $U\phi$ is fully normalized. As $U$ is not fully normalized, $\phi$ does not extend to an element of $A(S)$. Now by Theorem~\ref{AlpGold1}, there is $\psi\in A(S)$ such that $N_S(U)\psi$ is contained in an essential subgroup. Since every $\F$-conjugate of an essential subgroup is again essential, this yields the assertion. 
\end{proof}

Given a saturated fusion system $\tilde{\F}$ on a finite $p$-group $\tilde{S}$, we call a group isomorphism $\alpha:~S\rightarrow~\tilde{S}$ an \textbf{isomorphism} (of fusion systems) from $\F$ to $\tilde{\F}$ if for all subgroups $A,B$ of $S$, $$\alpha^{-1}Mor_\F(A,B)\alpha=Mor(A\alpha,B\alpha).$$ 
An immediate consequence of Theorem~\ref{AlpGold1} is the following remark.

\begin{remark}\label{AlpIso}
Let $\tilde{\F}$ be a saturated fusion system on a finite $p$-group $\tilde{S}$. Let $\m{E}$ be a set of representatives of the essential classes of $\F$ and $\m{C}=\m{E}\cup \{S\}$. Then a group isomorphism $\alpha:S\rightarrow \tilde{S}$ is an isomorphism between $\F$ and $\tilde{\F}$ if and only if $\{P\alpha:P\in\m{E}\}$ is a set of representatives of the essential classes of $\tilde{\F}$ and $\alpha^{-1}A(P)\alpha=Aut_{\tilde{\F}}(P\alpha)$ for every $P\in\m{C}$. 
\end{remark}

We conclude this section with some results that are important in Section~\ref{chapterAut} where we show the existence of Thompson-restricted subgroups. Recall from Notation~\ref{ENDef} that, for a subsystem $\m{N}$ of $\F$ on $S$, we write $\FN$ for the set of centric subgroups $Q$ of $\F$ for which there exists an element in $A(Q)$ that is not a morphism in $\m{N}$. Furthermore, we introduce the following notation.

\begin{notation}\label{FN*Def}
Let $\m{N}$ be a subsystem of $\F$ on $S$. Then we write $\FNT$ for the set of Thompson-maximal members of $\FN$.\footnote{Recall Definition~\ref{ThMax}.}
\end{notation}

We get the following three corollaries to Theorem~\ref{AlpGold1}.

\begin{corollary}\label{Cor1}
Let $\m{N}$ be a subsystem of $\F$ on $S$. Then $\FN\neq\emptyset$ if and only if $\F\neq\m{N}$.
\end{corollary}

\begin{corollary}\label{Cor2}
Let $\m{N}$ be a proper subsystem of $\F$ on $S$. Let $X\in \FNT$, $J(X)\leq R\leq S$ and $\phi\in Mor_\F(R,S)$. Then $J(R)=J(X)$ or $\phi\in \m{N}$.
\end{corollary}

\begin{proof}
Assume $\phi\not\in\m{N}$. Then by Theorem~\ref{AlpGold1}, $R$ is $\F$-conjugate to a subgroup of an element of $\FN$. Hence, as $J(X)\leq R$, the subgroup $X$ being Thompson-maximal in $\FN$ implies $J(R)=J(X)$.
\end{proof}

As a special case of Corollary~\ref{Cor2} we get

\begin{corollary}\label{Cor3}
Let $\m{N}$ be a proper subsystem of $\F$ on $S$. Let $X\in\FNT$ and $Q\in\F$ such that $J(X)\leq Q$ and $A(Q)\not\leq \m{N}$. Then $J(Q)=J(X)$.
\end{corollary}

\begin{remark}\label{ThomQ}
Let $Q\in\F$ such that $J(S)\not\leq Q$. Then $J(N_S(J(Q)))\not\leq Q$.
\end{remark}

\begin{proof}
Otherwise $J(N_S(J(Q)))=J(Q)$ and so $$N_S(N_S(J(Q)))\leq N_S(J(N_S(J(Q))))=N_S(J(Q)).$$ 
Then $S=N_S(J(Q))$ and so $J(S)\leq Q$, a contradiction.
\end{proof}

\begin{lemma}\label{JQFullNorm}
Let $\m{N}$ be a proper subsystem of $\F$ on $S$. Then there exists $X\in\FNT$ such that $J(X)$ is fully normalized.
\end{lemma}

\begin{proof}
By Corollary \ref{Cor1}, we can choose $X_0\in\FNT$. Set $U_0=J(X_0)$ and let $U\in U_0^\F$ be fully normalized. Then by Lemma~\ref{fullNormChar}, there exists $\phi\in Mor_\F(N_S(U_0),N_S(U))$ such that $U_0\phi=U$. Set $X:=X_0\phi$. 
If $J(S)\leq X_0$ then observe that $U_0=J(S)=U$, so $U_0$ is fully normalized. Therefore, we may assume that $J(S)\not\leq X_0$. Hence, by Remark~\ref{ThomQ}, $J(N_S(U_0))\not\leq X_0$. It follows now from Corollary \ref{Cor2} that $\phi$ is a morphism in $\m{N}$. Since $A(X_0)=\phi A(X)\phi^{-1}$ and $A(X_0)\not\leq \m{N}$, we get $A(X)\not\leq \m{N}$ and so $X\in \FN$.
As $X_0\in \FNT$ and $X\in X_0^\F$, it follows $X\in \FNT$. Since $J(X)=U$ is fully normalized, this shows the assertion.
\end{proof}

\begin{lemma}\label{AUnotleqN}
Let $\m{N}$ be a proper saturated subsystem of $\F$ containing $C_\F(\Omega(Z(S)))$. Let $X\in\FN$ such that $J(X)$ is fully normalized. Then $A(J(X))\not\leq \m{N}$. 
\end{lemma}

\begin{proof}
Set $U:=J(X)$ and assume $A(U)\leq \m{N}$. Let $\phi\in A(X)$ such that $\phi\not\in\m{N}$. Then $\alpha:=\phi_{|U,U}\in \m{N}$ and, by Remark~\ref{SatAut}(a), $X_U\alpha^*=X_U$.\footnote{Recall Notation~\ref{SP} and Notation~\ref{phi*}.} In particular, $X\leq N_\alpha$. Observe that $U$ is fully normalized in $\m{N}$. Hence, as $\m{N}$ is saturated, $\alpha$ extends to an element $\psi\in Mor_{\m{N}}(X,S)$. Note that $\phi^{-1}\psi$ is the identity on $U$. By definition of $\FN$, $X$ is centric and therefore $\Omega(Z(S))\leq X$. This yields $\Omega(Z(S))\leq J(X)=U$. Thus, $\phi^{-1}\psi\in C_\F(\Omega(Z(S)))\leq \m{N}$ and so $\phi\in \m{N}$, a contradiction. Hence, $A(U)\not\leq \m{N}$.
\end{proof}

\begin{lemma}\label{ThomRestrHelp0}
Let $\m{N}$ be a proper saturated subsystem of $\F$ containing $C_\F(\Omega(Z(S)))$. Then there exists $Q\in \FNT$ such that $N_S(J(Q))=N_S(Q)$ and $J(Q)$ is fully normalized.
\end{lemma}

\begin{proof}
By Lemma~\ref{JQFullNorm}, we can choose $X\in\FNT$ such that $U:=J(X)$ is fully normalized. Then, by Lemma \ref{AUnotleqN}, we have $A(U)\not\leq \m{N}$. Set $V:=\Omega(Z(U))$ and $Q:=C_S(V)\cap N_S(U)$. Then $U\leq Q$ and $N_S(U)\leq N_S(Q)$. Since $U$ is fully normalized, $S_U\in Syl_p(A(U))$. Hence, $Q_U\in Syl_p(C_{A(U)}(V))$ and, by a Frattini-Argument,
$$ A(U)=C_{A(U)}(V)N_{A(U)}(Q_U).$$
As $X$ is centric, we have $\Omega(Z(S))\leq J(X)=U$ and thus, $\Omega(Z(S))\leq V$. Therefore, $C_{A(U)}(V)\leq C_\F(\Omega(Z(S)))\leq \m{N}$ and so $N_{A(U)}(Q_U)\not\leq \m{N}$.
Since $C_S(U)\leq C_S(V)\cap N_S(U)=Q$, it follows from Remark~\ref{SatAut}(b) that every element of $N_{A(U)}(Q_U)$ extends to an element of $A(Q)$. Hence, $A(Q)\not\leq \m{N}$. Now Corollary \ref{Cor3} implies that $J(Q)=U=J(X)$. Hence, $N_S(U)\leq N_S(Q)\leq N_S(J(Q))=N_S(U)$ and, by Remark~\ref{FullNormQ}, $Q$ is fully normalized. In particular, $Q$ is fully centralized and therefore, as $C_S(Q)\leq Q$, centric. Thus, $Q\in \FNT$. This proves the assertion.
\end{proof}

\section{The Frattini subgroup}

Throughout this section let $G$ be a finite group. Recall that the \textbf{Frattini subgroup} $\Phi(G)$ of $G$ is the intersection of all maximal subgroups of $G$.

\begin{remark}\label{Frat0}
\begin{itemize}
\item[(a)] Let $H$ be a subgroup of $G$. If $G=H\Phi(G)$ then $G=H$.
\item[(b)] $\Phi(G)$ is nilpotent.
\item[(c)] $\Phi(G/N)=\Phi(G)/N$, for every normal subgroup $N$ of $G$ contained in $\Phi(G)$. In particular, $\Phi(G/\Phi(G))=1$.
\item[(d)] $\Phi(N)\leq \Phi(G)$, for every normal subgroup $N$ of $G$.
\end{itemize}
\end{remark}

\begin{proof}
For (a) and (b) see 5.2.3 and 5.2.5(a) in \cite{KS}. Let $N$ be normal in $G$. If $N\leq \Phi(G)$, then a subgroup $M$ of $G$ is maximal in $G$ if and only if $N\leq M$ and $M/N$ is maximal in $G/N$. This shows (c). For the proof of (d) assume by contradiction that there is a maximal subgroup $M$ of $G$ such that $\Phi(N)\not\leq M$. Then $G=\Phi(N)M$ and $N=\Phi(N)(M\cap N)$. Hence, by (a), $\Phi(N)\leq N=M\cap N\leq M$, a contradiction.
\end{proof}

The main aim of this section is the proof of the following lemma that the author learned from Stellmacher and was probably first proved by Meierfrankenfeld in the case $p=2$. It will be useful in connection with the pushing up arguments in Section~\ref{ThRestrProp} and Section~\ref{PUFs}.

\begin{lemma}\label{PhiN}
Let $G$ be a finite group with $O_p(G)=1$, and let $N$ be a normal subgroup of $G$ such that $G/N$ is a $p$-group. Then $\Phi(G)=\Phi(N)$.   
\end{lemma}

\begin{proof}
Assume the assertion is wrong and let $G$ be a minimal counterexample. Then $O_p(G)=1$ and we may choose a normal subgroup $N$ of $G$ such that $G/N$ is a $p$-group and $\Phi(G)\neq \Phi(N)$. We choose this normal subgroup $N$ of maximal order. By Remark \ref{Frat0}(d), we have

\bigskip

(1)\;\; $\Phi(N)\leq \Phi(G)$.

\bigskip

The assumption $O_p(G)=1$ implies $O_p(\Phi(G))=1$. Hence, Remark \ref{Frat0}(b) gives

\bigskip

(2)\;\;$\Phi(G)$ has order prime to $p$.

\bigskip

Consider now $\ov{G}:=G/\Phi(N)$. Let $X$ be the full preimage of $O_p(\ov{G})$ in $G$ and $P\in Syl_p(X)$. Then $X=\Phi(N)P$ and the Frattini Argument gives $G=XN_G(P)=\Phi(N)N_G(P)$. Now (1) and Remark \ref{Frat0}(a) imply $G=N_G(P)$. Hence, as $O_p(G)=1$, we have $P=1$ and so $O_p(\ov{G})=1$. Assume now $\Phi(N)\neq 1$. Then $|\ov{G}|<|G|$ and, as $G$ is a minimal counterexample, $\phi(\ov{G})=\Phi(\ov{N})$. Now by Remark \ref{Frat0}(c), $\Phi(\ov{G})=1$. Thus, by Remark \ref{Frat0}, $\Phi(G)=\Phi(N)$, a contradiction. This shows

\bigskip

(3)\;\;$\Phi(N)=1$.

\bigskip

Set now $G_0:=N\Phi(G)$. Observe that, by Remark \ref{Frat0}(a), $G_0$ is a proper subgroup of $G$. As $O_p(G)=1$ and $G_0$ is normal in $G$, we have $O_p(G_0)=1$. Hence, the minimality of $G$ yields $\Phi(G_0)=\Phi(N)$. If $\Phi(G)\not\leq N$, then the maximality of $|N|$ implies $\Phi(G)=\Phi(G_0)=\Phi(N)$, a contradiction. Hence, 

\bigskip

(4)\;\; $\Phi(G)\leq N$.

\bigskip

Set $V:=Z(\Phi(G))$. Observe that, by (3) and Remark \ref{Frat0}(b),

\bigskip

(5)\;\;$V\neq 1$.

\bigskip

We show next

\bigskip

(6)\;\;$V$ has  a complement in $N$.

\bigskip

By (3) and (5), there is a maximal subgroup $M_0$ of $N$ such that $V\not\leq M_0$. Then $N=VM_0$. Hence, there is a non-empty set $\m{E}$ of maximal subgroups of $N$ such that $N=VU$, for $U:=\bigcap\m{E}$. We choose such a set $\m{E}$ of maximal order. If $U\cap V\neq 1$, then (3) implies the existence of a maximal subgroup $M$ of $N$ such that $U\cap V\not\leq M$. Then, in particular, $U\not\leq M$ and so $M\not\in\m{E}$. Moreover, $N=(U\cap V)M$, so $U=(U\cap V)(U\cap M)$ and $N=VU=V(U\cap M)$. This is a contradiction to the maximality of $|\m{E}|$. Hence, $U\cap V=1$ and (6) holds.

\bigskip

We now derive the final contradiction. By (2),(6) and a Theorem of Gasch\"utz (see e.g. \cite[3.3.2]{KS}), there is a complement $K$ of $V$ in $G$, i.e. $K\cap V=1$ and $G=KV=K\Phi(G)$. Now Remark \ref{Frat0}(a) implies $G=K$ and so $V=1$, a contradiction to (5). 
\end{proof}

\section{Minimal parabolics}

Let $G$ be a finite group and $T\in Syl_p(G)$.

\begin{definition}
$G$ is called \textbf{minimal parabolic} (with respect to $p$) if $T$ is not normal in $G$ and there is a unique maximal subgroup of $G$ containing $T$.
\end{definition}

This concept is originally due to McBride. One of the main properties of minimal parabolic groups is the following.

\begin{lemma}\label{minpar}
Let $G$ be minimal parabolic with respect to $p$ and let $N$ be normal in $G$. Then the following hold.
\begin{itemize}
\item[(a)] $N\cap T\unlhd G$ or $O^p(G)\leq N$.
\item[(b)] If $O_p(G)=1$ then $O^p(G)\leq N$ or $N\leq \Phi(G)$.
\end{itemize}
\end{lemma}

\begin{proof}
For (a) see \cite[1.3(b)]{PPS}. For the proof of (b) assume by contradiction, $O_p(G)=1$, $N\not\leq \Phi(G)$ and $O^p(G)\not\leq N$. Then there is a maximal subgroup $M$ of $G$ such that $N\not\leq M$, and, by (a), $T\cap N=1$. Hence, $G=MN$ and $M$ contains a Sylow $p$-subgroup $S$ of $G$. As $NS\not\leq M$ and $S$ is contained in a unique maximal subgroup of $G$, this implies $G=NS$ and so $O^p(G)\leq N$.
\end{proof}

Given a group $G$ which is not $p$-closed, it is easy to obtain minimal parabolic subgroups of $G$ containing a Sylow $p$-subgroup of $G$. This is a consequence of the following remark which is elementary to check.

\begin{remark}\label{minparg}
Let $H$ be a subgroup of $G$ such that $N_G(T)\leq H< G$. Assume that $P$ is a subgroup of $G$ which is minimal with the properties $T\leq P$ and $P\not\leq H$. Then $P$ is minimal parabolic and $H\cap P$ is the unique maximal subgroup of $P$ containing $T$.
\end{remark}

\section{FF-modules}\label{FF}

Throughout this section let $G$ be a finite group, $p$ be a prime dividing $|G|$, $T\in Syl_p(G)$, and let $V$ be a finite dimensional $GF(p)G$-module.

\begin{definition}\label{OffDef}
\begin{itemize}
\item A subgroup $A$ of $G$ is said to be an \textbf{offender} on $V$, if 
\begin{itemize}
\item[(a)] $A/C_A(V)$ is a non-trivial elementary abelian $p$-group,
\item[(b)] $|V/C_V(A)|\leq |A/C_A(V)|$.
\end{itemize}
\item A subgroup $A$ of $G$ is called \textbf{best offender} if (a) holds and
\begin{itemize}
\item[(b')] $|A/C_A(V)||C_V(A)|\geq |A^*/C_{A^*}(V)||C_V(A^*)|$ for all subgroups $A^*$ of $A$. 
\end{itemize}
We write $\m{O}_G(V)$ for the set of all best offenders in $G$ on $V$.
\item The module $V$ is called an \textbf{FF-module} for $G$, if there is an offender in $G$ on $V$.
\item An offender $A$ on $V$ is called an \textbf{over-offender} on $V$ if $|V/C_V(A)|<|A/C_A(V)|$.
\item If $\m{O}_G(V)\neq\emptyset$, we set
$$m_G(V):=max\{|A/C_A(V)||C_V(A)|\,:\,A\in\m{O}_G(V)\},$$
and define $\mathcal{A}_G(V)$ to be the set of minimal (by inclusion) members of the set
$$\{A\in\m{O}_G(V)\,:\,|A/C_A(V)||C_V(A)|=m_G(V)\}.$$
\item For a set of subgroups $\m{D}$ of $G$ and $E\leq G$ set
$$\m{D}\cap E=\{A\in\m{D}:A\leq E\}.$$
\end{itemize}
\end{definition}

By \cite[2.5(a),(b)]{MS}, every best offender on $V$ is an offender on $V$, and $V$ is an FF-module if and only if there is a best offender on $V$. We will use this fact frequently and without reference.

\begin{definition}
Write $\mathcal{A}(G)$ for the set of all elementary abelian $p$-subgroups of $G$ of maximal order. Recall that the Thompson subgroup $J(G)$ is the subgroup of $G$ generated by $\m{A}(G)$. 
\end{definition}

\begin{lemma}\label{Baum}
Let $V$ be an elementary abelian normal $p$-subgroup of $G$. Let $A\in \mathcal{A}(G)$ and suppose that $A$ does not centralize $V$.
\begin{itemize}
\item[(a)] $A$ is a best offender on $V$.
\item[(b)] If $A$ is not an over-offender on $V$, then $VC_A(V)\in \mathcal{A}(G)$. In particular, we have then $\mathcal{A}(C_G(V))\subseteq \mathcal{A}(G)$ and $J(C_G(V))\leq J(G)$.
\end{itemize}
\end{lemma}

\begin{proof}
For the proof of (a) see \cite[2.8(e)]{BHS}. For the proof of (b) see \cite[B.2.4]{AS}.
\end{proof}

\begin{lemma}\label{1}
Let $V,W$ be normal elementary abelian $p$-subgroups of $G$ with $V\leq W$ and $[V,J(G)]\neq 1$. Let $A\in\m{A}(G)$ such that $[V,A]\neq 1$, and $AC_G(W)$ is a minimal with respect to inclusion element of the set
$$\{BC_G(W)\;:\;B\in \m{A}(G),\;[W,B]\neq 1\}.$$
Assume $A$ is not an over-offender on $V$. Then we have $|W/C_W(A)|=|A/C_A(W)|=|V/C_V(A)|$ and $W=VC_W(A)$.
\end{lemma}

\begin{proof}
It follows from Lemma~\ref{Baum} that
$|V/C_V(A)|=|A/C_A(V)|$, $B:=C_A(V)V\in\m{A}(G)$ and $|W/C_W(A)|\leq |A/C_A(W)|$.
Since $[V,A]\neq 1$ and $[V,B]=1$, $BC_G(W)$ is a proper subgroup of $AC_G(W)$. Hence, the minimality of $AC_G(W)$ yields $[W,B]=1$.
Thus, $C_A(V)=C_A(W)$. It follows that
\begin{eqnarray*}
|W/C_W(A)|&\leq& |A/C_A(W)|=|A/C_A(V)|\\
&=&|V/C_{V}(A)|=|VC_W(A)/C_W(A)|\leq |W/C_W(A)|.
\end{eqnarray*}
Now equality holds above, i.e $|A/C_A(W)|=|W/C_W(A)|=|V/C_V(A)|$ and $W=VC_W(A)$. 
\end{proof}

We continue by looking at natural $SL_2(q)$-modules and natural $S_n$-modules. These modules provide important examples of FF-modules. 

\begin{definition}\label{NatSL2qM}
Suppose $G\cong SL_2(q)$ for some power $q$ of $p$. Then $V$ is called a \textbf{natural $SL_2(q)$-module} for $G$ if $V$ is irreducible, $F:=End_G(V)\cong GF(q)$ and $V$ is a $2$-dimensional $FG$-module.
\end{definition}

The following lemma about natural $SL_2(q)$-modules is well known and elementary to check.

\begin{lemma}\label{natSL2q}
Assume that $G\cong SL_2(q)$ and $V$ is a natural $SL_2(q)$-module for $G$. Then the following hold:
\begin{itemize}
\item[(a)] $|C_V(T)|=q$ and $C_V(T)=[V,T]=C_V(a)$ for each $a\in T^\#$.
\item[(b)] We have $\m{O}_G(V)=\{A\leq G\;:\;A\;\mbox{is an offender on V}\;\}=Syl_p(G)$. Moreover, there are no over-offenders in $G$ on $V$.
\item[(c)] $C_G(C_V(T))=T$.
\item[(d)] Every element of $G$ of order coprime to $p$ acts fixed point freely on $V$.
\end{itemize}
\end{lemma}

\begin{lemma}\label{AutSL2qModule}
Let $H\unlhd G$ such that $H\cong SL_2(q)$ and $C_T(H)\leq H$. Let $V$ be a natural $SL_2(q)$-module for $H$ and assume $C_G(V)=1$. Then $C_T(C_V(T\cap H))\leq H$.
\end{lemma}

\begin{proof}
Set $Z:=C_V(T\cap H)$ and $T_0:=C_T(Z)$. Observe that, by the structure of $Aut(SL_2(q))$, there is an element $x\in H\backslash N_H(T\cap H)$ such that $T_0=(T\cap H)C_{T_0}(x)$. Then $[Z^x,C_{T_0}(x)]=1$ and so, $V=ZZ^x$ is centralized by $C_{T_0}(x)$. Hence, $C_{T_0}(x)=1$ and $T_0\leq H$. 
\end{proof}

\begin{lemma}\label{natSL2qOffFactor}
Let $G\cong SL_2(q)$ and $V/C_V(G)$ be a natural $SL_2(q)$-module for $G$. Let $A\leq G$ be an offender on $V$. Then
\begin{itemize}
\item[(a)] $|V/C_V(A)|=|A|=q$ and $C_V(A)=C_V(a)$ for every $a\in A^\#$,
\item[(b)] $[V,A,A]=1$.
\end{itemize}
\end{lemma}

\begin{proof}
As $G\cong SL_2(q)$, for every $a\in A^\#$ there exists $g\in G$ such that $G=\<A,a^g\>$. Hence, 
$$|V/C_V(G)|\leq |V/C_V(A)||V/C_V(a)|\leq |V/C_V(A)|^2\leq |A|^2=q^2=|V/C_V(G)|.$$
Thus, the inequalities are equalities and (a) holds. Together with Lemma~\ref{natSL2q}(a) this implies (b).
\end{proof}

\begin{lemma}\label{AT}
Let $V$ be an elementary abelian normal subgroup of $G$. Assume $G/V\cong SL_2(q)$ and $V/C_V(G)$ is a natural $SL_2(q)$-module. Then $V\in\m{A}(T)$. Moreover, the following hold:
\begin{itemize}
\item[(a)] For $R\in \m{A}(T)\backslash \{V\}$, we have $T=VR$, $R\cap V=Z(T)$ and $C_{V/C_V(J(G))}(T)=Z(T)/C_V(G)$.
\item[(b)] If $p=2$ and $J(T)\neq V$ then $|\m{A}(T)|=2$ and every elementary abelian subgroup of $T$ is contained in an element of $\m{A}(T)$.
\end{itemize}
\end{lemma}

\begin{proof}
Property (a) and $V\in \m{A}(T)$ is a consequence of Lemma~\ref{Baum}(a) and Lemma~\ref{natSL2q}(a),(b). Now (b) is a consequence of (a), Lemma~\ref{natSL2q}(a) and the fact that the product of two involutions is an involution if and only if these two involutions commute.
\end{proof}

\begin{lemma}\label{StructureS}
Let $p=2$ and let $V$ be an elementary abelian normal $2$-subgroup of $G$. Suppose $S$ is a $2$-group containing $T$ as a subgroup. Assume the following conditions hold: 
\begin{itemize}
\item[(i)] $V\leq J(G)$, and $J(G)/V\cong SL_2(q)$ for some power $q$ of $2$.
\item[(ii)] $V/C_V(J(G))$ is a natural $SL_2(q)$-module for $J(G)/V$.
\item[(iii)] $S\neq T=N_S(V)=N_S(U)$, for every $1\neq U\leq C_V(J(G))$ with $U\unlhd T$.
\item[(iv)] $C_T(J(G)/V)\leq V$.
\end{itemize}
Then the following hold:
\begin{itemize}
\item[(a)] $|N_S(T):T|=|N_S(J(T)):T|=2$ and $N_S(T)=N_S(J(T))$.
\item[(b)] If $J(N_S(J(T)))\not\leq T$ then $q=2$ and $|V|=4$.
\item[(c)] If $J(N_S(J(T)))\leq T$ then $J(T)=J(S)$, and $|S:T|=2$.
\item[(d)] If $T=J(T)$ and $Z(T)=Z(S)$ then $C_T(u)=Z(S)$ for every involution $u\in N_S(T)\backslash T$.
\end{itemize}
\end{lemma}

\begin{proof}
Since $S\neq T$ and $N_S(V)=T$, there is a conjugate of $V$ in $T$ distinct from $V$. Now by Lemma~\ref{AT}, $|\m{A}(T)|=2$ and $V\in\m{A}(T)$. As $S\neq T=N_S(V)$, this implies (a). Assume now there is  $R\in\m{A}(N_S(J(T)))$ such that $R\not\leq T$. Observe that by Lemma~\ref{AT}, $R\cap~J(T)\leq V\cap~V^x=Z(J(T))$ for $x\in R\backslash T$. Moreover, by (iii), $C_V(J(G))\cap C_V(J(G))^x=1$ and so $R\cap~J(T)\cap~C_V(J(G))=1$. Hence, if $R\cap T\leq J(T)$ then $|R\cap T|\leq q$. By (iv), $T/V$ embeds into $Aut(J(G)/V)$, so we have that $T/J(T)$ is cyclic and thus $|R\cap T/R\cap J(T)|\leq 2$. Moreover, by Lemma~\ref{AutSL2qModule}, $|C_{Z(J(T))/C_V(J(G))}(t)|<q$ for $t\in T\backslash J(T)$. Hence, if $R\cap T\not\leq J(T)$ then $|R\cap J(T)|<q$ and, again, $|R\cap~T|\leq q$.
Now by (a), $q^2\leq |V|\leq |R|\leq 2\cdot q$, so $q=2$ and $|V|=q^2=4$. 
This shows (b). Since $S$ is nilpotent, (c) is a consequence of (a).

\bigskip

For the proof of (d) assume now $T=J(T)$ and $Z(T)=Z(S)$. Let $u\in N_S(T)\backslash T$ be an involution and $y\in C_T(u)$. By Lemma~\ref{AT}(a), there exist $a,\tilde{a}\in V$ such that  $y=a\tilde{a}^u$. Then $a\tilde{a}^u=(a\tilde{a}^u)^u=a^u\tilde{a}$. Now $[V,V^u]\leq V\cap V^u=Z(T)$ implies $aZ(T)=\tilde{a}Z(T)$. Let $z\in Z(T)$ such that $\tilde{a}=az$. Then, as $Z(T)=Z(S)$, $y=aa^uz$ and $aa^u=yz=(yz)^u=a^ua$. Hence, Lemma~\ref{natSL2q}(a) implies $a\in Z(T)$ and so $y\in Z(T)=Z(S)$.
\end{proof}

\begin{definition}
Let $G \cong S_m$ for some $m\geq 3$. 
\begin{itemize}
\item We call a $GF(2)G$-module a \textbf{permutation module} for $G$, if it has a basis $\{v_1,v_2,\dots,v_m\}$ of length $m$ on which $G$ acts faithfully.
\item A $GF(2)G$-module is called a \textbf{natural $S_m$-module} for $G$ if it is isomorphic to a non-central irreducible section of the permutation module. 
\end{itemize}
\end{definition}

Observe that natural $S_m$-modules are by this definition uniquely determined up to isomorphism.

\begin{lemma}\label{natSnOff}
Assume $p=2$, $G= S_{2n+1}$ and $V$ is a natural $G$-module. Then the following conditions hold:
\begin{itemize}
\item[(a)] The elements in $\m{O}_G(V)$ are precisely the subgroups generated by commuting transpositions. 
\item[(b)] $N_G(J)/J\cong S_n$ for $J:=\<\m{O}_G(V)\>$.
\item[(c)] There are no over-offenders in $G$ on $V$.
\end{itemize}
\end{lemma}

\begin{proof}
Part (a) follows from \cite[2.15]{BHS}. Claims (b) and (c) are consequences of (a).
\end{proof}

\section{Pushing up}\label{PUGr}

Throughout this section, let $G$ be a finite group, $p$ a prime dividing $|G|$ and $T\in Syl_p(G)$. Let $q$ be a power of $p$.

\subsection{A result by Baumann and Niles}

The group $G$ is said to have the \textbf{pushing up property} (with respect to $p$) if the following holds:

\bigskip

(PU)\;\; No non-trivial characteristic subgroup of $T$ is normal in $G$.

\bigskip

Note that this property does not depend on the choice of $T$ since all Sylow $p$-subgroups of $G$ are conjugate in $G$. The problem of determining the non-central chief factors of $G$ in $O_p(G)$ under the additional hypothesis

\bigskip

(*)\;\; $\ov{G}/\Phi(\ov{G})\cong L_2(q)$ for $\ov{G}=G/O_p(G)$

\bigskip

was first solved by Baumann \cite{BAUM} and Niles \cite{N} independently. Later Stellmacher \cite{St} gave a shorter proof. We state here a slight modification of the result.

\begin{hyp}\label{ModHyp}
Let $Q:=O_p(G)$ and let $W\leq \Omega(Z(Q))$ be normal in $G$. Suppose the following conditions hold:
\begin{itemize}
\item[(1)] $G/C_G(W)\cong SL_2(q)$,
\item[(2)] $W/C_W(G)$ is a natural $SL_2(q)$-module for $G/C_G(W)$,
\item[(3)] $G$ has the pushing up property (PU), and (*) holds.
\end{itemize}
\end{hyp}

\begin{theorem}\label{PUGrCor}
Suppose Hypothesis~\ref{ModHyp} holds. Then one of the following holds for $V:=[Q,O^p(G)]$.
\begin{itemize}
\item[(I)] $V\leq \Omega(Z(O_p(G)))$ and $V/C_V(G)$ is a natural $SL_2(q)$-module for $G/C_G(W)$.
\item[(II)] $Z(V)\leq Z(Q)$, $p=3$, and $\Phi(V)=C_V(G)$ has order $q$. Moreover, $V/Z(V)$ and $Z(V)/\Phi(V)$ are natural $SL_2(q)$-modules for $G/C_G(W)$.
\end{itemize}
Furthermore, the following hold for every $\phi\in Aut(T)$ with $V\phi\not\leq Q$.
\begin{itemize}
\item[(a)] $Q=VC_Q(L)$ for some subgroup $L$ of $G$ with $O^p(G)\leq L$ and $G=LQ$.
\item[(b)] If (II) holds then $Q\phi^2=Q$.
\item[(c)] $\Phi(C_{Q}(O^p(G)))\phi
=\Phi(C_{Q}(O^p(G)))$.
\item[(d)] If (II) holds then $T$ does not act quadratically on $V/\Phi(V)$.
\item[(e)] If (II) holds then $W\phi\leq Q$ and $V\leq W\<(W\phi)^G\>$.
\item[(f)] $V\not\leq Q\phi$.
\end{itemize}
\end{theorem}

\begin{proof}
Theorem~1 in \cite{St} and \cite[3.2]{N} give us the existence of $\psi\in Aut(T)$ such that 
$$L/V_0O_{p^\prime}(L)\cong SL_2(q)\;\mbox{for}\;L=(V\psi)O^p(G)\;\mbox{and}\;V_0=V(L\cap Z(G)),$$
and one of the following hold:
\begin{itemize}
\item[(I')] $V\leq \Omega(Z(O_p(G)))$ and $V/C_V(G)$ is a natural $SL_2(q)$-module for $L/V_0O_{p^\prime}(L)$.
\item[(II')] $Z(V)\leq Z(O_p(G))$, $p=3$, and $\Phi(V)=C_V(G)$ has order $q$. Moreover, $V/Z(V)$ and $Z(V)/\Phi(V)$ are natural $SL_2(q)$-modules for $L/V_0O_{p^\prime}(L)$.
\end{itemize}
Observe that $LQ$ contains $O^p(G)$ and a Sylow $p$-subgroup of $G$, so $G=LQ$. Now (a) is a consequence of Theorem~2 in \cite{St}. Moreover, (b),(c) and (d) follow from 2.4, 3.2, 3.3 and 3.4(b),(c) in \cite{St}. Clearly (I') implies (I). Moreover, if (I') holds then $C_T(V)=Q$ and $C_T(V\phi)=Q\phi$, so (f) holds in this case.

\bigskip

We assume from now on that (II') holds and show next that $G/C_G(W)$ acts on $V/C_V(G)$. Note that $[V,V]\leq C_V(G)$ and so by (a), $[V,Q]\leq C_V(G)$. 
Therefore, as $[V,O_{p^\prime}(L)]=1=[W,O_{p^\prime}(L)]$ and $[W,O^p(G)]=[Z(V),O^p(G)]$, we have $C_L(W)=O_{p^\prime}(L)V_0$ and $C_G(W)=C_{QL}(W)=QC_L(W)\leq C_G(V/C_V(G))$.
So $G/C_G(W)$ acts on $V/C_V(G)$ and (II) holds.

\bigskip

Let now $\phi\in Aut(T)$ such that $V\phi\not\leq Q$. It follows from (2.4) and (3.2) in \cite{St} that $[W,O^p(G)]\phi\leq Q$. Hence, since $W=[W,O^p(G)]C_W(G)\leq [W,O^p(G)]Z(T)$, we have $W\phi\leq Q$. As  $\ov{Q}=Q/C_Q(O^p(G))\cong V/\Phi(V)$, it follows that $\ov{W}$ and $\ov{Q}/\ov{W}$ are natural $SL_2(q)$-modules. In particular, $[\ov{W},V\phi]\neq 1$ and so $\ov{W}\neq \ov{W\phi}$. Therefore, $\ov{Q}=\ov{W\<(W\phi)^G\>}$. This implies (e). For the proof of (f) assume $V\leq Q\phi$. Then by (a), $[V,V\phi]\leq [Q\phi,V\phi]\leq C_V(G)\phi\leq Z(T)\phi=Z(T)$. As $V\phi\not\leq Q$ we have $T=\<(V\phi)^{N_G(T)}\>Q$ and so $[V,T]\leq Z(T)$, a contradiction to (d). This proves (f).
\end{proof}

\subsection{The Baumann subgroup}

A useful subgroup while dealing with pushing up situations is the following:

\begin{definition}\label{BaumDef}
The subgroup
$$B(G)=\<C_P(\Omega(Z(J(P))))\;:\;P\in Syl_p(G)\>$$
is called the \textbf{Baumann subgroup} of $G$.
\end{definition}

Often it is not possible to show immediately that $G$ has the pushing up
property. In many of these situations it helps to look at a subgroup $X$ of $G$ such that $B(T)\in Syl_p(X)$ and to show that $X$ has the pushing up property. Here one uses that $B(T)$ is characteristic in $T$, so a characteristic subgroup of $B(T)$ is also a characteristic subgroup of $T$. Usually one can then determine the structure of $X$ and thus also of $B(T)$. This often leads to $T=B(T)\leq X$, in which case also the $p$-structure of $G$ is restricted. When using this method later, we will need the results stated below.

\begin{hyp}\label{BaumHyp}
Let $V\leq \Omega(Z(O_p(G)))$ be a normal subgroup of $G$ such that
\begin{itemize}
\item $G/C_G(V)\cong SL_2(q)$, 
\item $V/C_V(G)$ is a natural $SL_2(q)$-module for $G/C_G(V)$,
\item $C_G(V)/O_p(G)$ is a $p^\prime$-group and $[V,J(T)]\neq 1$.
\end{itemize}
\end{hyp}

\begin{lemma}\label{BaumR}
Assume Hypothesis \ref{BaumHyp} and suppose there is $d\in G$ such that $G=\<T,T^d\>$. Then $G=C_G(V)B(G)$, $\Omega(Z(J(T)))V$ is normal in $G$, and $B(T)\in Syl_p(B(G))$.
\end{lemma}

\begin{proof}
Set $Q:=O_p(G)$. Let $A\in \m{A}(T)$ such that $[V,A]\neq 1$. Then by Lemma~\ref{natSL2q}(b) and Lemma~\ref{Baum},
$|V/C_V(A)|=|A/C_A(V)|=q$
and $V(A\cap Q)\leq J(Q)\leq J(T)$. In particular, $T=J(T)Q$ and, since $C_G(V)/Q$ is a $p^\prime$-group, $\Omega(Z(J(T)))\leq C_T(V)=Q$. Now $W:=\Omega(Z(J(T)))V\leq \Omega(Z(J(Q)))$ and
$$|W/C_W(J(T))|=|VC_W(J(T))/C_W(J(T))|=|V/C_V(J(T))|=|V/C_V(T)|=q.$$
By assumption, we may choose $d\in G$ such that $G=\<T,T^d\>$. Then for $X_0:=\<J(T),J(T)^d\>$,  we have $G=X_0Q$. 
Moreover, for $B\in \m{A}(T^d)$, we have $V(B\cap Q)\in\m{A}(Q)$. So $W\leq \Omega(Z(J(Q)))\leq V(B\cap Q)$ and $W=V(B\cap W)$ is normalized by $B$. Hence, $W$ is normal in $J(T)^d$ and thus also in $G=X_0Q$. We get now
\begin{eqnarray*}
|VC_W(X_0)/C_W(X_0)|&\leq& |W/C_W(X_0)|\\
&\leq& |W/C_W(J(T))|^2= q^2=|V/C_V(X_0)|=|VC_W(X_0)/C_W(X_0)|.
\end{eqnarray*}
Hence, we have equality above and therefore $W=VC_W(X_0)$.

\bigskip
 
Set $Z_0:=C_{\Omega(Z(J(T)))}(X_0)$. As $W\leq J(T)$, we have $C_W(X_0)\leq \Omega(Z(J(T)))$. Thus, $C_W(X_0)=Z_0$ and $W=VZ_0$. Now Dedekind's Law implies $\Omega(Z(J(T)))=Z_0(\Omega(Z(J(T)))\cap~V)$. So, using $T=J(T)Q$, we get $B(T)=C_T(Z_0)$. Note that $Z_0$ is normal in $G=QX_0$. Hence, $X:=C_G(Z_0)\unlhd G$. This yields $B(T)=T\cap X\in Syl_p(X)$ and $B(G)\leq X$, so $B(T)\in Syl_p(B(G))$. Since $G=QX_0=QB(G)$, this implies the assertion.
\end{proof}

\begin{lemma}\label{BaumCor}
Assume Hypothesis \ref{BaumHyp}. Then $G=C_G(V)B(G)$ and $B(T)\in Syl_p(B(G))$.
\end{lemma}

\begin{proof}
Set $W :=\Omega(Z(J(T))V$ and $H:=O^{p^\prime}(G)$.  
As $G/C_G(V)\cong SL_2(q)$, we can choose $d\in G$ such that $G=C_G(V)H_0$, for $H_0:=\<T,T^d\>$. By Lemma~\ref{BaumR}, $W \trianglelefteq H_0$ and $G=C_G(V)B(H_0)$. In particular, $W=V\Omega(Z(J(T^d)))$. So, again by Lemma~\ref{BaumR} (now applied with $T^d$ in place of $T$), $W$ is normal in $\<\hat{T},T^d\>$, for every $\hat{T}\in Syl_p(G)$ with $\hat{T}C_G(V)=TC_G(V)$. Hence, the arbitrary choice of $d$ gives that $W$ is normalized by every Sylow $p$-subgroup of $G$ and therefore,
$$W\unlhd H.$$
Note that, as $B(H_0)\leq B(G)$, we have $B(G)=B(H_0)C_{B(G)}(V)$ and $G=C_G(V)B(G)$.
Also observe that $B(G)=B(H)=\langle B(T)^H\rangle=\<B(T)^{B(G)}\>$. In particular, $[W,B(G)]\leq V$. Hence, $[W,C_{B(G)}(V),C_{B(G)}(V)]=1$, and coprime action shows that $C_{B(G)}(V)\leq C_{B(G)}(W)Q$. Therefore, we get $B(G)=B(H_0)C_{B(G)}(V)\leq H_0C_{B(G)}(W)$. So  $B(H_0)C_{B(G)}(W)$ is a normal subgroup of $B(G)$ containing $B(T)$ and thus, $B(G)=B(H_0)C_{B(G)}(W)$. By Lemma~\ref{BaumR}, $B(T)\in Syl_p(B(H_0))$. Hence, $(T\cap B(G))C_{B(G)}(W)=B(T)C_{B(G)}(W)$ and $T\cap B(G)\leq B(T)C_T(W)\leq B(T)$. This shows $B(T)\in Syl_p(B(G))$. 
\end{proof}

\section{Amalgams}\label{AmalgamSection}

An amalgam $\m{A}$ is a tuple $(G_1,G_2,B,\phi_1,\phi_2)$ where $G_1,G_2$ and $B$ are groups and $\phi_i:B\rightarrow G_i$ is a monomorphism for $i=1,2$. We write $G_1*_{B} G_2$ for the free product of $G_1$ and $G_2$ with $B$ amalgamated. Note that we suppress here mention of the monomorphisms $\phi_1$ and $\phi_2$. 
We will usually identify $G_1$, $G_2$ and $B$ with their images in $G_1*_{B} G_2$. Then $G_1\cap G_2=B$, and the monomorphisms $\phi_1,\phi_2$ become inclusion maps. We will need the following lemma.

\begin{lemma}\label{NormalForms}
Let $G$ be a group such that $G=\<G_1,G_2\>$ for finite subgroups $G_1,G_2$ of $G$. Set $B=G_1\cap G_2$. For $i=1,2$, let $K_i$ be a set of right coset representatives of $B$ in $G_i$ and $\iota_i:B\rightarrow G_i$ the  inclusion map. By $G_1*_B G_2$ we mean the free amalgamated product with respect to $(G_1,G_2,B,\iota_1,\iota_2)$. Let $g\in G$. Then $g$ can be expressed in the form
\begin{itemize}
\item[(*)] $g=bg_1\dots g_n$ where $b\in B$, $n\in\N$, $g_1,\dots,g_n\in (K_1\cup K_2)\backslash B$ and, for every $1\leq k<n$ and $i\in\{1,2\}$, $g_{k+1}\in K_i$ if and only if $g_k\in K_{3-i}$. 
\end{itemize}
This expression is unique if and only if $G\cong G_1*_B G_2$.
\end{lemma}

\begin{proof}
As $X:=G_1*_B G_2$ is the universal completion of $(G_1,G_2,B,\iota_1,\iota_2)$, the group $G$ is isomorphic to a factor group of $X$ modulo a normal subgroup $N$ of $X$ with $N\cap G_1=N\cap G_2=1$. By (7.9) of Part~I in \cite{DGS}, every element $g\in X$ can be uniquely expressed in the form (*). This implies the assertion.
\end{proof}

A triple $(\beta_1,\beta_2,\beta)$ of group isomorphisms $\beta:B\rightarrow
\tilde{B}$ and $\beta_i:G_i\rightarrow \tilde{G_i}$, for $i=1,2$, is said to
be an \textbf{isomorphism} from $\m{A}$ to an amalgam
$\m{B}=(\tilde{G_1},\tilde{G_2},\tilde{G}_{12},\psi_1,\psi_2)$, if  the
obvious diagram commutes, i.e. if $\phi_i\beta_i=\beta\psi_i$, for $i=1,2$. An
\textbf{automorphism} of $\m{A}$ is an isomorphism from $\m{A}$ to $\m{A}$. The group of automorphisms of $\m{A}$ will be denoted by $Aut(\m{A})$. If $B=G_1\cap G_2$ and $\phi_1,\phi_2$ are inclusion maps, then ${\alpha_i}_{|B}=\alpha$, for every automorphism $(\alpha_1,\alpha_2,\alpha)$ of $\m{A}$.

\bigskip

A finite $p$-subgroup $S$ of a group $G$ is called a \textbf{Sylow $p$-subgroup} of $G$, if every finite $p$-subgroup of $G$ is conjugate to a subgroup of $S$. We write $S\in Syl_p(G)$.
We will use the following result, which is stated in this form in \cite{CP}. A similar result was proved first in \cite{R}.

\begin{theorem}[Robinson]\label{Rob}
Let $(G_1,G_2,B,\phi_1,\phi_2)$ be an amalgam, and let $G=G_1*_{B} G_2$ be the corresponding free amalgamated product. Suppose there is $S\in Syl_p(G_1)$ and $T\in Syl_p(G_2)\cap Syl_p(B)$ with $T\leq S$. Then $S\in Syl_p(G)$ and 
$$\mathcal{F}_{S}(G)=\langle \mathcal{F}_{S}(G_1),\mathcal{F}_{T}(G_2)\rangle.$$
\end{theorem}

\begin{proof}
See 3.1 in \cite{CP}.
\end{proof}

When we prove our classification result for $p=2$ we will apply Theorem~\ref{Rob} and the following theorem in order to identify a subsystem of a given saturated fusion system $\F$.

\begin{theorem}\label{MainAmalgamThm}
Let $(G_1,G_2,B,\phi_1,\phi_2)$ be an amalgam of finite groups $G_1,G_2$ and $G=G_1*_B G_2$ the corresponding free amalgamated product. Suppose the following hold for $S\in Syl_2(G_1)$, $Q:=O_2(G_2)$ and $M:=J(G_2)$. 
\begin{itemize}
\item[(i)] $N_S(Q)\in Syl_2(G_2)$ and $C_{G_2}(Q)\leq Q\leq M$.
\item[(ii)] $B=N_{G_1}(Q)=N_{G_2}(J(N_S(Q)))$.
\item[(iii)] $|G_1:B|=2$.
\item[(iv)] $M/Q\cong SL_2(q)$ where $q=2^e>2$, $\Phi(Q)=1$, and $Q/C_Q(M)$ is a natural $SL_2(q)$-module for $M/Q$.
\item[(v)] No non-trivial normal $p$-subgroup of $MN_S(Q)$ is normal in $G_1$.
\end{itemize}
Then there exists a free normal subgroup $N$ of $G$ such that $N\cap G_i=1$ for $i=1,2$, $H:=G/N$ is finite, $SN/N\in Syl_2(H)$ and $F^*(H)\cong L_3(q)$ or $Sp_4(q)$.
\end{theorem}

We will prove Theorem~\ref{MainAmalgamThm} at the end of this section. For that we need one more definition and some preliminary results.

\begin{definition}\label{WeakBNDef}
Let $G$ be a group with finite subgroups $G_1$ and $G_2$. Set $B:=G_1\cap G_2$.
\begin{itemize}
\item Let $q$ be a power of $p$. The pair $(G_1,G_2)$ is called a \textbf{weak BN-pair of $G$ involving $SL_2(q)$} if, for $i=1,2$, there are normal subgroups $G_i^*$ of $G_i$ such that the following properties hold:
\begin{itemize}
\item $G=\<G_1,G_2\>$,
\item no non-trivial normal subgroup of $G$ is contained in $B$,
\item $C_{G_i}(O_p(G_i))\leq O_p(G_i)\leq G_i^*$,
\item $G_i=G_i^*B$,
\item $G_i^*\cap B$ is the normalizer in $G_i^*$ of a Sylow $p$-subgroup of $G_i^*$ and $G_i^*/O_p(G_i)\cong SL_2(q)$.
\end{itemize}
\item If $(G_1,G_2)$ is a weak BN-pair of $G$ involving $SL_2(q)$, and $\iota_i:B\rightarrow G_i$ is the inclusion map for $i=1,2$, then we call $(G_1,G_2,B,\iota_1,\iota_2)$ the \textbf{amalgam corresponding to $(G_1,G_2)$}.
\end{itemize}
\end{definition}

The main tool is the following Theorem.

\begin{theorem}\label{ThmA}
Let $(G_1,G_2,B,\phi_1,\phi_2)$ be an amalgam and $G=G_1*_B G_2$ be the
corresponding free amalgamated product. Let $q$ be a power of $p$. Suppose $(G_1,G_2)$ is a weak BN-pair
of $G$ involving $SL_2(q)$, and $O_p(G_i)$ is elementary abelian for
$i=1,2$. Then there is a free normal subgroup $N$ of $G$ such that $G_i\cap N=1$ for $i=1,2$, $H:=G/N$ is
finite, and $F^*(H)\cong L_3(q)$ or $p=2$ and $F^*(H)\cong Sp_4(q)$. 
\end{theorem}

\begin{proof}
This is a consequence of Theorem~A in \cite{DGS}.
\end{proof}

In the situation of Theorem~\ref{ThmA}, it follows from the structure of $Aut(L_3(q))$ and $Aut(Sp_4(q))$ that $H$ embeds into $\Gamma L_3(q)$ respectively $\Gamma Sp_4(q)$. Moreover, for $i=1,2$ and $\ov{G}:=G/N$, $F^*(H)\cap \ov{G_i}$ is a parabolic subgroup of $F^*(H)$ in the Lie theoretic sense, and $\ov{B}\cap F^*(H)$ is the normalizer of a Sylow $p$-subgroup of $F^*(H)$.

\begin{lemma}\label{AutLocal}
Let $G$ be a finite group such that for $M=O^{p^\prime}(G)$, $T\in Syl_p(M)$ and $Q:=O_p(G)$ the following hold.
\begin{itemize}
\item[(i)] $M/Q\cong SL_2(q)$ and $G/Q$ is isomorphic to a subgroup of $GL_2(q)$, for some power $q$ of $p$.
\item[(ii)] $G/Q$ acts faithfully on $Q/Z(M)$, $Q$ is elementary abelian, $Q=[Q,M]$, $|Q/C_Q(T)|=q$, and $Q/Z(M)$ is a natural $SL_2(q)$-module for $M/Q$.
\item[(iii)] $Z(G)=1$.
\end{itemize}
Set $A:=Aut(G)$. Then $C_A(Q)=C_A(Q/Z(M))\leq C_A(G/Q)$ and $C_A(Q)$ is an elementary abelian $p$-group. Moreover, $C_A(T)\cong Z(T)$ for $T\in Syl_p(G)$.
\end{lemma}

\begin{proof}
Set $\ov{G}=G/Z(M)$ and $W:=C_A(\ov{Q})$. Throughout this proof we will identify $G$ with the group of inner automorphism of $G$. Note that this is possible by (iii). Observe that by (ii), $[G,W]\leq C_G(\ov{Q})\leq Q$. Hence, $[M,W,Q]\leq [Q,Q]=1$. As $[W,Q,M]=[Z(M),M]=1$ it follows from the Three-Subgroups Lemma that
$[Q,W]=[Q,M,W]=1$. So we have shown that $W=C_A(Q)\leq C_A(G/Q)$.
As $[W,G]\leq Q\leq C(W)$ it follows from the Three-Subgroups Lemma that $[W,W,G]=1$, i.e. $[W,W]=1$ and $W$ is abelian.
Since $[G,W,W] = 1$ and $[G,W]\leq Q$ is elementary abelian, we have $[g,w^p] = [g,w]^p=1$ for every $g\in G$ and $w\in W$. Hence $W$ is a group of exponent $p$ and thus an elementary abelian $p$-group.

\bigskip

Set now $C:=C_A(M)$. Then $C\leq W$ and so $C$ is elementary abelian. Since
$G$ acts coprimely on $C$, by Maschke's Theorem there is a $G$-invariant complement $C_0$ of
$Z(M)$ in $C$. Then $[C_0,G]\leq (C\cap G)\cap C_0=Z(M)\cap C_0=1$. Hence, $C_0=1$
and $C=Z(M)$.

\bigskip

Set now $W_0:=QC_W(T)$.  Note that $W_0$ is
$G$-invariant as $[W,G]\leq Q$, and that, by (ii),  $|W_0/C_W(T)|=|Q/C_Q(T)|=q$. If $C_W(T)\not\leq
Q$ then $|W_0/Z(M)|>q^2$ and, for $T\neq S\in Syl_p(G)$, $|Z(M)|<|C_W(T)\cap
C_W(S)|=|C_W(M)|$. This contradicts $C=Z(M)$. Hence,
$C_A(T)=C_W(T)=C_Q(T)=Z(T)$.
\end{proof}

\begin{lemma}\label{AutAmExt}
Let $q>2$ be a power of $2$ and $G\cong L_3(q)$ or $Sp_4(q)$. Let 
$(G_1,G_2)$ be a weak BN-pair of $G$ involving $SL_2(q)$. Let $B:=G_1\cap G_2$ and $\m{A}$ be the amalgam corresponding to $(G_1,G_2)$. Then for every $\hat{\alpha}\in Aut(\m{A})$ there exists $\beta\in Aut(G)$ such that $\hat{\alpha}=(\beta_{|G_1},\beta_{|G_2},\beta_{|B})$.
\end{lemma}

\begin{proof}
Let $i\in \{1,2\}$. Set 
\begin{eqnarray*}
T&=&O_p(B),\\
Q_i&=&O_p(G_i),\\ 
M_i&=&O^{p^\prime}(G_i),\\
A_i&=&\{\alpha_i:(\alpha_1,\alpha_2,\alpha)\in Aut(\m{A})\}\leq Aut(G_i),\\
C_i&=&C_{A_i}(Q_i),\\
A_0&=& N_{Aut(G)}(T)\cap N_{Aut(G)}(Q_1).
\end{eqnarray*}
By the structure of $G$, $(G_1,G_2)$ is a pair of parabolics of $G$ in the Lie theoretic sense, and $B$ is the normalizer of a Sylow $2$-subgroup. Moreover the following properties hold:
\begin{itemize}
\item[(1)] $M_i/Q_i\cong SL_2(q)$ and $G_i/Q_i$ is isomorphic to a subgroup of $GL_2(q)$.
\item[(2)] $G_i$ acts faithfully on $Q_i/Z(M_i)$, $\Phi(Q_i)=1$, $[Q_i,M_i]=Q_i$, $|Q_i/C_{Q_i}(T)|=q$, and $Q_i/Z(M_i)$ is a natural $SL_2(q)$-module for $M_i/Q_i$.
\item[(3)] $Z(G_i)=1=Z(B)$. 
\item[(4)] $T\in Syl_p(G)$ and $B=N_G(T)$.
\end{itemize}
The automorphism group of $G$ is generated by a graph automorphism and the elements of $\Gamma L_3(q)$ respectively $\Gamma Sp_4(q)$. Therefore, we get the following property: 
\begin{itemize}
\item[(5)] $A_0\leq N(Q_2)$ and $A_0/Q_i \cong N_{\Gamma GL_2(q)}(\tilde{T})$ for $\tilde{T}\in Syl_2(GL_2(q))$, where we identify $Q_i$ with the inner automorphisms of $G$ induced by $Q_i$.
\end{itemize}
Properties (1) and (2) give in particular that $G_i$ fulfills the hypothesis of Lemma~\ref{AutLocal}. Hence, we have
\begin{itemize}
\item[(6)] $C_i=C_{A_i}(Q_i/Z(M_i))\leq C(G_i/Q_i)$ and $C_i$ is a $2$-group.
\item[(7)] $C_{A_i}(T)\cong Z(T)$.
\end{itemize}

Observe that the map
$$ \phi: A_0\rightarrow Aut(\m{A})\mbox{ defined by }\alpha\mapsto (\alpha_{|G_1},\alpha_{|G_2},\alpha_{|B})$$
is well-defined and a monomorphism of groups. Recall that ${\alpha_1}_{|B}=\alpha_{|B}={\alpha_2}_{|B}$ for $(\alpha_1,\alpha_2,\alpha)\in Aut(\m{A})$. Moreover, by (3) and (7), $C_{A_1}(B)=1$ and $C_{A_2}(B)=1$. Hence, for $i=1,2$, the maps
$$ \psi_i : Aut(\m{A})\rightarrow A_i \mbox{ defined by } (\alpha_1,\alpha_2,\alpha)\mapsto \alpha_i $$
are isomorphisms of groups. In particular, it is therefore sufficient to show $|A_1|=|A_0|$.
Observe that, by (5),
\begin{itemize}
\item[(8)] $A_i/C_{A_i}(M_i/Q_i)\cong N_{Aut(M_i/Q_i)}(T/Q_i)$ for $i=1,2$.
\end{itemize}
Furthermore, since every element in $C_{A_i}(M_i/Q_i)$ acts on $Q_i/Z(M)$ as a scalar from $End_{M_i}(Q_i/Z(M))\cong GF(q)$, it follows from (2),(5) and (6) that
\begin{itemize}
\item[(9)] $C_{A_i}(M_i/Q_i)/C_i\cong C_{q-1}$ for $i=1,2$. 
\end{itemize}
Hence, by (5), it is sufficient to show that $|C_1|\leq |Q_1|$. 
In order to prove that set $C:=C_1\psi_1^{-1}\psi_2$.
Note that ${\alpha_1}_{|B}=(\alpha_1\psi_1^{-1}\psi_2)_{|B}$ for every $\alpha_1\in A_1$. Thus, $[Q_1,C]=1$ and $[T,C]=[Q_1Q_2,C]\leq Q_2$. By (6), $C\cong C_1$ is  a $2$-group. Hence, by (8), $C\leq TC_{A_2}(M_2/Q_2)$, and by (9), $C_0:=C\cap C(M_2/Q_2)\leq C_2$. Thus, $|C/C_0|\leq q$ and, by (7), $C_0\leq C_{A_2}(Q_1Q_2)=C_{A_2}(T)\cong Z(T)$. So $|C_1|=|C|\leq q\cdot |Z(T)|=|Q_1|$. As argued above this proves the assertion.
\end{proof}

\begin{lemma}\label{KerEqual}
Let $G$ be a group, let $q>2$ be a power of $2$, and let $(G_1,G_2)$ be a weak BN-pair of $G$ involving $SL_2(q)$. Let $H$ be a finite group such that $F^*(H)\cong L_3(q)$ or $Sp_4(q)$ for a power $q>2$ of $2$. Let $\phi$ and $\psi$ be epimorphisms from $G$ to $H$ such that $G_i\cap ker\phi=G_i\cap ker\psi= 1$ for $i=1,2$. Then $ker\phi=ker\psi$. 
\end{lemma}

\begin{proof}
Since $G_i\cap ker\phi=1$ for $i=1,2$, it is easy to check from Definition~\ref{WeakBNDef} that $(G_1\phi,G_2\phi)$ is a weak BN-pair of $H$ involving $SL_2(q)$. Since $H$ embeds into $Aut(L_3(q))$ respectively $Aut(Sp_4(q))$, it follows from the structure of these groups that $H$ embeds into $\Gamma L_3(q)$ respectively $\Gamma Sp_4(q)$. Furthermore, 
$$(F^*(H)\cap G_1\phi,F^*(H)\cap G_2\phi)$$ is a pair of parabolic subgroups of $F^*(H)$ in the Lie theoretic sense and a weak BN-pair of $F^*(H)$
involving $SL_2(q)$. Let $G_i^\circ$ be the preimage of $F^*(H)\cap G_i\phi$
in $G_i$ for $i=1,2$. Then $(G_1^\circ,G_2^\circ)$ is a weak BN-pair of
$G^\circ=\<G_1^\circ,G_2^\circ\>$ involving $SL_2(q)$, and
$G^\circ\phi=F^*(H)$. Moreover, for $i=1,2$, $G_i^\circ$ is normal in $G_i$
and $G_i=G_i^\circ B$. Thus, $G^\circ$ is normal in $G=\<G^\circ, B\>$. Set
now $N:=ker \phi$ and $\ov{G}=G/(G^\circ\cap N)$.  
Then $\ov{G^\circ}\cong F^*(H)$, $\ov B \cong B\phi$, and with  Dedekind's Law $B(G^\circ\cap N)\cap G^\circ=(B\cap G^\circ)(G^\circ\cap N)$, so $\ov B \cap \ov {G^\circ} = \ov{B\cap G^\circ}$. Moreover, $B\cap G^\circ=B\cap G_i\cap G^\circ=B\cap G_i^\circ$ and so $\ov{B}\cap\ov{G^\circ}=\ov{B\cap G_i^\circ}\cong B\cap G_i^\circ\cong (B\cap G_i^\circ)\phi = B\phi \cap G_i^\circ \phi = B\phi \cap G_i\phi\cap F^*(H)=B\phi\cap F^*(H)$, for $i=1,2$. 
Hence, 
$$|\ov{G}|=|\ov{G^\circ B}|=|\ov{G^\circ}||\ov B/\ov{B} \cap \ov{G^\circ}| =
|F^*(H)||B\phi/B\phi \cap F^*(H)| = |F^*(H)(B\phi)|=|H|,$$
and so $N\leq G_0$. The same holds with $\psi$ instead of $\phi$. Thus, we may assume without loss of generality that $H=F^*(H)\cong L_3(q)$ or $Sp_4(q)$. 

\bigskip

Then the weak BN-pairs of $H$ are precisely the pairs of parabolic subgroups of $H$ (in the Lie theoretic sense) intersecting in the normalizer of a Sylow $2$-subgroup. Now using \cite[Section~12.3]{Car} one sees that $Aut(H)$ acts transitively on the weak BN-pairs of $H$. Therefore, we may assume that $G_i\phi=G_i\psi$ for $i=1,2$. Then $((\phi_{|G_1})^{-1}\psi,(\phi_{|G_2})^{-1}\psi,(\phi_{|B})^{-1}\psi)$ is an automorphism of the amalgam corresponding to $(G_1\phi,G_2\phi)$. Hence, by Lemma~\ref{AutAmExt}, there is an automorphism $\alpha$ of $H$ such that $(\phi_{|G_i})^{-1}\psi=\alpha_{|G_i}$ for $i=1,2$. This implies $\psi=\phi\alpha$ and $ker\psi=ker\phi\alpha=ker\phi$.
\end{proof}

\textit{The proof of Theorem~\ref{MainAmalgamThm}.}
Let $G_1,G_2,B,S,Q,q$ and $M$ be as in the hypothesis of Theorem~\ref{MainAmalgamThm}. Set $T:=N_S(Q)$. Let $t\in S\backslash T$ and $X:=\<G_2,G_2^t\>$. As $G_2^{t^2}=G_2$, $X$ is normal in $G=\<t,G_2\>$. 

\bigskip

Let $K_1$ be a set of right coset representatives of $B$ in $G_2$. Then $K_1^t$ is a set of right coset representatives of $B=B^t$ in $G_2^t$. So, as $Bt^{-1}=Bt$, the set
$$K_2:=\{tkt\;:\:k\in K_1\}$$
is also a set of right coset representatives of $B$ in $G_2^t$. Let $g\in G$. By Lemma~\ref{NormalForms}, there exists $b\in B$, $n\in\N$ and $g_1,g_2,\dots,g_n\in (K_1\cup K_2)\backslash B$ such that 
$$g=bg_1\dots g_n$$ 
and, for all $1\leq k<n$ and $i\in\{1,2\}$, $g_{k+1}\in K_i$ if and only if $g_k\in K_{3-i}$.
Since $G=G_1*_B G_2$ and $\{t,1\}$ is a set of right coset representatives of $B$ in $G_1$, it follows from Lemma~\ref{NormalForms} and the definition of $K_2$ that this expression is unique. Hence, again by Lemma~\ref{NormalForms}, $X=G_2*_B G_2^t$.

\bigskip

Assume there is $1\neq U\leq B$ such that $U$ is normal in $X$.
If $U$ is a $p$-group then, as $U$ is normal in $G_2$ and $G_2^t$, it follows from Lemma~\ref{AT}(a) that $U\leq Q\cap Q^t=Z(J(T))$. Hence, as $Q/C_Q(M)$ and $Q^t/C_{Q^t}(M^t)$ are irreducible modules for $M$ respectively $M^t$, we have $U\leq U_0:=C_Q(M)\cap C_Q(M)^t$. Since $U_0$ is normal in $MT$ and $G_1=B\<t\>$, it follows from our assumptions that $U_0=1$. Hence, $U=1$, a contradiction. So $U$ is not a $p$-group and, as $U$ was arbitrary, also $O_p(U)=1$. Since  $J(T)$ is normal in $B$, it follows $[U,J(T)]\leq U\cap J(T)\leq O_p(U)=1$. In particular, $U\leq C_{G_2}(Q)\leq Q$, contradicting $U$ not being a $p$-group. Hence, no non-trivial normal $p$-subgroup of $X$ is contained in $B$. Thus, it is now easy to check that $(G_2,G_2^t)$ is a weak BN-pair of $X$ involving $SL_2(q)$. Hence, by Theorem~\ref{ThmA}, there is a free normal subgroup $N$ of $X$ such that $N\cap G_2=1=N\cap G_2^t$, $\ov{X}:=X/N$ is finite and $F^*(\ov{X})\cong L_3(q)$ or $Sp_4(q)$. One can check now from Definition~\ref{WeakBNDef} that $(\ov{G_2},\ov{G_2^t})$ is also a weak BN-pair of $\ov{X}$. As $\ov{X}$ embeds into $Aut(L_3(q))$ respectively $Aut(Sp_4(q))$, the structure of these groups yields $\ov{T}\in Syl_2(\ov{X})$.

\bigskip

Define epimorphisms $\phi$ and $\psi$ from $X$ to $\ov{X}$ via $x\phi=\ov{x}$ and $x\psi=\ov{x^t}$ for all $x\in X$. Then it follows from Lemma~\ref{KerEqual} that $N=ker\phi=ker \psi= N^{t^-1}$. Hence, $N$ is normal in $G=X\<t\>$. Observe now that $H:=G/N$ is finite, and $\ov{X}$ has index $2$ in $H$. So $SN/N\in Syl_2(H)$ and $F^*(H)=F^*(\ov{X})\cong L_3(q)$ or $Sp_4(q)$.

\section{Classification for $p=2$}\label{ClassifyGSection}

Throughout this section let $\F$ be a fusion system on a finite $2$-group $S$.

\begin{hyp}\label{ClassHyp}
Assume every parabolic subsystem of $\F$ is constrained. Let $Q\in\F$ such that $Q$ is centric and fully normalized. Set $T:=N_S(Q)$ and $M:=J(G(Q))$, and assume the following hold:
\begin{itemize}
\item[(i)] $Q\leq M$, $M/Q\cong SL_2(q)$ for some power $q$ of $p$, and $C_T(M/Q)\leq Q$.
\item[(ii)] $Q$ is elementary abelian and $Q/C_Q(M)$ is a natural $SL_2(q)$-module for $M/Q$.
\item[(iii)] $T<S$ and $N_S(U)=T$ for every subgroup $1\neq U\leq Q$ with $U\unlhd MT$.
\item[(iv)] If $t\in T\backslash J(T)$ is an involution and $\<t\>$ is fully centralized, then $C_\F(\<t\>)$ is constrained.
\end{itemize}
\end{hyp}

Recall here from Notation~\ref{ModNot} that, for every fully normalized subgroup $P\in \F$, $G(P)$ denotes a model for $N_\F(P)$, provided $N_\F(P)$ is constrained. The aim of this section is to prove the following theorem.

\begin{theorem}\label{ClassifyG}
Assume Hypothesis~\ref{ClassHyp}. Then there is a finite group $G$ containing $S$ as a Sylow $2$-subgroup such that $\F\cong\F_S(G)$ and one of the following holds:
\begin{itemize}
\item[(a)] $S$ is dihedral of order at least $16$, $Q\cong C_2\times C_2$ and $G\cong L_2(r)$ or $PGL_2(r)$, for some odd prime power $r$.
\item[(b)] $S$ is semidihedral, $Q\cong C_2\times C_2$ and $G$ is an extension of $L_2(r^2)$ by an automorphism of order $2$, for some odd prime power $r$.
\item[(c)] $S$ is semidihedral of order $16$, $Q\cong C_2\times C_2$ and $G\cong L_3(3)$.
\item[(d)] $|S|=32$, $Q$ has order $8$, and $G\cong Aut(A_6)$ or $Aut(L_3(3))$.
\item[(e)] $|S|=2^7$ and $G\cong J_3$.
\item[(f)] $F^*(G)\cong L_3(q)$ or $Sp_4(q)$, $|O^2(G):F^*(G)|$ is odd and $|G:O^2(G)|=2$. Moreover, if $F^*(G)\cong Sp_4(q)$ then $q=2^e$ where $e$ is odd.
\end{itemize}
\end{theorem}

Recall Notation~\ref{SP} and Notation~\ref{phi*} which we will use frequently in this section. Moreover, to ease notation we set
$$A(P):=Aut_\F(P),\mbox{ for every }P\in\F.$$

\subsection{Preliminary results}

We start with some group theoretical results. For Lemma~\ref{HS1Help}--Lemma~\ref{NoEssL34} let $G$ be a finite group.

\begin{lemma}\label{HS1Help}
Let $S\in Syl_2(G)$, $T:=J(S)$ and $t\in S\backslash T$. Assume the following hold.
\begin{itemize}
\item[(i)] $|S:T|=2$.
\item[(ii)] $A\leq Z(S)$ for every elementary abelian subgroup $A$ of $T$ with $C_S(A)\not\leq T$.
\item[(iii)] $Z(S)\leq T$, $Z(T)$ is elementary abelian, and $|Z(T)/Z(S)|>2$.
\item[(iv)] $Z(T)\<t\>\not\leq T^g$ for any $g\in G$.
\end{itemize}
Then $t\not\in T^g$ for any $g\in G$.
\end{lemma}

\begin{proof}
Set $Z:=Z(T)$. Assume there exists $g\in G$ such that $t\in T^g$ and choose this element $g$ such that $|Z(S)\cap T^g|$ is maximal. We show first

\bigskip

(1)\;\; $Z^x\leq T$ for all $x\in G$ with $Z^x\leq S$.

\bigskip

For the proof assume there is $x\in G$ such that $Z^x\leq S$ and $Z^x\not\leq T$. Then by (ii) and (iii), $Z^x\cap T\leq Z(S)$ and so $|Z/Z(S)|\leq |Z/(Z^x\cap T)|=|Z^x/(Z^x\cap T)|=2$, a contradiction to (iii). This shows (1).
Set
$$ Z^*:=(Z(S)\cap T^g)\<t\>\;\mbox{and } N:=N_G(Z^*).$$
Let $h\in G$ such that $Z\cap N\leq N\cap S^h\in Syl_2(N)$. Note that $t\in Z^*\leq O_p(N)\leq S^h$. We show next

\bigskip

(2)\;\; $t\in T^h$.

\bigskip

For the proof assume $t\not\in T^h$. As $Z^*\leq T^g$, we have $[Z^*,Z^g]=1$. Therefore, since $C_G(Z^*)\cap S^h\in Syl_2(C_G(Z^*))$, there exists $c\in C_G(Z^*)$ such that $Z^{gc}\leq S^h$. Now by (1), $Z^{gc}\leq T^h$. Note that $[Z^{gc},t]=1$, so by (ii), $Z^{gc}\leq Z(S)^h$, a contradiction to (iii). This shows (2).

\bigskip

In particular, by (iv), $Z\not\leq T^h$ and so, by (1), $Z\not\leq S^h$. Thus, the choice of $h$ gives $Z\not\leq N$. By (i), $S=T\<t\>$ and $t^2\in T$. So $[Z,T]=1$ implies $[Z,S]=[Z,t]\leq C_Z(t)\leq Z(S)$. Hence, if $Z(S)\leq T^g$ then $Z(S)\leq Z^*$ and so $[Z,Z^*]\leq [Z,S]\leq Z(S)\leq Z^*$, contradicting $Z\not\leq N$. This proves

\bigskip

(3)\;\; $Z(S)\not\leq T^g$.

\bigskip

Because of the maximality of $|Z(S)\cap T^g|$, properties (2) and (3) give now $Z(S)\not\leq T^h$. Note that $Z(S)\leq Z\cap N\leq S^h$ and so $S^h=T^hZ(S)$. Thus, $Z\cap N=Z(S)(Z\cap N\cap T^h)$. Moreover as $Z(S)\not\leq T^h$, (ii) and $t\in S^h$ imply $Z\cap N\cap T^h\leq Z(S)^h\leq C_G(t)$. Therefore, $Z\cap N\leq C_Z(t)=Z(S)$, and so $Z\cap N=Z(S)$. Hence, for $z\in Z\backslash Z(S)$, we have $z\not\in N$ and so $[z,t]\not\in Z(S)\cap T^g$. As $[Z,t]\leq Z(S)$, this gives $[Z,t]\cap T^g=1$. Hence, $|Z/Z(S)|=|Z/C_Z(t)|=|[Z,t]|=|[Z,t]T^g/T^g|\leq |Z(S)/Z(S)\cap T^g|$. Now the maximality of $|Z(S)\cap T^g|$ yields 
$|Z/Z(S)|\leq |Z(S)/Z(S)\cap T^h|= |S^h/T^h|=2$, a contradiction to (iii). This proves the assertion. 
\end{proof}

\begin{corollary}\label{HS1}
Let $S\in Syl_2(G)$. Let $T$ be a subgroup of $S$ such that 
\begin{itemize}
\item[(i)] $T$ is elementary abelian and $|S:T|=2$.
\item[(ii)] $|C_T(S)|^2=|T|$.
\end{itemize}
Then $T$ is strongly closed in $\F_S(G)$ or $|T|=4$.
\end{corollary}

\begin{proof}
This is a direct consequence of Lemma~\ref{HS1Help}.
\end{proof}

\begin{lemma}\label{HS2}
Let $S\in Syl_2(G)$, $T\leq S$ and $K\leq N_G(T)$ such that for $Z:=Z(T)$ the following hold.
\begin{itemize}
\item[(i)] $|S:T|=2$ and $|Z/C_Z(S)|=2$.
\item[(ii)] $|K|$ is odd and $K$ acts irreducibly on $Z/C_Z(K)$.
\item[(iii)] $C_Z(K)\cap C_Z(S)=1$.
\end{itemize}
Then $|Z|=4$.
\end{lemma}

\begin{proof}
Without loss of generality assume $G=N_G(Z)$. Set $\ov{G}=G/C_G(Z)$. Then $|\ov{S}|=2$ and by Cayley's Theorem there is a normal subgroup $U$ of $\ov{G}$ such that $|U|$ has odd order and $\ov{G}=\ov{S}U$. Set $R:=[\ov{S},U]$. If $R=1$ then $[\ov{K},\ov{S}]=1$ and $C_Z(K)$ is $S$-invariant. Hence, by (iii), $C_Z(K)=1$ and by (ii), $[Z,S]=1$, a contradiction to (i). Thus $1\neq R\leq R_0=\<\ov{S}^{U}\>$. If $O_2(R_0)\neq 1$ then $\ov{S}=O_2(R_0)$ is normal in $\ov{G}$ and $R=1$, a contradiction. Thus, $O_2(R_0)=1$. With a Theorem of Glauberman \cite[9.3.7]{KS} it follows from (i) that $R_0\cong S_3$ and $|Z/C_Z(R_0)|=4$. In particular, for $D:=O^p(R)$, $|D|=3$, $|[Z,D]|=|[Z,R]|=4$ and $C_Z(D)=C_Z(R)=C_Z(R_0)\leq C_Z(S)$. Since $R$ is normal in $G$, $\ov{K}$ acts on $R$. So, as $K$ has odd order, $[D,\ov{K}]=1$. Since $C_{[Z,D]}(S)\neq 1$, (iii) yields $[Z,D]\not\leq C_Z(K)$. As $D$ acts irreducibly on $[Z,D]$, we have then $[Z,D]\cap C_Z(K)=1$ and so $[C_Z(K),D]=1$. Hence, $C_Z(K)\leq C_Z(D)\leq C_Z(S)$ and by (iii), $C_Z(K)=1$. So, by (ii), $Z=[Z,D]$ has order $4$. 
\end{proof}

We will refer to the following lemma which is elementary to check.

\begin{lemma}\label{Aut2Group}
Assume one of the following holds:
\begin{itemize}
\item[(a)] $G\cong D_8$, $G\cong C_4\times C_2$, $G\cong D_8\times C_2$ or $G\cong C_4*D_8$.
\item[(b)] There are subgroups $V,K$ of $G$ such that $G=K\ltimes V$, $K\neq 1$ is cyclic of order at most $4$, $V$ is elementary abelian of order at most $2^3$ and $[V,K]\neq 1$.
\end{itemize}
Then $Aut(G)$ is a $2$-group.
\end{lemma}

\begin{lemma}\label{NoEssL34}
Suppose $G\cong L_3(4)$ or $Sp_4(4)$. Let $S\in Syl_2(Aut(G))$ and identify $G$ with its group of inner automorphisms. Let $t\in S\backslash G$ be a field automorphism of $G$ and $C_S(t)\leq P<S$. Then $Aut(P)$ is a $2$-group.
\end{lemma}

\begin{proof}
Let $Q\in\m{A}(S)$. Set $T:=N_S(Q)$ and $Z:=Z(J(S))$. It follows from the structure of $Aut(G)$ that $J(S)\in Syl_p(G)$, $T=J(S)\<t\>$, $|S/J(S)|=4$ and $S=J(S)C_S(t)$. Furthermore, if $G\cong L_3(4)$, we may choose an involution $s\in C_S(t)\backslash T$ such that $[Z,s]=1$. If $G\cong Sp_4(4)$, then it follows from \cite[Section~12.3]{Car}, that $S/J(S)$ is cyclic and we can pick $s\in C_S(t)\backslash T$ such that $s^2=t$. In both cases, we set
$$M:=O^{p^\prime}(N_G(Q)),\;W:=C_Q(t)\mbox{ and }Z_0:=Z\cap P.$$
By the structure of $G$, $M/Q\cong SL_2(q)$, $Q/C_Q(M)$ is a natural $SL_2(q)$-module, and in the case $G\cong Sp_4(q)$, $Q$ is the $3$-dimensional orthogonal module. Together with Lemma~\ref{AT}, this gives the following property:

\bigskip

(1)\;\;For every $x\in S\backslash T$, we have $\m{A}(S)=\{Q,Q^x\}$, $Z=Q\cap Q^x=[Q,Q^x]$, and every elementary abelian subgroup of $J(S)$ is contained in $Q$ or $Q^x$.

\bigskip

The Structure of $Aut(G)$ gives also $C_M(t)/W\cong S_3$ and $[W,C_M(t)]$ is a natural $S_3$-module for $C_M(t)/W$. Furthermore,

\bigskip

(2)\;\;$|Z(S)|=2$ and $Z(S)=[W,W^s]$.

\bigskip

In particular, if $Q=(Q\cap P)Z$ then $Q^s=(Q^s\cap P)Z$ and $Z=[Q,Q^s]\leq P$. Hence, $Q\leq P$ and so $S=(QQ^s)C_S(t)\leq P$, a contradiction. As $|Q:(WZ)|=2$, this shows

\bigskip

(3)\;\;$P\cap Q=WZ_0$ and $P\cap Q^s=W^sZ_0$.

\bigskip

Assume now the assertion is wrong. Pick a non-trivial element $\alpha\in Aut(P)$ of odd order. We show next

\bigskip

(4)\;\; $C_S(t)<P$.

\bigskip

Assume $P=C_S(t)$. Then $\Omega(Z(P))\leq C_S(W)\leq T$ and, by (1), $\Omega(Z(P))\cap J(S)\leq C_Z(P)=Z(S)$. Hence, $\Omega(Z(P))=Z(S)\<t\>$. In particular, $P/\Omega(Z(P))\cong D_8$ if $G\cong L_3(4)$, and $P/\Omega(Z(P))\cong D_8\times C_2$ if $G\cong Sp_4(4)$. Hence, Lemma~\ref{Aut2Group} gives $[P,\alpha]\leq \Omega(Z(P))$. Moreover, by (2), $|Z(S)|=2$ and $Z(S)=[W,W^s]\leq P^\prime\leq J(S)$, so $Z(S)=\Omega(Z(P))\cap P^\prime$. Coprime action shows now $[P,\alpha]=1$, a contradiction. Thus, (4) holds. We show next

\bigskip

(5)\;\;$C_Z(t)<Z_0$.

\bigskip

If $Z_0=C_Z(t)$ then (3), (4) and $S=J(S)C_S(t)$ imply $J(S)=(P\cap J(S))Q$. Recall that $Q/C_Q(M)$ is a natural $SL_2(4)$-module for $M/Q$, $J(S)\in Syl_p(M)$ and $Z/C_Q(M)=C_{Q/C_Q(M)}(J(S))$. Hence, $[W,P\cap J(S)]\leq P\cap J(S)^\prime=P\cap Z=Z_0$. Furthermore, as $W\not\leq Z$ and $J(S)=(P\cap J(S))Q$, also $[W,P\cap J(S)]C_Q(M)=[W,J(S)]C_Q(M)=Z$ and so $[W,P\cap J(S)]\not\leq C_Z(t)$. Thus, $Z_0\not\leq C_Z(t)$, contradicting our assumption. Therefore, (5) holds. Since $\Omega(Z(P))\leq C_S(W)\leq T$, (5) gives in particular that $\Omega(Z(P))\leq J(S)$. Hence, (1) implies $\Omega(Z(P))\leq C_Z(P)=Z(S)$. So, by (2),

\bigskip

(6)\;\;$\Omega(Z(P))=Z(S)$.

\bigskip

We show next

\bigskip

(7)\;\;$Z_0\alpha\neq Z_0$.

\bigskip

Assume $Z_0\alpha=Z_0$ and set $\ov{P}=P/Z_0$. Then $\ov{J(S)\cap P}$ is elementary abelian of order at most $2^3$. For $G\cong Sp_4(4)$ we get $[\ov{P},\alpha]=1$ as an immediate consequence of Lemma~\ref{Aut2Group}. For $G\cong L_3(4)$ note that, by (5), $C_P(Z_0)=(J(S)\cap P)\<s\>$ has index $2$ in $P$ and, by Lemma~\ref{Aut2Group}, $[\ov{C_P(Z_0)},\alpha]=1$. Hence, in both cases $[\ov{P},\alpha]=1$ and so coprime action gives $[Z_0,\alpha]\neq 1$. Now (6) yields $G\cong Sp_4(4)$. Therefore, $C_Z(t)=Z(\Omega(P))\cap Z_0$ is $\alpha$-invariant. Now, by (2) and (6), in the series
$$1\neq Z(S)=\Omega(Z(P))\leq C_Z(t)\leq C_Z(t)[Z_0,P]\leq Z_0$$
every factor has order at most $2$. Hence, $[Z_0,\alpha]=1$, a contradiction. This shows (7). We prove next

\bigskip

(8)\;\;$T=J(S)(Z_0\alpha)$ and $[Z_0,Z_0\alpha]\neq 1$.

\bigskip

Note that $Z_0\alpha$ is an elementary abelian normal subgroup of $P$. Hence, by (1), $(Z_0\alpha)\cap J(S)\leq Z_0$ and so, by (7), $Z_0\alpha\not\leq J(S)$. Moreover, $[P\cap Q,Z_0\alpha]\leq J(S)\cap (Z_0\alpha)\leq Z_0\leq P\cap Q$ and so, again by (1), $Z_0\alpha\leq T$ as $Q\cap P\not\leq Z$. This shows $T=J(S)(Z_0\alpha)$ and (8) follows from (5).

\bigskip

Set now $P^*:=P$ if $G\cong L_3(4)$, and $P^*:=\Omega(P)$ if $G\cong Sp_4(4)$. We show next

\bigskip

(9)\;\;$P^*\cap J(S)=C_{J(S)}(t)Z_0$ and $|Z_0:C_Z(t)|=2.$ 

\bigskip

Set $U:=C_{P^*}(Z_0)$ and observe that $|P^*:U|=2$. Hence, also $|P^*:(U\alpha)|=2$ and $|(P^*\cap J(S)):(J(S)\cap (U\alpha))|\leq 2$. By the structure of $Aut(G)$, we have $|C_{J(S)}(u)|\leq |C_{J(S)}(t)|$, for every involution $u\in T\backslash J(S)$. Hence, by (8), $|J(S)\cap (U\alpha)|\leq |C_{J(S)}(t)|$. Now (9) follows from (5). We show next

\bigskip

(10)\;\;$G\cong Sp_4(4)$.

\bigskip

Assume $G\cong L_3(4)$. Then, by (9), $P=C_S(t)Z$. By (2) and (6), $[Z(S),\alpha]=1$. Observe $$\ov{P}:=P/Z(S)=\<\ov{W},\ov{s}\>\times\<\ov{t}\>\times\ov{Z}\cong D_8\times C_2\times C_2,$$  
$D:=\<Z,t\>\cong D_8$, and $Z(\ov{P})=\ov{P^\prime D}\cong C_2\times C_2\times C_2$. So $P^\prime D$ is characteristic in $P$. Furthermore, by (2), $Z(S)=[W,W^s]\leq P^\prime$, and so we have $P^\prime \cong C_4$, $D\cap P^\prime=Z(S)$ and $P^\prime D\cong C_4*D_8$. Hence, by Lemma \ref{Aut2Group}, $[P^\prime D,\alpha]=1$. Moreover, $\ov{P^\prime D}=Z(\ov{P}) $ has index $2$ in 
$$\<x\in \ov{P}:o(x)=4\>\cong C_4\times C_2\times C_2$$ 
and hence, $[P,\alpha]=1$. This shows (10). We show next

\bigskip

(11)\;\;$C_Z(t)\alpha=C_Z(t)$.

\bigskip

Note that, by (1) and (8), $Z_0\cap (Z_0\alpha)=(Z_0\alpha)\cap J(S)\leq C_Z(Z_0\alpha)=C_Z(t)$ and $|Z_0/(Z_0\cap (Z_0\alpha))|=2$. Hence, by (5), $Z_0\cap (Z_0\alpha)=C_Z(t)$. The same holds with $\alpha^2$ in place of $\alpha$, so $Z_0\cap (Z_0\alpha^2)=C_Z(t)$. Hence, $C_Z(t)\leq (Z_0\alpha)\cap (Z_0\alpha^2)$ and, as $|C_Z(t)|=|(Z_0\alpha)\cap (Z_0\alpha^2)|$, we have $C_Z(t)=(Z_0\alpha)\cap (Z_0\alpha^2)=(Z_0\cap (Z_0\alpha))\alpha=C_Z(t)\alpha$. This shows (11).

\bigskip

We now derive the final contradiction. Set $\hat{P}:=P/C_Z(t)$. If $|\widehat{P\cap J(S)}|\leq 2^3$, then by Lemma~\ref{Aut2Group}, $Aut(\widehat{P})$ is a $2$-group and (11) implies $[P,\alpha]=1$, a contradiction. Therefore, $|\widehat{P\cap J(S)}|\geq 2^4$ and so, by (9), $(P\cap J(S))Q=J(S)$. Hence, $[W,P\cap J(S)]\not\leq C_Z(t)$ and, again by (9), $Z_0=(P^\prime\cap Z)C_Z(t)$. As $P^\prime\leq J(S)$ it follows from (1) that $\Omega(Z(P^\prime))=Z\cap P^\prime$. Now (11) yields a contradiction to (7).
\end{proof}

\begin{lemma}\label{Bound|Q|}
Assume Hypothesis~\ref{ClassHyp}. Then $|Q|\leq q^3$.
\end{lemma}

\begin{proof}
By Hypothesis~\ref{ClassHyp}(iii), we can choose $t\in N_S(T)\backslash T$ such that $t^2\in T$, and have  then $U:=C_Q(M)\cap C_Q(M)^t=1$. So, as $Q/C_Q(M)$ is a natural $SL_2(q)$-module for $M/Q$,
$$|C_Q(M)|=|C_Q(M)^t/U|=|C_Q(M)^tC_Q(M)/C_Q(M)|\leq |Z(J(T))/C_Q(M)|\leq q$$
and $|Q|\leq q^3$.
\end{proof}

\begin{lemma}\label{StructureSC}
Assume Hypothesis~\ref{ClassHyp}. If $q>2$ or $|Q|>4$ then $J(T)=J(S)$ and $|S:T|=2$.
\end{lemma}

\begin{proof}
This is a direct consequence of Lemma~\ref{StructureS}(b),(c).
\end{proof}

In the next proof and throughout this section we will use the well-known fact that a $2$-group $S$ is dihedral or semidihedral if it contains a subgroup $V$ such that $V\cong C_2\times C_2$ and $C_S(V)\leq V$.

\begin{lemma}\label{HelpP}
Assume Hypothesis~\ref{ClassHyp} and $T=J(T)$. Let $|Q|>4$ and $P\in\F\backslash (\{T\}\cup Q^\F)$ be essential in $\F$. Then the following hold.
\begin{itemize}
\item[(a)] $P\not\leq T$ and $P\cap T$ is not $A(P)$-invariant.
\item[(b)] $\Omega(Z(P))\cap T=Z(S)$.
\item[(c)] $P$ is not elementary abelian.
\item[(d)] If $Z(S)$ is $A(P)$-invariant then $Z(T)\leq P$.
\end{itemize}
\end{lemma}

\begin{proof}
Recall that by Lemma~\ref{StructureSC}, $|S:T|=2$. Let $t\in S\backslash T$. Assume $P\cap T$ is $A(P)$-invariant. If $P\leq T$ then, as $P$ is centric, $Z(T)< P$. Since $P$ is centric, $Q$ is abelian and $P\not\in Q^\F$, we have $P\not\leq Q$ and $P\not\leq Q^t$. So by Lemma~\ref{AT}, $\Omega(Z(P))\leq C_Q(P)=C_{Q^t}(P)=Z(T)$. Thus, $Z(T)=\Omega(Z(P))$ is $A(P)$-invariant. If $P\not\leq T$ then we may take $t\in P$, so again by Lemma~\ref{AT}, $\Omega(Z(P\cap T))=Z(T)\cap P$. So in any case, $Z(T)\cap P$ is $A(P)$-invariant. Set $X:=\<(T_P)^{A(P)}\>$. As $T$ is normal in $S$, $[P,N_T(P)]\leq P\cap T$, so as $P\cap T$ is $A(P)$-invariant, $[P,X]\leq P\cap T$. Then $[P\cap T,N_T(P)]\leq P\cap [T,T]=Z(T)\cap P$, so as $Z(T)\cap P$ is $A(P)$-invariant, $[P\cap T,X]\leq Z(T)\cap P$. Similarly, $[Z(T)\cap P,X]=1$. Hence, $X$ is a normal $2$-subgroup of $A(P)$. Since $P$ is essential, this yields $T_P\leq X \leq Inn(P)$ and $T\leq P$, a contradiction to $P\neq T$. This shows (a). In particular, $S=TP$ and, by Lemma~\ref{AT}, $\Omega(Z(P))\cap T\leq Z(S)$. Since $P$ is centric and $Z(S)$ is elementary abelian, this shows (b).

\bigskip

For the proof of (d) assume that $Z(S)$ is $A(P)$-invariant.  As $|S:T|=2$, $S$ acts quadratically on $Z(T)$ and so $[Z(T),P]\leq Z(S)$. Hence, $[P,Y]\leq Z(S)$ and $[Z(S),Y]=1$ for $Y:=\<{(Z(T)_P)}^{A(P)}\>$. Therefore, $Y$ is a normal $2$-subgroup of $A(P)$ and, as $P$ is essential, we get $Y\leq Inn(P)$ and $Z(T)\leq P$. This shows (d).

\bigskip

Assume now $P$ is elementary abelian. Since $|Q|>4$, $S$ is not dihedral or semidihedral and hence,

\bigskip

(1)\;\; $|P|>4$.

\bigskip

By (b), $P\cap T=Z(S)$. Hence, 

\bigskip

(2)\;\; $|P/C_P(S_P)|=|P/Z(S)|=2$. 

\bigskip

Moreover, by Hypothesis~\ref{ClassHyp}(iii), $P\cap C_Q(M)=1$. Thus, $|P\cap T|\leq q$ and so $|P|\leq 2\cdot q$. In particular, by (1),

\bigskip

(3)\;\; $q>2$.

\bigskip

Since $P$ is essential, $A(P)$ has a strongly $2$-embedded subgroup. So there exists $\phi\in A(P)$ such that $S_P\cap {S_P}\phi^*=1$. Set $L=\<S_P,S_P\phi^*\>$. Then, by (2), $\ov{P}:=P/C_P(L)$ has order $4$ and $L/C_L(\ov{P})\cong S_3$. Observe that $C_L(\ov{P})$ is a normal $2$-subgroup of $L$ and thus contained in $S_P\cap S_P\phi^*=1$. Hence, $L\cong S_3$ and $|N_S(P):P|=|S_P|=2$. As $[Z(T),P]\leq Z(S)$, we have $Z(T)\leq N_S(P)$. Therefore, $|Z(T)/Z(S)|= |Z(T)/Z(T)\cap P|\leq 2$. As $|S:T|=2$ and $T=QQ^t$, $q=|C_{T/Z(T)}(S)|$. Thus, if $Z(T)=Z(S)$ then $q=|C_{T/Z(S)}(S)|\leq |N_T(P)/Z(S)|\leq 2$, a contradiction to (3). Hence, $|Z(T)/Z(S)|=2$. So by (3), $G=G(T)$ fulfills the Hypothesis of Lemma~\ref{HS2}, for a subgroup $K$ of $N_{G(T)}(Q)$ such that $|K|=q-1$ and $Aut_K(Q)$ is a Cartan subgroup of $Aut_M(Q)$. Hence, Lemma~\ref{HS2} yields $q\leq |Z(T)|=4$. Thus, $|Z(S)|=2$ and (2) yields a contradiction to (1). This shows (c). 
\end{proof}

\subsection{The case $q=2$}

Throughout this section assume Hypothesis~\ref{ClassHyp} and $q=2$. Note that $T/Q$ embeds into $Aut(M/Q)\cong Aut(SL_2(q))$ and, by Lemma~\ref{AT}, $J(T)\in Syl_2(M)$. Therefore, $T=J(T)\in Syl_2(M)$.

\begin{lemma}\label{Psemidi}
Assume $|Q|=4$ and let $P$ be essential in $\F$. If $P$ is not a fours group, then $P$ is quaternion of order $8$, $A(P)=Aut(P)$, and $S$ is semidihedral of order $16$.
\end{lemma}

\begin{proof}
It follows from $|Q|=4$ that $S$ is dihedral or semidihedral. Let $X\leq S$ be cyclic of index $2$. As $Aut(P)$ is not a $2$-group, $P\not\leq X$ and $P\cap X$ is not characteristic in $P$. Assume now $P$ is not a fours group. Then $|P\cap X|=4$, $S$ is semidihedral, $P$ is quaternion of order $8$, $Z(P)=Z(S)$ and $A(P)=Aut(P)\cong S_4$. In particular, $N_\F(P)$ is a subsystem of $\m{N}:=N_\F(Z(S))$ and $P$ is essential in $\m{N}$. By Hypothesis~\ref{ClassHyp}, $\m{N}$ is constrained and so $Z(S) < O_p(\m{N})$. Now it follows from Lemma~\ref{NormalEssential} that $P=O_p(\m{N})$. In particular, $P$ is normal in $S$ and thus $S$ has order $16$.
\end{proof}

\begin{lemma}\label{q=2Q=4}
Assume $|Q|=4$. Then there exists a finite group $G$ such that $S\in Syl_p(G)$, $\F\cong \F_S(G)$ and one of the following holds.
\begin{itemize}
\item[(a)] $S$ is dihedral and $G\cong PGL_2(r)$ or $L_2(r)$ for an odd prime power $r$.
\item[(b)] $S$ is semidihedral and, for some odd prime power $r$, $G$ is an extension of $L_2(r^2)$ by an automorphism of order $2$.
\item[(c)] $S$ is semidihedral or order $16$ and $G\cong L_3(3)$.
\end{itemize}
\end{lemma}

\begin{proof}
Recall that $S$ is dihedral or semidihedral. Note that $A(S)=Inn(S)$ since $S$ has no automorphisms of odd order. By Lemma~\ref{Psemidi}, $A(P)=Aut(P)$ for every essential subgroup $P$ of $\F$, and either every essential subgroup of $\F$ is a fours group, or $S$ is semidihedral of order $16$ and the only essential subgroup of $S$ that is not a fours group is the quaternion subgroup of $S$ of order $8$. If $S$ is dihedral then there are two conjugacy classes of subgroups of $S$ that are fours groups, and they are conjugate under $Aut(S)$. By Remark~\ref{AlpIso}, if $\F$ has only one conjugacy class of essential subgroups then $\F$ is isomorphic to the $2$-fusion system of $PGL_2(r)$, and if $\F$ has two conjugacy classes of essential subgroups then $\F$ is isomorphic to the $2$-fusion system of $L_2(r)$, in both cases for some odd prime power $r$. Let now $S$ be semidihedral. Then $S$ has only one conjugacy class of fours groups. Recall that there is always an odd prime power $r$ and an extension $H$ of $L_2(r^2)$ by an automorphism of order $2$ that has semidihedral Sylow $2$-subgroups of order $|S|$. If the fours groups are the only essential subgroups in $\F$ then it follows from Remark~\ref{AlpIso} that $\F$ is isomorphic to the $2$-fusion system of $H$. Otherwise, it follows from the above and Remark~\ref{AlpIso} that $\F$ is isomorphic to the $2$-fusion system of $L_3(3)$.
\end{proof}

\begin{lemma}\label{Help|Q|=8}
Assume $|Q|>4$. 
\begin{itemize}
\item[(a)] $|Q|=8$, $M=G(Q)\cong S_4\times C_2$ and $Z(S)=\Phi(T)\leq [Q,M]$.
\item[(b)] Let $u\in [Q,M]\backslash Z(S)$ and $1\neq c\in C_Q(M)$. Then there exists an element $y\in S\backslash T$ of order $8$ such that $y^u=y^{-1}$, $y^c=y^5$ and $S=\<c,u\>\ltimes \<y\>$. In particular, $S$ is uniquely determined up to isomorphism.
\end{itemize} 
\end{lemma}

\begin{proof}
It follows from Lemma~\ref{Bound|Q|} that $|Q|=8$. Hence, $M=G(Q)\cong S_4\times C_2$. In particular, $T\cong D_8\times C_2$ and so $|\Phi(T)|=2$. Now Hypothesis~\ref{ClassHyp}(iii) implies $Z(S)=\Phi(T)$. In particular, $Z(S)\leq [Q,M]$. This shows (a).

\bigskip

Recall that by Lemma~\ref{StructureSC}, $|S:T|=2$.
Set $\ov{S}=S/Z(T)$ and $C=C_Q(M)$. Note that $\ov{T}$ is elementary abelian, and $\ov{S}$ is non-abelian, since $Q$ is not normal in $S$. In particular, there exists an element $y\in S\backslash T$ such that $\ov{y}$ has order $4$. Then $y^4\in Z(T)$, and $S=T\<y\>$ implies $y^4\in Z(S)$.
If $y^4=1$ then $y^2\in T$ is an involution and, by Lemma~\ref{AT},  $y^2=(y^2)^y\in Q\cap Q^y=Z(T)$, so $\ov{y}$ has order $2$, a contradiction. Therefore, $y^4$ is an involution and $y$ has order $8$.

\bigskip

Since $Z(T)$ is normal in $S$ and $[Z(T),y^2]=1$, $y$ acts quadratically on $Z(T)$. Hence, $[C,y]\leq [Z(T),y]\leq C_{Z(T)}(y)=Z(S)=\<y^4\>$, so $\<y\>$ is normalized by $C$. By Hypothesis~\ref{ClassHyp}(iii), $[C,y]\neq 1$. Now $[y^2,C]=1$ implies $y^c=y^5$. Set $N:=\<y\>C$. Observe that $\<y\>$ and $\<yc\>$ are the only cyclic subgroups of $N$ of order $8$. Moreover, $|S:N|=2$ and so $N$ is normal in $S$. Hence, $u$ acts on $N$ and either normalizes $\<y\>$ or swaps $\<y\>$ and $\<yc\>$.

\bigskip

Assume first $y^u\in \<yc\>=\<y^2\>\cup \<y^2\>yc$. Since $y^u\not\in T$, $y^u=y^i c$ for some $i\in\{1,3,5,7\}$. Then
$y^2\neq (y^2)^u=(y^u)^2=(y^i c)^2=y^i(y^i)^c=y^i(y^c)^i=y^iy^{5i}=y^{6i}$ implies $i\in\{1,5\}$. Hence, $[y,u]=y^{-1}y^u\in \<y^4\>c=Z(S)c\leq Q$ and $Q$ is normalized by $y$, a contradiction. Thus, $\<y\>$ is normal in $S$ and $S$ is the semidirect product of $\<c,u\>$ and $\<y\>$. Since $[y^2,u]\neq 1$, $y^u\in\{y^{-1},y^3\}$. If $y^u=y^3$ then $(yc)^u=y^3c=(yc)^{-1}$ and $(yc)^c=y^5c=(yy^5)^2yc=(yy^c)^2yc=(yc)^5$, so we may in this case replace $y$ by $yc$ and assume $y^u=y^{-1}$. This shows (b).
\end{proof}

\begin{lemma}\label{q=2Q>4}
Assume $|Q|>4$. Then there exists a finite group $G$ such that $S\in Syl_p(G)$, $\F\cong \F_S(G)$ and $G\cong Aut(A_6)$ or $Aut(L_3(3))$.
\end{lemma}

\begin{proof}
Let $G$ be a finite group isomorphic to $Aut(A_6)$ or $L_3(3)$, $\hat{S}\in Syl_2(G)$, $\hat{Q}\in \m{A}(\hat{S})$ and $\hat{\F}=\F_{\hat{S}}(G)$. Then by the structure of $G$, $\hat{Q}$ is essential in $\hat{\F}$ and $\hat{M}:=N_G(\hat{Q})\cong C_2\times S_4$. By Lemma~\ref{Help|Q|=8}, there is a group isomorphism $\alpha: S\rightarrow \hat{S}$ such that $Q\alpha=\hat{Q}$ and $[Q,M]\alpha =[\hat{Q},\hat{M}]$. This implies $\alpha^{-1}A(Q)\alpha=Aut_{\hat{\F}}(Q\alpha)$.

\bigskip

Assume first $G\cong Aut(A_6)$ and observe that $\hat{Q}^{\hat{\F}}$ is the only essential class in $\hat{\F}$. Therefore, if $Q^\F$ is the only essential class in $\F$, then it follows from Remark~\ref{AlpIso} that $\F\cong \hat{\F}$.

\bigskip

Therefore, we may assume from now on that there is an essential subgroup $P\in \F\backslash Q^\F$. By Lemma~\ref{AT}, $|\m{A}(T)|=2$ and $Q\in \m{A}(T)$. Hence, every automorphism of $T$ of odd order normalizes $Q$. Thus, Lemma~\ref{Help|Q|=8}(a) implies that $Aut_\F(T)$ is a $2$-group and thus $P\neq T$. 
Therefore, it follows from Lemma~\ref{HelpP} that $P\not\leq T$ and $P$ is not elementary abelian. In particular, as $|S:T|=2$ by Lemma~\ref{StructureSC}, we have $S=TP$. We first show

\bigskip

(1)\;\; $Z(T)\leq P$.

\bigskip

By Lemma~\ref{HelpP}(d), it is sufficient to show that $Z(S)=\Omega(Z(P))$. By Lemma~\ref{HelpP}(b), we may assume that $\Omega(Z(P))\not\leq T$. As $P$ is not elementary abelian and $|S:T|=2$, $P\cap T\not\leq Z(S)$. By Lemma~\ref{AT}, we have $Q\cap P \leq C_Q(\Omega(Z(P)))=Z(S)$. So $|P\cap T|=4$ and $P$ is dihedral of order $8$. Then $A(P)$ is $2$-group contradicting $P$ being essential. This shows (1). 

\bigskip

By Lemma~\ref{Help|Q|=8}, we can choose $y\in S\backslash T$, $1\neq c\in C_Q(M)$ and $u\in [Q,M]\backslash Z(S)$ such that $o(y)=8$, $y^c=y^5$ and $y^u=y^{-1}$. Since $A(P)$ is not a $2$-group, $P$ is not dihedral of order $8$ and so  $P\cap T\neq Z(T)$. By (1), $Z(T)\leq P$. If $Q\leq P$ then, as $P\not\leq T$, $T=J(T)\leq P$ and $S=P$, a contradiction. Hence, $|P\cap T|=8$, $\Omega(P\cap T)=Z(T)$, and there is an element of order $4$ in $P\cap T$. As $T=\<y^2,c,u\>\cong D_8\times C_2$, this gives $$P\cap T=Z(T)\<y^2\>=\<c,y^2\>=\<y^2,y^2c\>=\<a\in T\;:\;o(a)=4\>.$$ 
By Lemma~\ref{HelpP}(a), $P\cap T$ is not $A(P)$-invariant. So, as the elements in $P\cap T$ have order at most $4$, there is an element $x\in P\backslash T$ of order at most $4$. As $|P/(P\cap T)|=2$ we have $P=(P\cap T)\<x\>$. Moreover, by Lemma~\ref{AT}, $x^2\in Z(T)$. So $Z(T)\<x\>$ is dihedral of order $8$ and we may assume $o(x)=4$. Set $t:=uc$ and note that $y^t=y^3$. An easy calculation shows that the elements in $S\backslash T$ of order $4$ are precisely the elements of the form $y^it$ for some odd integer $i$. Hence, $P=\<y^2,c,yt\>$ and we may assume $x=yt$. Moreover, this shows

\bigskip

(2)\;\; $P=\<a\in S\;:\;o(a)=4\>$

\bigskip

In particular, the arbitrary choice of $P$ yields

\bigskip

(3)\;\;$P$ is the only essential subgroup of $\F$ in $\F\backslash Q^\F$.

Note that $(y^2c)^x=(y^2c^y)^t=(yy^cc)^t=(y^6c)^t=y^2c$ and so $[x,y^2c]=1$. Moreover, as $[y^2,c]=1$ and $[y^2,t]\neq 1$, we have $[y^2c,y^2]=1$ and $[y^2,x]\neq 1$. Also observe $o(y^4)=2$, $(y^2c)^2=y^4=(y^2)^2$ and $x^2=yy^t=y^4$. Hence, $\<y^2,x\>\cong Q_8$ and $P=\<y^2,x\>\<y^2c\>\cong Q_8* C_4$. In particular, $Aut(P)/Inn(P)\cong S_3\times C_2$ and, as $A(P)/Inn(P)$ has a strongly $2$-embedded subgroup, the following property holds:

\bigskip

(4)\;\; $P\cong Q_8*C_4$, $A(P)=O^2(Aut(P))S_P$ and $A(P)/Inn(P)\cong  S_3$.

\bigskip

Let now $G$, $\hat{S}$, $\hat{Q}$, $\hat{\F}$ and $\alpha$ be as above and assume $G\cong Aut(L_3(3))$. Then Hypothesis~\ref{ClassHyp} is fulfilled with $\hat{\F}$ in place of $\F$. Moreover, $$\hat{P}=\<a\in \hat{S}: o(a)=4\>\cong Q_8*C_4$$ 
is essential in $\hat{\F}$. So (3) and (4) applied to $\hat{\F}$ instead of $\F$ give that $Aut_{\hat{\F}}(\hat{P})=O^2(Aut(\hat{P}))\hat{S}_{\hat{P}}$, and $\hat{P}$ is the only essential subgroup of $\hat{\F}$ in $\hat{\F}\backslash \hat{Q}^{\hat{\F}}$. It follows from (2) and (3) that $P$ is the only essential subgroup of $\F$ in $\F\backslash Q^\F$, and that $P\alpha =\hat{P}$. By (4), $\alpha^{-1}A(P)\alpha=Aut_{\hat{\F}}(P\alpha)$ and so by Remark~\ref{AlpIso}, $\alpha$ is an isomorphism from $\F$ to $\hat{\F}$. This shows the assertion.
\end{proof}

\subsection{The case $q\geq 4$}

Throughout this section assume Hypothesis~\ref{ClassHyp} and $q\geq 4$. 

\bigskip

Set $G_1=G(J(T))$, $G_2=G(Q)$, $M=J(G_2)$ and $\F_0:=\<N_\F(J(T)), N_\F(Q)\>$. 

\bigskip

We will use from now on without reference that, by Lemma~\ref{AT} and Lemma~\ref{StructureSC}, $J(S)=J(T)\in Syl_2(M)$ and $|S/T|=2$. In particular, $N_\F(J(T))$  is parabolic, so by Hypothesis~\ref{ClassHyp}, $N_\F(J(T))$ is constrained and $G_1$ is well-defined. Moreover, $\F_0$ is a fusion system on $S$, and $S\in Syl_2(G_1)$.

\begin{lemma}\label{BuildAmalgam}
There is an isomorphism $\phi$ from $N_{G_1}(Q)$ to $N_{G_2}(J(T))$ such that $\phi$ is the identity on $T$. If $X=G_1*_{N_{G_1}(Q)}G_2$ is the free amalgamated product with respect to $\m{A}=(G_1,G_2,N_{G_1}(Q),id,\phi)$, then $\F_0=\<\F_S(G_1),\F_S(G_2)\>=\F_S(X)$.
\end{lemma}

\begin{proof}
Observe that $\m{N}:=\F_T(N_{G_1}(Q))=N_{N_\F(J(T))}(Q)=N_{N_\F(Q)}(J(T))
=\F_T(N_{G_2}(J(T)))$. Also note that, as $Q$ is fully normalized in $\F$, $Q$ is fully normalized in $N_\F(J(T))$, so $\m{N}$ is saturated. Moreover, $N_{G_2}(J(T))$ has characteristic $2$, since $G_2$ has characteristic $2$. Let $x\in C_{G_1}(J(T))$ be of odd order. As $|S:T|=2$ and $T/J(T)$ is cyclic, we have $[O_2(G_1), x]=[O_2(G_1),x,x]\leq [J(T)\cap O_2(G_1),x]=1$. Hence, $x\leq C_{G_1}(O_2(G_1))\leq O_2(G_1)$ and $x=1$. This proves $C_{G_1}(J(T))\leq O_2(G_1)$, and $N_{G_1}(Q)$ has characteristic $2$. 
Therefore, $N_{G_1}(Q)$ and $N_{G_2}(J(T))$ are models for $\m{N}$. Hence, by Theorem~\ref{ModExUnique}, there exists an isomorphism $\phi$ between these two groups that is the identity on $T$. Now the assertion follows from Theorem~\ref{Rob} and the definitions of $G_1$, $G_2$ and $\F_0$.
\end{proof}

\begin{lemma}\label{F0}
There exists a finite group $G$ with $S\in Syl_2(G)$ such that $\F_0=\F_S(G)$ and $F^*(G)\cong L_3(q)$ or $Sp_4(q)$.
\end{lemma}

\begin{proof}
Let $X$ be as in Lemma~\ref{BuildAmalgam}. By Theorem~\ref{MainAmalgamThm}, there is a free normal subgroup $N$ of $X$ such that $N\cap G_1=1=N\cap G_2$, $\ov{X}:=X/N$ is finite, $\ov{S}\in Syl_2(\ov{X})$, and $F^*(\ov{X})\cong L_3(q)$ or $Sp_4(q)$. It is elementary to check that the natural epimorphism from $S$ to $\ov{S}$ is an isomorphism from $\F_0$ to a subsystem $\F_1$ of $\F_{\ov{S}}(\ov{X})$ containing $\F_{\ov{S}}(\ov{G_1})$ and $\F_{\ov{S}}(\ov{G_2})$. By the structure of $Aut(F^*(\ov{X}))$, $\F_{\ov{S}}(\ov{X})$ is generated by $\F_{\ov{S}}(\ov{G_1})$ and $\F_{\ov{S}}(\ov{G_2})$. This implies $\F_1=\F_{\ov{S}}(\ov{X})$ and thus the assertion.
\end{proof}

\begin{lemma}\label{TJT}
$T=J(T)=J(S)$.
\end{lemma}

\begin{proof}
Set $Z:=Z(J(S))$ and let $G$ be a finite group such that $S\in Syl_2(G)$,  $\F_0=\F_S(G)$ and $F^*(G)\cong L_3(q)$ or $Sp_4(q)$. Note that $G$ exists by \ref{F0}. By the structure of $Aut(L_3(q))$ respectively $Aut(Sp_4(q))$ and by Lemma~\ref{AT}, the following properties hold:

\bigskip

(1)\;\; $J(S)=J(T)\in Syl_2(F^*(G))$, and $N_S(P)=T$ for all $P\in\m{A}(S)$.

\bigskip

(2)\;\; $\m{A}(S)=\{Q,Q^x\}$ and $Q\cap Q^x=Z$ for all $x\in S\backslash T$.

\bigskip

(3)\;\; Every elementary abelian subgroup of $J(T)$ is contained in an element of $\m{A}(S)$.

\bigskip

Assume the assertion is wrong. Then we can pick $t\in T\backslash J(T)$ corresponding to a field automorphism of $F^*(G)$. In particular, there is $q_0\in\mathbb{N}$ such that $q_0^2=q$ and $C_{F^*(G)}(t)\cong L_3(q_0)$ or $Sp_4(q_0)$. 
If $F^*(G)\cong L_3(q)$ then, by the structure of $Aut(L_3(q))$, $S=C_S(t)J(T)$. If $F^*(G)\cong Sp_4(q)$ then, by \cite[Section~12.3]{Car}, we can choose $s\in S\backslash T$ such that $S=J(S)\<s\>$, $J(S)\cap \<s\>=1$ and $t\in \<s\>$. In both cases, set
$$W:=C_Q(t),\;L:=O^{p^\prime}(N_{F^*(G)}(W)\cap C_{F^*(G)}(t))\mbox{ and }L^*:=L(C_S(t)\cap N_S(W)).$$
The structure of $Aut(L_3(q))$, $Aut(Sp_4(q))$, $L_3(q_0)$ and $Sp_4(q_0)$ gives also the following properties:

\bigskip

(4)\;\; $L/W\cong SL_2(q_0)$, and $W/C_W(L)$ is a natural $SL_2(q_0)$-module for $L/W$.

\bigskip

(5)\;\; $C_S(W)\cap C_S(t)=W\<t\>$ and $L^*/C_{L^*}(W)$ embeds into the automorphism group of $$LC_{L^*}(W)/C_{L^*}(W)\cong L/W\cong SL_2(q_0).$$ 
In particular, $O_2(L^*/C_{L^*}(W))=1$.

\bigskip

(6)\;\; $Z(S)\cap C_W(L)=1$.

\bigskip

(7)\;\; $|C_A(x)|\leq |W|$ for every $A\in \m{A}(T)$ and every $x\in T\backslash J(T)$.

\bigskip

(8)\;\; For every involution $u\in T\backslash J(T)$ we have $|C_{J(T)}(u)|\leq |C_{J(T)}(t)|$.

\bigskip
 
Let $R\in\m{A}(T)\backslash\{Q\}$ and set $\hat{W}:=C_R(t)$. 
Note that the situation is symmetric in $Q$ and $R$. Moreover, $C_{J(S)}(t)=W\hat{W}$ and $\m{A}(C_{J(S)}(t))=\{W,\hat{W}\}$. Hence, (4) and Lemma~\ref{natSL2qOffFactor} give

\bigskip

(9)\;\; $|W/W\cap \hat{W}|=|\hat{W}/W\cap \hat{W}|=q_0$, $W\cap \hat{W}=C_W(\hat{W})=C_W(a)=C_{\hat{W}}(W)=C_{\hat{W}}(b)$ for all $a\in \hat{W}\backslash W$, $b\in W\backslash \hat{W}$.

\bigskip

We show next

\bigskip

(10)\;\; $\<t\>$ is not fully centralized.

\bigskip

Assume (10) is wrong. Then by Hypothesis~\ref{ClassHyp}(iv), $\m{C}:=C_\F(\<t\>)$ is constrained. Set $C:=O_2(\m{C})$ and $F:=N_C(W)$. Observe that every element of $Aut_{\m{C}}(W)$ extends to an element of $Aut_{\m{C}}(WF)$ and hence, by Remark~\ref{SatAut}(a), $(WF)_W$ is normal in $Aut_{\m{C}}(W)$. In particular, $(WF)_W$ is normal in $Aut_{L^*}(W)\cong L^*/C_{L^*}(W)$. So, by (5), $C\cap T\leq F\leq C_S(W)\cap C_S(t)=\<t\>W$ and $C\cap T= \<t\>(W\cap C)$. Since the situation is symmetric in $Q$ and $R$, we get also $C\cap T=\<t\>(\hat{W}\cap C)$ and thus, $(W\cap C)C_W(L)<W$. Hence, as $W/C_W(L)$ is an irreducible $L$-module, $W\cap C\leq C_W(L)$. Therefore, since $\m{C}$ is constrained, $Z(S)\leq C_W(C)\leq W\cap C\leq C_W(L)$, a contradiction to (6). This proves (10). In particular, by  Lemma~\ref{fullNormChar3}, $C_S(t)$ is contained in an essential subgroup of $\F$. Moreover, if $q=4$ then $S$ is isomorphic to a Sylow $2$-subgroup of $Aut(F^*(G))$. Therefore,  Lemma~\ref{NoEssL34} gives

\bigskip

(11)\;\; $q>4$.

\bigskip

In particular, $q_0>2$. Therefore, the structure of $L_3(q_0)$ and $Sp_4(q_0)$ gives 

\bigskip

(12)\;\; $[W,\hat{W}]=W\cap\hat{W}$.

\bigskip

By Lemma~\ref{fullNormChar}, we can choose $\phi\in Mor_\F(C_S(t),S)$ such that $\<t\phi\>$ is fully centralized. 
We set 
$$W_1:=W\<t\> \mbox{ and }\hat{W}_1:=\hat{W}\<t\>.$$
Note that $W_1$ and $\hat{W}_1$ are elementary abelian. We show next

\bigskip

(13)\;\; $|W_1\phi/W_1\phi\cap J(T)|\leq 2$ and $|\hat{W}_1\phi/\hat{W}_1\cap J(T)|\leq 2$.

\bigskip

If $F^*(G)\cong Sp_4(q)$ then $S/J(T)$ is cyclic and hence (13) holds, as $W_1\phi$ is elementary abelian. Thus, we may assume $F^*(G)\cong L_3(q)$ and  $|W_1\phi/W_1\phi\cap J(T)|=4$. Then, as $T/J(T)$ is cyclic, $t\in J(T)(W_1\phi\cap T)$ and $W_1\phi\not\leq T$. Hence, by (2) and (3), we have $W_1\phi\cap J(T)\leq C_{Z}(t)=W\cap \hat{W}$. Therefore, $2\cdot q_0^2=|W_1\phi|\leq 4\cdot |W\cap \hat{W}|=4\cdot q_0$ and $q_0\leq 2$, a contradiction to (11). As the situation is symmetric in $W$ and $\hat{W}$, this shows (13).

\bigskip

(14)\;\; $W_1\phi\leq T$ and $\hat{W}_1\phi\leq T$.

\bigskip

Assume $W_1\phi\not\leq T$. Then, by Lemma~\ref{AT}, $W\phi\cap J(T)\leq Z(J(T))$ and so $[W\phi\cap J(T),\hat{W}\phi\cap J(T)]=1$. Now (9) and (13) yield $q_0=2$, a contradiction to (11). As the situation is symmetric in $W$ and $\hat{W}$, this shows (14).

\bigskip

(15)\;\; $t\phi\in Z(J(T))$.

\bigskip

By (14), $t\phi\in T$. Hence, as $S=J(T)C_S(t)$ and $t\phi$ is fully centralized, it follows from (8) and (10) that $t\phi\in J(T)$. Suppose now (15) is wrong. Then (2) and (3) imply $t\phi\in P\backslash Z(J(T))$ for some $P\in \m{A}(T)$ and $C_S(t\phi)=C_T(t\phi)$. Moreover, by Lemma~\ref{natSL2q}(a), $C_{J(T)}(t\phi)=P$. Hence, by (13), $|W\phi/W\phi\cap P|\leq 2$ and $|\hat{W}\phi/\hat{W}\phi\cap P|\leq 2$. As $[W\phi\cap P,\hat{W}\phi\cap P]=1$, it follows now from (9) that $q_0=2$, a contradiction to (11). Hence, (15) holds.

\bigskip

(16)\;\; $W\phi\leq J(T)$ and $\hat{W}\phi\leq J(T)$.

\bigskip

By (14), $W\phi\leq T$. By (9),(11) and (13), $[W\phi\cap J(T),\hat{W}\phi\cap J(T)]\neq 1$. 
So, by (3), there are $P_1,P_2\in \m{A}(T)$ such that $P_1\neq P_2$, $W\phi\cap J(T)\leq P_1$ and $\hat{W}\phi\cap J(T)\leq P_2$. As, by (12), $W\phi\cap \hat{W}\phi=[W\phi,\hat{W}\phi]\leq J(T)$, this implies $W\phi\cap \hat{W}\phi\leq P_1\cap P_2=Z$. Assume $W\phi\not\leq J(T)$. Then $t\in (W\phi)J(T)$ and, by (15),  $\<t\phi\>(W\phi\cap \hat{W}\phi)\leq C_{Z}(t)=W\cap \hat{W}$, a contradiction. Hence, $W\phi\leq J(T)$ and, as the situation is symmetric in $W$ and $\hat{W}$, property (16) holds.

\bigskip

(17)\;\; Let $W_1\phi\leq B\in \m{A}(T)$. Then $W_1\phi$ is fully centralized and $C_S(W_1\phi)=B$.

\bigskip

By Lemma~\ref{fullNormChar}, we can choose $\psi\in Mor_\F(N_S(W_1\phi),S)$ such that $\dot{W_1}:=W_1\phi\psi$ is fully normalized. Note $B\leq C_S(W_1\phi)$ and  $\dot{W_1}\leq B\psi\in \m{A}(S)=\m{A}(T)$ and, by (1), $T=N_S(B\psi)$. Let $F:=C_S(\dot{W_1})\cap N_S(B\psi)$. Then $\dot{W_1}\leq C_{B\psi}(F)$ and so $|C_{B\psi}(F)|\geq |W_1|>|W|$. Thus, by (7), $F\leq T$. Assume now $B\psi<F$. Then $W\phi\psi\leq \dot{W_1}\leq C_{B\psi}(F)=Z$. Note that  $(\hat{W}\phi)B\leq N_S(W_1\phi)$ and, by (1) and (14),  $\hat{W}\phi\leq T=N_S(B)$. Hence, we have $\hat{W}\phi\psi\leq N_S(B\psi)=T$. In particular, $|(\hat{W}\phi\psi)/((\hat{W}\phi\psi)\cap J(T))|\leq 2$. As $W\phi\psi\leq Z$, this yields $$|\hat{W}/C_{\hat{W}}(W)|=|(\hat{W}\phi\psi)/C_{\hat{W}\phi\psi}(W\phi\psi)|\leq 2,$$ a contradiction to (9) and (11). This shows $F=B\psi$. Thus,  $C_S(\dot{W_1})=B\psi$ and (17) holds.

\bigskip

We now derive the final contradiction. Set $L_1:=Aut_L(W_1)$. Then $[t,L_1]=1$, $L_1\leq A(W_1)$ and, by (4),  $L_1\cong SL_2(q_0)$. So $L_2:=L_1\phi^*\cong SL_2(q_0)$ and $[t\phi, L_2]=1$. By (15), (16) and (3), there are $B,\hat{B}\in \m{A}(T)$ such that $W_1\phi\leq B$ and $\hat{W}_1\phi\leq \hat{B}$. Set $E:=O^{p^\prime}(Aut_{F^*(G)}(B))$ and $\hat{E}:=O^{p^\prime}(Aut_{F^*(G)}(\hat{B}))$. Note that $B\neq \hat{B}$ and $B,\hat{B}$ are conjugate to $Q$. Hence, it follows from Hypothesis~\ref{ClassHyp}(iii) that $C_B(E)\cap C_{\hat{B}}(\hat{E})=1$. In particular, either $t\phi\not\leq C_B(E)$ or $t\phi\not\leq C_{\hat{B}}(\hat{E})$. As the situation is symmetric in $W$ and $\hat{W}$, we may assume $$t\phi\not\in C_B(E).$$ 
By (17), $W_1\phi$ is fully centralized and $C_S(W_1\phi)=B$. Hence, by the saturation properties, every element of $L_2$ extends to an element of $A(B)$. So, for
$$X:=\{\phi\in A(B):\phi_{|W_1\phi,W_1\phi}\in L_2\},$$
we have $X/C_{X}(W_1\phi)\cong L_2\cong SL_2(q_0)$ and $[t\phi,X]=1$. Note that $E\cong SL_2(q)$, and $B/C_B(E)$ is a natural $SL_2(q)$-module for $E$. 
As $B$ is conjugate to $Q$ in $S$, by Hypothesis~\ref{ClassHyp}, $E$ is a normal subgroup of $A(B)$ and $S_B$ embeds into $Aut(E)$.
By Lemma~\ref{natSL2q}(d), every element of $E$ of odd order acts fixed point freely on $B/C_B(E)$. Therefore, as $[t\phi,X]=1$ and $t\phi\not\in C_B(E)$, $X\cap E$ is a normal $p$-subgroup of $X$. Hence, as $O_p(SL_2(q_0))=1$, we have $X\cap E\leq C_{X}(W_1\phi)$. Since $S_B$ embeds into $Aut(E)$, $A(B)/E$ has cyclic Sylow $2$-subgroups. In particular, $X/C_{X}(W_1\phi)\cong SL_2(q_0)$ has cyclic Sylow $2$-subgroups. This gives $q_0=2$, a contradiction to (11).
\end{proof}

\begin{lemma}\label{F01}
There is a finite group $G$ such that $S\in Syl_p(G)$, $\F_0=\F_S(G)$, $F^*(G)\cong L_3(q)$ or $Sp_4(q)$, $|O^2(G):F^*(G)|$ is odd and $|G:O^2(G)|=2$. Furthermore, if $F^*(G)\cong Sp_4(q)$ then $q=2^e$ where $e\in\N$ is odd. 
\end{lemma}

\begin{proof}
By Lemma~\ref{F0}, there is a finite group $G$ such that $S\in Syl_p(G)$, $\F_0=\F_S(G)$, and $F^*(G)\cong L_3(q)$ or $Sp_4(q)$. By  Lemma~\ref{TJT} and the structure of $Aut(F^*(G))$, no element of $G$ induces a field automorphism of $F^*(G)$ of even order. So $|G:O^2(G)|=2$ and $|O^2(G):F^*(G)|$ is odd. If $q=2^e$ then $Aut(Sp_4(q))/Inn(Sp_4(q))$ is cyclic of order $2e$ and generated by the image of any graph automorphism of $Sp_4(q)$. Hence, as any element of $S\backslash T$ induces a graph automorphism $F^*(G)$, this implies the assertion.
\end{proof}

\begin{lemma}\label{q=Z=4}
Let $\F\neq \F_0$. Then $q=|Z(T)|=4$.
\end{lemma}

\begin{proof}
We will use throughout the proof that, by Lemma~\ref{TJT}, $T=J(T)$. Set $Z:=Z(T)$ and assume $|Z|>4$.
As $\F\neq \F_0$ it follows from Theorem~\ref{AlpGold1} that there is an essential subgroup $P$ of $\F$ such that $P\not\in \{T\}\cup Q^\F$.
Recall that, by Lemma~\ref{HelpP}(a), $P\not\leq T$ and there is $t\in P\backslash T$ such that $t\phi\in T$ for some $\phi\in A(P)$.
By Corollary~\ref{HS1}, applied to $G(Z)/Z$ in place of $G$, we have

\bigskip

(1)\;\; $T$ is strongly closed in $N_\F(Z)$.

\bigskip

We show next

\bigskip

(2)\;\; $Z\neq Z(S)$.

\bigskip

For the proof assume $Z=Z(S)$. If $\Omega(Z(P))\not\leq T$ then Lemma~\ref{StructureS}(d) implies $P\cap T=Z$ and $P$ is elementary abelian, a contradiction to Lemma~\ref{HelpP}(c). Hence, by Lemma~\ref{HelpP}(b), $\Omega(Z(P))=Z(S)$, so $Z=Z(S)$ is $A(P)$-invariant. This is a contradiction to Lemma~\ref{HelpP}(a) and (1), so (2) holds. We show next

\bigskip

(3)\;\; $|Z/Z(S)|>2$.

\bigskip

Assume (3) is wrong, then by (2), $|Z/Z(S)|=2$. Take $K$ to be a subgroup of $G_1=G(T)$ such that $K\cong C_{q-1}$ and $Aut_K(Q)$ is a Cartan subgroup of $Aut_M(Q)$. Then  Lemma~\ref{HS2}, applied with $G_1$ in place of $G$, yields $|Z|=4$. As this contradicts our assumption, (3) holds.

\bigskip

Recall that by Lemma~\ref{HelpP}(c), $P$ is not elementary abelian, i.e. $P_0:=\Omega(C_{\Phi(P)}(P))\neq 1$. Observe that, by Lemma~\ref{HelpP}(b), $P_0\leq Z(S)$, so $N_\F(P_0)$ is parabolic and by assumption constrained. Thus, we may set $$G:=G(P_0).$$ 
As $P_0$ is characteristic in $P$, $A(P)\leq N_\F(P_0)=\F_S(G)$. Hence, there is $g\in G$ such that $t\in T^g$. Thus, by Lemma~\ref{HS1Help}, there exists $h\in G$ such that $Z\<t\>\leq T^h$. 
By Lemma~\ref{AT}, there is $B\in \m{A}(S^h)$ such that $Z\leq B$. Observe that $t\in T^h=N_{S^h}(B)$ and $B\<t\>\leq N_G(Z)$. In particular, there is $x\in N_G(Z)$ such that $B\<t\>\leq S^x$. Then $t\in N_{S^x}(B)=T^x$, a contradiction to (1). This shows the assertion.
\end{proof}

\begin{lemma}\label{q>2FneqF0}
Let $\F\neq \F_0$. Then $\F\cong \F_S(G)$ for a group $G$ with $G\cong J_3$ and $S\in Syl_2(G)$.
\end{lemma}

\begin{proof}
We will use frequently that, by Lemma~\ref{TJT}, $T=J(T)\in Syl_p(M)$. Set 
$$Z:=Z(T),\; R:=[S,S]\mbox{ and }\ov{S}:=S/Z.$$ 
By Lemma~\ref{q=Z=4} we have $q=|Z|=4$. It follows from Theorem~\ref{AlpGold1} that there is an essential subgroup $P$ of $\F$ such that $P\neq T$ and $P\not\in Q^\F\cup \{T\}$. Recall that by Lemma~\ref{HelpP}(a), $P\not\leq T$. Let $t\in P\backslash T$ of minimal order. We will use frequently that, by Lemma~\ref{AT}, $\m{A}(T)=\{Q,Q^t\}$, every elementary abelian subgroup of $T$ is contained in $Q$ or $Q^t$, and $Z=Q\cap Q^t$.
We show first

\bigskip

(1)\;\; $Z(S)$ is $A(P)$-invariant.

\bigskip

If $\Omega(Z(P))\leq T$ then $\Omega(Z(P))=Z(S)$. Thus we may assume $\Omega(Z(P))\not\leq T$.
If $Z=Z(S)$ then, by Lemma~\ref{StructureS}(d), $C_T(z)=Z(S)$ for every involution $z\in S\backslash T$. Hence, $P\cap T=C_T(\Omega(Z(P)))=Z(S)$ and $P$ is elementary abelian, contradicting Lemma~\ref{HelpP}(c).  Thus, $Z\neq Z(S)$ and $|Z(S)|=2$. As $P$ is not elementary abelian and $P=(P\cap T)\Omega(Z(P))$, we have $1\neq \Phi(P)=\Phi(P\cap T)\leq \Phi(T)\cap C(\Omega(Z(P)))=C_Z(\Omega(Z(P)))=Z(S)$. Hence, $Z(S)=\Phi(P)$ and (1) holds. Thus, by Lemma~\ref{HelpP}(d), we have

\bigskip

(2)\;\; $Z\leq P$.

\bigskip

Set now $$R_0=O_p(N_\F(Z)).$$ We show next

\bigskip

(3)\;\; $|(R_0\cap T)/Z|=4$ or $T\leq R_0$.

\bigskip

Assume $T\not\leq R_0$. As $N_\F(Z)$ is by assumption constrained, we have then $Z<R_0$. So, if $Z=R_0\cap T$ then $R_0\not\leq T$ and $[Q,S]\leq [T,TR_0]\leq [T,T](R_0\cap T)= Z$, a contradiction to $Q$ not being normal in $S$. Hence, $Z<R_0\cap T$ and (3) follows from $R_0\cap T$ being $A(T)$-invariant.

\bigskip

(4)\;\; $Z$ is not $A(P)$-invariant and $Z\neq Z(S)$.

\bigskip

Assume $Z$ is $A(P)$-invariant. Then $P$ is essential in $N_\F(Z)$. So by  Lemma~\ref{NormalEssential}, $R_0\leq P$ and $R_0$ is $A(P)$-invariant. If $Z(S)<Z$ then, by (1), $[Z,O^p(A(P))]=1$. If $Z=Z(S)$ then $[Z,O^{p^\prime}(A(P))]=[Z,\<S_P^{A(P)}\>]=1$. So, in any case, $[Z,O^p(O^{p^\prime}(A(P)))]=1$. Hence, as $A(P)$ is not $p$-closed, $[\ov{P},O^p(O^{p^\prime}(A(P)))]\neq 1$. As $P\cap T$ is not $A(P)$-invariant and elementary abelian, there is an involution in $\ov{P}\backslash \ov{P\cap T}$. Moreover, $\ov{P\cap T}$ has order at most $2^3$. Hence, by Lemma~\ref{Aut2Group}, $[\ov{P\cap T},\ov{P}]=1$. Thus, $\ov{P}$ is elementary abelian and $\ov{P\cap T}\leq C_{\ov{T}}(\ov{P})=\ov{R}$. Now, by (3), $R_0\cap T=P\cap T=R$. As $P\cap T$ is not $A(P)$-invariant, this implies $R_0=P$. In particular, $P$ is normal in $S$ and $|S_P/Inn(P)|=|S/P|=4$. 
On the other hand, $|\ov{P}/C_{\ov{P}}(S_P)|\leq 2$ and, as $P$ is essential, we may choose $\phi\in A(P)$ such that $S_P\cap S_P\phi^*= Inn(P)$. Set $Y:=\<S_P,S_P\phi^*\>$. Then $\hat{P}=\ov{P}/C_{\ov{P}}(Y)$ has order $4$ and $Y/C\cong S_3$ for $C=C_Y(\hat{P})$. Since $Y\leq O^{p^\prime}(A(Q))$ and $[Z,O^p(O^{p^\prime}(A(P)))]=1$, we have $[P,O^p(C)]=1$ and $C$ is a normal $2$-subgroup of $Y$. Thus, $C\leq S_P\cap S_P\phi^*=Inn(P)$ and $|S_P/Inn(P)|\leq 2$, a contradiction. Hence, $Z$ is not $A(P)$-invariant and (4) follows from (1).
We show next

\bigskip

(5)\;\; $|S:P|>2$.

\bigskip

Assume $|S:P|=2$. Then $|T:(T\cap P)|=2$. As $T=\<Q^P\>\not\leq P$, we have $Q\not\leq P$ and so $|Q\cap P|=8$. Observe now 
$$J(P\cap T)=\<(Q\cap P),(Q^t\cap P)\>.$$ 
If $J(P)=J(P\cap T)$ then, by Lemma~\ref{AT}, $Z=Z(J(P))$, a contradiction to (4). Thus, $J(P)\not\leq T$. Let $A\in \m{A}(P)$ such that $A\not\leq T$. Then $8=|Q\cap P|\leq |A|=2\cdot |A\cap T|$ and, by Lemma~\ref{AT}, $A\cap T\leq Z(S)$. Hence, $|Z(S)|\geq 4$ and $Z(S)=Z$, again a contradiction to (4). This shows (5).

\bigskip

(6)\;\; $R=P\cap T$.

\bigskip

Set $R_1=O_p(N_\F(Z(S)))$. By (1) and Lemma~\ref{NormalEssential}, $R_1\leq P$ and $R_1$ is $A(P)$-invariant. If $R_1\leq T$ then by (4) and Lemma~\ref{HelpP}(a), $Z<R_1 < P\cap T$. So, by (5), $|\ov{R_1}|=2$. Since $R_1$ is normal in $S$ it follows that $R_1$ is not elementary abelian. Hence, $\Omega(R_1)=Z$, a contradiction to (4). Therefore, $R_1\not\leq T$. Thus, $R\leq [T,R_1]Z\leq R_1Z$ and, by (2), $R\leq P$. Now (6) follows from (5).

\bigskip

(7)\;\; $t$ is an involution, $P=R\<t\>=R\<z : z\in S\backslash T, z^2=1\>$, and the essential subgroups of $\F$ are $Q,Q^t,T$ and $P$.

\bigskip

Observe that every element in $T$ has order at most $4$. By Lemma~\ref{HelpP}(a), $P\cap T$ is not $A(P)$-invariant, and so there is an element $x$ in $P\backslash T$ of order at most $4$. Then $x^2$ has order at most $2$ and is centralized by $x\in S\backslash T$. Hence, $x^2\in Z$ and $\<Z,x\>$ has order $8$. By (4), $\<Z,x\>$ is non-abelian and thus dihedral. Hence, by (2) there is an involution in $P\backslash T$. Since $t$ has minimal order, $t$ is then an involution as well. Hence, the involutions in $\ov{S}\backslash \ov{T}$ are the elements in $C_{\ov{T}}(t)\ov{t}=\ov{R}\ov{t}$. Thus, by (6), every involution in $S\backslash T$ is contained in $P$ and (7) holds. 

\bigskip

(8)\;\; $P\cong Q_8*D_8$, $A(P)/Inn(P)\cong A_5$ and $A(P)=O^p(Aut(P))$.

\bigskip

By (7), $t$ is an involution and $P=R\<t\>$. By (4), $Z\<t\>$ is dihedral of order $8$. An elementary calculation shows  $$C_T(t)=\<qq^tz: q\in Q\backslash Z,\;z\in Z\backslash Z(S)\>Z(S)\cong Q_8$$ 
and $R=P\cap T=C_T(t)Z$. Hence, $P=C_T(t)(Z\<t\>)$. As $C_T(t)\cap (Z\<t\>)=Z(S)$, this shows $P\cong Q_8*D_8$. In particular, $Aut(P)/Inn(P)\cong S_5$ and $P/Z(P)$ is a natural $S_5$-module for $Aut(P)/Inn(P)$. Since, by (7), $S_P/Inn(P)\cong S/P\cong C_2\times C_2$, a Sylow $p$-subgroup of $A(P)/Inn(P)$ is a fours group. As $A(P)/Inn(P)$ has a strongly $2$-embedded subgroup, this implies $A(P)/Inn(P)\cong A_5$ and $A(P)=O^p(Aut(P))$. Hence, (8) holds. We show next:

\bigskip

(9)\;\; $\F_0$ is isomorphic to the $2$-fusion system of the extension of $PGL_3(4)$ by the automorphism that is the product of the contragredient and the field automorphism.

\bigskip

Recall that by Lemma~\ref{q=Z=4}, $q=4$. By (8), we have in particular that $P^\F=\{P\}$, so $P$ is fully normalized. Moreover, there is an element of order $3$ in $N_{A(P)}(S_P)$, which by Remark~\ref{SatAut}(b) extends to an element of $A(S)$. Since $J(S)=T$, it follows from Lemma~\ref{AT}(b) that every element of $A(S)$ of odd order normalizes $Q$. Hence, there is an automorphism of $Q$ of order $3$ which centralizes $Z(S)$. Therefore, $A(Q)\cong GL_2(4)$. Now Lemma~\ref{F01}, (4) and the structure of $Aut(L_3(4))$ imply (9).

\bigskip

We now are able to prove the assertion. Let $G\cong J_3$ and $\hat{S}\in Syl_2(G)$. Set $\hat{\F}=\F_{\hat{S}}(G)$. 
Let $\hat{Q}\in\m{A}(\hat{S})$, $\hat{T}=N_{\hat{S}}(\hat{Q})$ and $\hat{\F_0}=\<N_{\hat{\F}}(\hat{Q}),N_{\hat{\F}}(\hat{T})\>$. From the structure of $J_3$ we will use that Hypothesis~\ref{ClassHyp} is fulfilled with $(\hat{\F},\hat{S},\hat{Q})$ in place of $(\F,S,Q)$, and that there is an essential subgroup $\hat{P}\in \hat{\F}$ with $\hat{P}\not\in \hat{Q}^{\hat{\F}}\cup \{\hat{T}\}$. In particular, the properties we have shown for $\F$ hold for $\hat{\F}$ accordingly. So by (9), we have $\F_0\cong \hat{\F_0}$, i.e. there is a group isomorphism $\alpha:S\rightarrow \hat{S}$ which is an isomorphism of fusion systems from $\F_0$ to $\hat{\F_0}$. By (8), $P$ is the only essential subgroup of $\F$ in $\F\backslash (Q^\F\cup\{T\})$, $\hat{P}$ is the only essential subgroup of $\hat{\F}$ in $\hat{\F}\backslash (\hat{Q}^{\hat{\F}}\cup \{\hat{T}\})$, $P\alpha=\hat{P}$ and $Aut_{\hat{\F}}(\hat{P})=\alpha^{-1}A(P)\alpha$.
Now by Remark~\ref{AlpIso}, $\alpha$ is also an isomorphism from $\F$ to $\hat{\F}$. This shows the assertion.
\end{proof}

\textit{Proof of Theorem~\ref{ClassifyG}.} This is a consequence of Lemma~\ref{q=2Q=4}, Lemma~\ref{q=2Q>4}, Lemma~\ref{F01} and Lemma~\ref{q>2FneqF0}.

\section{Existence of Thompson-restricted subgroups}\label{chapterAut}

Throughout this section assume the following hypothesis.

\begin{hyp}\label{MaxParUnique}
Let $\mathcal{F}$ be a saturated fusion system on a finite $p$-group $S$. Set $$Z:=\Omega(Z(S)).$$ 
Let $\m{N}$ be a proper saturated subsystem of $\F$ containing $C_\F(Z)$. By $\FNT$ denote the set of Thompson-maximal members of $\FN$.
\end{hyp}

Recall here from Notation~\ref{ENDef} that $\FN$ is the set of centric subgroups $P$ of $\F$ such that $Aut_\F(P)\not\leq \m{N}$. 
Note that, by Corollary \ref{Cor1}, $\FN\neq\emptyset$. Also recall the definition of Thompson-restricted subgroups and Thompson-maximality of Definitions~\ref{ThRestrictDef} and \ref{ThMax}. As introduced in Notation \ref{FN*Def}, we write $\FNT$ for the set of Thompson-maximal members of $\FN$. The aim of this section is to prove Theorem~\ref{mainSL2qThm}, i.e. the existence of a Thompson-restricted subgroup of $\F$ in $\F_\m{N}^*$, provided $N_\F(J(S))\leq \m{N}$.  As before, we set, for every $P\in \F$,
$$A(P)=Aut_\F(P).$$
Recall from Notation~\ref{SP} that for $U\in \F$ and $R\leq S$ a subgroup $R_U$ of $A(U)$ is defined by 
$$R_U=\{{c_g}_{|U,U}\;:\;g\in N_R(U)\}.$$
Furthermore, set
$$\FN^+:=\{Q\in \FNT: N_S(Q)=N_S(J(Q))\mbox{ and }J(Q)\mbox{ is fully normalized}\}.$$

\begin{lemma}\label{ThomRestrHelp}
We have $\FN^+\neq\emptyset$.
\end{lemma}

\begin{proof}
This is just a restatement of Lemma~\ref{ThomRestrHelp0}.
\end{proof}

\begin{theorem}\label{ThomRestr}
Let $Q$ be a maximal with respect to inclusion member of $\FN^+$. Then $J(S)=J(Q)$ or $Q$ is Thompson-restricted.
\end{theorem}

\begin{proof}
Suppose $J(S)\not\leq Q$. As $Q\in\FN$, $Q$ is centric, and by Remark~\ref{FullNormQ}, $Q$ is fully normalized. In particular, we may choose a model $G$ of $N_\F(Q)$. Set
$$T:=N_S(Q)\;\mbox{and}\; H:=\{g\in G:\; {c_g}_{|Q,Q}\in Aut_\m{N}(Q)\}.$$
Note that $H$ is a proper subgroup of $G$, as $A(Q)\not\leq \m{N}$. By assumption, $J(S)\not\leq Q$ and $T=N_S(J(Q))$, so it follows from Remark~\ref{ThomQ} that

\bigskip

(1)\;\; $J(T)\not\leq Q$.

\bigskip

Corollary \ref{Cor3} yields $A(RQ)\leq \m{N}$ for every subgroup $R$ of $T$ with $J(RQ)\not\leq Q$. Thus, also the restriction of an element of $N_{A(RQ)}(Q)$ to an automorphism of $Q$ is a morphism in $\m{N}$. This yields 

\bigskip

(2)\;\; $N_G(R)\leq H$ for every subgroup $R$ of $T$ with $J(RQ)\not\leq Q$. 

\bigskip

We show next

\bigskip

(3)\;\; Let $Q_0$ be a normal subgroup of $T$ containing $Q$ such that $N_G(Q_0)\not\leq H$. Then $Q=Q_0$. 

\bigskip

For the proof of (3) set $M:=N_G(Q_0)$. Then every element of $Aut_M(Q)$ extends to an element of $A(Q_0)$. Furthermore, as $M\not\leq H$, we have $Aut_M(Q)\not\leq \m{N}$. Hence, $A(Q_0)\not\leq \m{N}$. Moreover, $Q_0$ is centric, since $Q$ is centric. Thus, $Q_0\in \FN$, so the Thompson-maximality of $Q$ yields $Q_0\in \FNT$ and $J(Q_0)=J(Q)$. In particular, as $Q\in\FN^+$, we have $N_S(Q_0)=T=N_S(J(Q_0))$ and $J(Q_0)$ is fully normalized. Therefore, $Q_0\in\FN^+$ and the maximality of $Q$ yields $Q=Q_0$. This shows (3). 
Note that, by (1) and (2), $N_G(T\cap J(G))\leq N_G(J(T))\leq H$. By a Frattini Argument, $G=N_G(T\cap J(G))J(G)$ and hence,

\bigskip

(4)\;\;$J(G)\not\leq H$. 

\bigskip

In particular, $X:=J(G)T\not\leq H$. 
Let $P\leq X$ be minimal with the property $T\leq P$ and $P\not\leq H$. As $N_G(T)\leq N_G(J(T))\leq H$, it follows from Remark~\ref{minparg} that $P$ is minimal parabolic and $P\cap H$ is the unique maximal subgroup of $P$ containing $T$. 
Observe that $Q\leq O_p(G)\leq O_p(P)$ and so, by (3),

\bigskip

(5)\;\; $O_p(P)=Q=O_p(G)$.

\bigskip

Let now $V\leq \Omega(Z(Q))$ be a normal subgroup of $X$ containing $\Omega(Z(T))$, and $\ov{P}=P/C_P(V)$.
Observe that $Q\leq C_S(V)$ and, by a Frattini Argument, $X=C_X(V)N_X(C_T(V))$. Also note $Z\leq \Omega(Z(T))\leq V$ and so $[C_X(V),Z]=1$. As $C_\F(Z)\leq \m{N}$, this yields $C_X(V)\leq H$ and thus $N_X(C_T(V))\not\leq H$. Now (3) implies $C_T(V)=Q$, i.e. $N_{C_S(V)}(Q)=Q$ and so, as $C_S(V)$ is nilpotent,

\bigskip

(6)\;\; $C_S(V)=Q$ and $C_G(V)/Q$ is a $p^\prime$-group.

\bigskip

In particular, by (1) and Lemma~\ref{Baum}(a), $\m{O}_{\ov{P}}(V)\neq \emptyset$. Let now $N$ be the preimage of $O_p(\ov{P})$ in $P$. Since $P\not\leq H$, we have $[Z,P]\neq 1$ and therefore, $\ov{P}$ is not a $p$-group. Hence, $O^p(P)\not\leq N$, so by Lemma~\ref{minpar}(a) and (5), $N\cap T\leq O_p(P)=Q$. Hence, 

\bigskip

(7)\;\;$O_p(\ov{P})=1$. 

\bigskip

Observe that $C_P(V)\leq C_P(Z)\leq P\cap H$ and therefore, $\ov{P}$ is minimal parabolic, $\ov{H\cap P}$ is the unique maximal subgroup of $\ov{P}$ containing $\ov{T}$, and $C_{\ov{P}}(C_V(\ov{T}))\leq C_{\ov{P}}(Z)\leq \ov{H\cap P}$. 
Now, for $\m{D}=\m{A}_{\ov{P}}(V)$,\footnote{Recall Definition~\ref{OffDef}} 
it follows from \cite[5.5]{BHS} that there are subgroups $E_1,\dots, E_r$ of $P$ containing $C_P(V)$ such that the following hold.
\begin{itemize}
\item[(i)] $\ov{P}=(\ov{E_1}\times \dots \times \ov{E_r})\ov{T}$ and $\ov{T}$ acts transitively on $\{\ov{E_1},\dots, \ov{E_r}\}$,
\item[(ii)] $\m{D}=(\m{D}\cap \ov{E_1})\cup \dots \cup (\m{D}\cap \ov{E_r})$,
\item[(iii)] $V=C_V(E_1\dots E_r)\prod_{i=1}^r [V,E_i]$, with $[V,E_i,E_j]=1$ for $j\neq i$,
\item[(iv)] $\ov{E_i}\cong SL_2(p^n)$, or $p=2$ and $\ov{E_i}\cong S_{2^n+1}$, for some $n\in\N$,
\item[(v)] $[V,E_i]/C_{[V,E_i]}(E_i)$ is a natural module for $\ov{E_i}$.
\end{itemize}
This implies together with \cite[4.6]{H}, Lemma~\ref{natSL2q}(b) and Lemma~\ref{natSnOff}(c) that
$|V/C_V(A)|=|A|$ for every $A\in \m{D}$. In particular, by the definition of $\m{D}$, $m_{\ov P}(V)= |V|$.\footnote{Recall Definition~\ref{OffDef}} Hence, we have

\bigskip

(8)\;\;\;There is no over-offender in $P$ on $V$, and $\m{D}$ is the set of minimal by inclusion elements of $\m{O}_{\ov{P}}(V)$.

\bigskip

For $B\in \m{A}(T)$, it follows from (2),(6) and a Frattini Argument that $N_{\ov{P}}(\ov{B})=\ov{N_P(BC_P(V))}=\ov{N_P(BQ)}\leq \ov{H}$.
By \cite[B.2.5]{AS} and (8), there exists $B\in\m{A}(T)$ such that $\ov{B}\in \m{D}$. Let $J$ be the full preimage of $\m{D}\cap \ov{T}$ in $T$. Observe that, by Lemma~\ref{natSL2q}(b) and Lemma~\ref{natSnOff}(b), $N_P(J)$ acts transitively on $\m{D}\cap \ov{T}$. Therefore, we get the following property. 

\bigskip

(9)\;\;\;For every $A\in \m{D}\cap \ov{T}$, there exists $B\in\m{A}(T)$ such that $A=\ov{B}$. In particular, $N_{\ov{P}}(A)\leq \ov{H}$.

\bigskip

Assume $r\neq 1$ or $E_1\cong S_{2^n+1}$ for some $n>1$. Then using Lemma~\ref{natSnOff}(a) we get
$$\ov{P}=\<N_{\ov{P}}(A): A\in \m{D}\cap \ov{T}\>\ov{T}.$$ 
Hence, (9) gives $\ov{P}\leq \ov{H}$. So, as $C_P(V)\leq C_P(Z)\leq H$, we get also $P\leq H$, contradicting the choice of $P$. Therefore, $r=1$ and for $E:=E_1$, we have 

\bigskip

(10)\;\; $P=ET$, $\ov{E}\cong SL_2(q)$ for some power $q$ of $p$, and $V/C_V(E)$ is a natural $SL_2(q)$-module for $\ov{E}$.

\bigskip

Note that $C_{\ov{P}}(\ov{E})/Z(\ov{E})\cong C_{\ov{P}}(\ov{E})\ov{E}/\ov{E}$ and so $C_{\ov{P}}(\ov{E})/Z(\ov{E})$ is a $p$-group. Moreover, as $\ov{E}\cong SL_2(q)$, $Z(\ov{E})$ has order prime to $p$. Hence, for $Y\in Syl_p(C_{\ov{P}}(\ov{E}))$, we have $C_{\ov{P}}(\ov{E})=Y\times Z(\ov{E})$ and $Y=O_p(C_{\ov{P}}(\ov{E}))\leq O_p(\ov{P})$. So by (7), 

\bigskip

(11)\;\; $C_{\ov{P}}(\ov{E})=Z(\ov{E})$. 

\bigskip

Let now $A\in \m{A}(T)$. Then by (8) there is $B\in \m{D}$ such that $B\leq \ov{A}$. By Lemma~\ref{natSL2q}(b) we have $B\in Syl_p(\ov{E})$. As $[B,\ov{A}]=1$, the structure of $Aut(\ov{E})$ yields together with (11) that $\ov{A}\leq \ov{E}C_{\ov{P}}(\ov{E})=\ov{E}$. 
Hence, $A\leq E$. Now it follows from (9),(10) and Lemma~\ref{natSL2q}(b) that 

\bigskip

(12)\;\;  $T\cap E=J(T)Q=AQ$, and $E=J(P)C_P(V)$.

\bigskip

Lemma~\ref{natSL2qOffFactor} gives the following two properties.

\bigskip

(13)\;\; $|V/C_V(A)|=|A/C_A(V)|=q$ and $C_V(A)=C_V(a)$ for every $a\in A\backslash C_G(V)$.

\bigskip

(14)\;\; $[V,A,A]=1$.

\bigskip

Set now $\widetilde{N_G(V)}:=N_G(V)/C_G(V)$ and $L:=J(G)C_G(V)$. Note that, by (12), $\ov{A}=\ov{J(T)}$ and thus also $\widetilde{A}=\widetilde{J(T)}$. Hence, $\widetilde{L}=\<\widetilde{A}^{\widetilde{L}}\>$. Moreover, $\widetilde{A}$ is weakly closed in $\widetilde{T}$ with
respect to $\widetilde{N_G(V)}$. In particular, the Frattini Argument gives $\widetilde{N_G(V)} = N_{\widetilde{N_G(V)}}(\widetilde A) \widetilde L$. By another application of the Frattini Argument and (1), (2), we get $N_{\widetilde{N_G(V)}}(\widetilde A)\leq \widetilde{N_G(V)}\cap \widetilde{N_G(J(T))}\leq\widetilde{N_H(V)}$.  Moreover, $C_{\widetilde{N_G(V)}}(C_V(\widetilde{T}))\leq C_{\widetilde{N_G(V)}}(Z)\leq \widetilde{N_H(V)}$. 
By (4), $J(G)\not\leq H$ and thus, by (13) and (14), the hypothesis of \cite[4.14]{BHS} is fulfilled with $\widetilde{N_G(V)}$, $\widetilde{N_H(V)}$ and $\widetilde{A}$ in place of $G$, $M$ and $A$.
Hence, we get $\widetilde{L}\cong SL_2(q)$ and $V/C_V(L)$ is a natural $SL_2(q)$-module for $\widetilde{L}$. 
Observe that, by (11), $C_{\ov{T}}(\ov{E})=1$ and so, by (6) and (12), $C_T(J(G)/C_{J(G)}(V))\leq C_T(\ov{E})\leq Q$. This completes the proof.
\end{proof}

\textit{The proof of Theorem~\ref{mainSL2qThm}.}
If $N_\F(J(S))\leq \m{N}$ then $A(Q)\leq \m{N}$, for every $Q\in\F$ with $J(S)=J(Q)$. Hence, the assertion follows from Lemma~\ref{ThomRestrHelp} and Theorem~\ref{ThomRestr}.

\section{Properties of Thompson-restricted subgroups}\label{ThRestrProp}

In the next section we will prove Theorem~\ref{CasesL} and Theorem~\ref{Classify}. Crucial are the properties of Thompson-restricted subgroups which we will state in this section. Throughout this section we assume the following hypothesis.

\begin{hyp}
Let $\F$ be a saturated fusion system on a finite $p$-group $S$ and let $Q\in\F$ be a Thompson-restricted subgroup. Set $T:=N_S(Q)$, $q:=|J(T)Q/Q|$ and $A(P):=Aut_\F(P)$, for every $P\in \F$.
\end{hyp}

Recall from Notation \ref{SP} that $R_P:=Aut_R(P):=\{{c_g}_{|P,P}:g\in N_R(P)\}$ for all  $P\leq R\leq S$. Furthermore, recall from Notation \ref{ModNot} that, for every fully normalized subgroup $P\in\F$, $G(P)$ denotes a model for $N_\F(P)$, provided $N_\F(P)$ is constrained.

\begin{notation}\label{AcDef}
For every $U\in \F$ set
$$V(U):=\Omega(Z(U)).$$
Moreover, we set
$$\Ac(Q):=\<(J(T)_Q)^{A(Q)}\>C_{A(Q)}(V(Q)).$$
\end{notation}

\begin{remark}\label{(IV)}
\begin{itemize}
\item[(a)] We have $J(T)_QInn(Q)\in Syl_p(\Ac(Q))$ and $J(T)Q=AQ$, for every $A\in \m{A}(T)$ with $A\not\leq Q$. 
\item[(b)] $C_{V(Q)}(J(T))=C_{V(Q)}(A)$ and $[V(Q),J(T),J(T)]=1$.
\item[(c)] Let $V\leq V(Q)$ such that $[V,\Ac(Q)]\neq 1$ and $V$ is $\Ac(Q)$-invariant. Then $V(Q)=VC_{V(Q)}(\Ac(Q))$, $C_S(V)=Q$, $|V/C_V(A)|=|A/C_A(V)|=q$, $\m{A}(Q)\subseteq \m{A}(T)$ and $(A\cap Q)V\in \m{A}(Q)$ for every $A\in\m{A}(T)$.
\item[(d)] $C_T(J(T)Q/Q)=J(T)Q$
\end{itemize}
\end{remark}

\begin{proof}
Since $Q$ is Thompson-restricted, (a) and (b) follow from Lemma~\ref{Baum}(a) and Lemma~\ref{natSL2qOffFactor}. Property (d) is a consequence of (a) and the structure of $Aut(SL_2(q))$. Let $V\leq V(Q)$ such that $[V,\Ac(Q)]\neq 1$ and $V$ is $\Ac(Q)$-invariant. Then, as $V(Q)/C_{V(Q)}(\Ac(Q))$ is irreducible, $V(Q)=VC_{V(Q)}(\Ac(Q))$. In particular, $C_{J(T)}(V)=C_{J(T)}(V(Q))\leq C_S(V(Q))=Q$. Hence, $[C_T(V),J(T)]\leq C_{J(T)}(V)\leq Q$ and, by (d), $C_T(V)=C_{J(T)Q}(V)=Q$. This means $N_{C_S(V)}(Q)=Q$ and so, as $C_S(V)$ is nilpotent, $C_S(V)=Q$. Now (b) follows from Lemma~\ref{Baum} and Lemma~\ref{natSL2qOffFactor}(a).
\end{proof}

Recall the Definition of the Baumann subgroup from Definition~\ref{BaumDef}.

\begin{lemma}\label{BTleq}
$B(T)\leq J(T)Q$.
\end{lemma}

\begin{proof}
Since $Q$ is Thompson-restricted, it follows from Remark~\ref{(IV)}(a), \ref{natSL2q}(a) and \ref{AutSL2qModule} that $C_T([V(Q),J(T)])\leq J(T)Q$. By Remark~\ref{(IV)}(c), we have $V(Q)\leq J(T)$. So, by Remark~\ref{(IV)}(b), $[V(Q),J(T)]\leq \Omega(Z(J(T)))$. Hence, $B(T)\leq C_T([V(Q),J(T)])\leq J(T)Q$.
\end{proof}

\begin{definition}\label{chDef}
We say that $U\in\F$ is \textbf{$\F$-characteristic} in $Q$ and write
$$U\;\ch\;Q$$
if $U\leq Q$, $U\unlhd T$ and $\Ac(Q)=C_{\Ac(Q)}(V(Q))N_{\Ac(Q)}(U)$.
\end{definition}

\begin{lemma}\label{BTX}
Set $G:=G(Q)$ and $M:=J(G)C_G(V(Q))$. Let $U\ch Q$ and set $X:= B(N_M(U))$. Then we have $B(T)\in Syl_p(X)$ and $M=C_G(V(Q))X$.
\end{lemma}

\begin{proof}
Observe that $T$ normalizes $N_M(U)$ thus also $X$. As $\Ac(Q)=C_{\Ac(Q)}(V(Q))N_{\Ac(Q)}(U)$, we have $M=C_M(V(Q))N_M(U)$. Therefore, Hypothesis \ref{BaumHyp} is fulfilled with $N_M(U)$ and $V(Q)$ in place of $G$ and $V$. Hence, Lemma~\ref{BaumCor} and Lemma~\ref{BTleq} imply $B(T)\in Syl_p(X)$ and $N_M(U)=(N_M(U)\cap C_M(V(Q)))X$. This implies the assertion.
\end{proof}

\begin{lemma}\label{BTH}
Set $G:=G(Q)$ and $M:=J(G)C_G(V(Q))$. Let $U\ch Q$. Then there is $H\leq N_M(U)$ such that $B(T)\in Syl_p(H)$, $H$ is normalized by $T$, $M=C_G(V(Q))H$ and, for $\hat{H}:=H/O_p(H)$,  
$$\hat{H}/\Phi(\hat{H})\cong L_2(q).$$
\end{lemma}

\begin{proof}
Set $X:=B(N_M(U))$. By Lemma~\ref{BTX}, we have $T_1:=B(T)\in Syl_p(X)$ and $M=C_G(V(Q))X$. Note that $X$ is normalized by $T$. Set $X_0:=XT$ and let $H_0\leq X_0$ be minimal such that $T\leq H_0$ and $H_0\not\leq N_{X_0}(T_1)C_{X_0}(V(Q))$. 
Set $H:=H_0\cap X$. Then $H\not\leq N_X(T_1)C_X(V(Q))$, as $H_0= HT$. Since $X/C_X(V(Q))\cong SL_2(q)$ is generated by two Sylow $p$-subgroups, we get $X=C_X(V(Q))H$ and $M=C_G(V(Q))H$. Moreover, $T_1\in Syl_p(H)$ and $H$ is normal in $H_0$. Thus, it remains to show that $\hat{H}/\Phi(\hat{H})\cong L_2(q)$.

\bigskip

Observe that $Q=O_p(H_0)$. Set $\ov{H_0}=H_0/Q$ and $C:=C_{H_0}(V(Q))$. By Remark~\ref{minparg}, $H_0$ is minimal parabolic, and so $\ov{H_0}$ is minimal parabolic as well. As $H_0/C$ is not a $p$-group, it follows now from Lemma~\ref{minpar}(b) that $\ov{C}\leq \Phi(\ov{H_0})$. 
Observe that $\ov{H_0}=\ov{TH}$, so by Lemma~\ref{PhiN}, $\Phi(\ov{H_0})=\Phi(\ov{H})$. Hence, $\ov{C}\leq \Phi(\ov{H})$, so by Remark~\ref{Frat0}(c), $\Phi(\ov{H}/\ov{C})=\Phi(\ov{H})/\ov{C}$ and, as $\ov{H}/\ov{C}\cong SL_2(q)$, then $$\ov{H}/\Phi(\ov{H})\cong (\ov{H}/\ov{C})/\Phi(\ov{H}/\ov{C})\cong L_2(q).$$
As $\ov{H}\cong H/(H\cap Q)=H/O_p(H)=\hat{H}$, this implies the assertion.
\end{proof}

\begin{lemma}\label{HelpInQ}
Let $U\ch Q$ such that $U$ is fully normalized. Let $U^*\leq Q$ be invariant under $N_{\Ac(Q)}(U)$ and $N_{A(S)}(U)$. Set $\m{N}:=N_{N_\F(U)}(U^*)$, $H:=N_{\Ac(Q)}(U)$ and $X:=HS_Q$. Suppose $O^p(H)\not\leq C_{X}(V(Q))C_{X}(Q/U^*)$. Then $O_p(\m{N}/U^*)\leq Q/U^*$.
\end{lemma}

\begin{proof}
Observe first that $N_\F(U)$ is saturated as $U$ is fully normalized. Moreover, $U^*\unlhd N_S(U)$ since $U^*$ is $N_{A(S)}(U)$-invariant, so $U^*$ is fully normalized in $N_\F(U)$ and $\m{N}$ is saturated. 
Set $$\ov{X}:=X/C_X(V(Q)).$$
Since $U\ch Q$ we have $X\leq N_{A(Q)}(U)$. In particular, as $C_{A(Q)}(V(Q))\leq \Ac(Q)$, we have $C_X(V(Q))\leq \Ac(Q)\cap X\leq H$. Since $U\ch Q$ and $Q$ is Thompson-restricted, we have
$$\ov{H}\cong \Ac(Q)/C_{\Ac(Q)}(V(Q))\cong SL_2(q).$$
Moreover, $C_{S_Q}(\ov{H})\leq Inn(Q)$, so $Z(\ov{X})=Z(\ov{H})$ and $\ov{X}/Z(\ov{X})$ embeds into $Aut(\ov{H})\cong \Gamma L_2(q)$. This gives the following property.

\bigskip

(1)\;\; Let $N$ be a normal subgroup of $X$ containing $C_X(V(Q))$ such that $O^p(H)\not\leq N$. Then $N\leq H$ and $\ov{N}\leq Z(\ov{H})$. In particular, $|N/C_N(V(Q))|\leq 2$ and $N/(N\cap Inn(Q))$ has order prime to $p$.

\bigskip 

Set $C:=C_X(Q/U^*)$ and $C_1:=CC_X(V(Q))$. By assumption, $O^p(H)\not\leq C_1$. Hence, by (1),

\bigskip

(2)\;\; $\ov{C}=\ov{C_1}\leq Z(\ov{H})$ and $C_S(Q/U^*)\leq Q$.

\bigskip

Set $\m{N}^+=\m{N}/U^*$, $R^+=RU^*/U^*$ for every subgroup $R$ of $N_S(U)$, and $L^+$ for the subgroup of $Aut_{\m{N}^+}(Q^+)$ induced by $L$, for every subgroup $L$ of $X$. Then $L^+\cong LC/C$ for every $L\leq X$. 
Observe that $Q^+$ is fully normalized in $\m{N}^+$ since $Q$ is fully normalized in $\F$. As $\m{N}^+$ is saturated, it follows in particular that $Q^+$ is fully centralized in $\m{N}^+$. Now (2) yields

\bigskip

(3)\;\; $Q^+$ is centric in $\m{N}^+$.

\bigskip

As already observed above, $\ov{X}/Z(\ov{X})$ embeds into $Aut(\ov{H})\cong \Gamma L_2(q)$. Hence, there is a subgroup $R$ of $T$ such that $Q\leq R$, $\ov{R_Q}$ is a complement of $\ov{J(T)_Q}$ in $\ov{T_Q}$, and $[\ov{R_Q},\ov{E}]=1$ for some subgroup $E$ of $H$ with $\ov{E}\cong SL_2(q_0)$ where $q_0\neq 1$ is a divisor of $q$. ($\ov{R_Q}$ corresponds to a group of field automorphisms of $SL_2(q)$.) We may choose $E$ such that $C_H(V(Q))\leq E$. Note that $C_H(V(Q))/Inn(Q)$ is a $p^\prime$-group and so $R_Q\in Syl_p(R_QC_H(V(Q)))$. As $E$ normalizes $R_QC_H(V(Q))$, it follows from a Frattini Argument that $E=E_0C_H(V(Q))$ for $E_0=N_E(R_Q)$. In particular, $\ov{E_0}\cong\ov{E}\cong SL_2(q_0)$.

\bigskip

Set $\m{E}:=N_{\m{N}}(J(Q))$. Observe that $J(Q)$ is fully normalized in $\m{N}$, as $J(Q)$ is fully normalized in $\F$, and so $\m{E}$ and $\m{E}^+:=\m{E}/U^*$ are saturated. Moreover, $Aut_{\m{E}}(Q)=Aut_{\m{N}}(Q)$ and so $E_1:=E_0^+\leq Aut_{\m{E}^+}(Q^+)$. Note that $E_1\cong E_0C/C$. As $\ov{E_0}\cong \ov{E}\cong SL_2(q_0)$, property (2) implies 
$$ \ov{E_0C_1}/\ov{C_1}\cong SL_2(q_0)\;\mbox{or}\;L_2(q_0).$$
Since $C\leq C_1$, we have $(E_0C_1)/C_1\cong (E_0C)/((E_0C)\cap C_1)$, and $E_1$ has a factor group isomorphic to $L_2(q_0)$. In particlar, $E_1$ is not  $p$-closed. Also observe that $E_1$ normalizes $(R^+)_{Q^+}=(R_Q)^+$, and $Q^+$ is fully normalized in $\m{E}^+$, as $Q^+$ is normal in $T^+$. Hence, it follows from (2) and Remark~\ref{SatAut} that every element of $E_1$ extends to an element of $Aut_{\m{E}^+}(R^+)$. Thus, $Aut_{\m{E}^+}(R^+)$ is not $p$-closed. Hence, by \ref{AlpGold1}, there is $P\in\m{E}$ such that $U^*\leq P$, $P^+$ is essential in $\m{E}^+$, and $(R^+)\phi\leq P^+$ for some element $\phi\in Aut_{\m{E}^+}(T^+)$. Then $R^+\leq (P^+)\phi^{-1}$ and so, replacing $P$ by the preimage of $(P^+)\phi^{-1}$ in $T$, we may assume that $R\leq P$.

\bigskip

By the choice of $R$, we have $Q\leq R\leq P$ and $T=J(T)R=J(T)P$. By Remark~\ref{(IV)}(a),(c), we have $\m{A}(Q)\subseteq \m{A}(T)$ and $J(T)Q=AQ$ for every $A\in \m{A}(T)\backslash \m{A}(Q)$. 
Hence, if there exists $A\in\m{A}(P)\backslash \m{A}(Q)$ then $J(T)\leq AQ\leq P$ and $P=T$, a contradiction. Thus, $J(P)=J(Q)$. 
In particular, $Aut_{\m{E}}(P)=Aut_{\m{N}}(P)$, i.e. $Aut_{\m{E}^+}(P^+)=Aut_{\m{N}^+}(P^+)$ and $Aut_{\m{N}^+}(P^+)/Inn(P^+)$ has a strongly $p$-embedded subgroup. As $Q^+\leq R^+\leq P^+$ it follows from (3) that $P^+$ is centric in $\m{N}^+$. Therefore, $P^+$ is essential in $\m{N}^+$ and by Lemma~\ref{NormalEssential}, $O_p(\m{N}^+)\leq P^+$.
In particular,

\bigskip

(4)\;\;$O_p(\m{N}^+)\leq T^+$.

\bigskip

Let $U^*\leq Y\leq N_S(U)$ such that $Y^+=O_p(\m{N}^+)$. Then by (4), $Y\leq T$. Moreover, every element of $X^+$ extends to an $\m{N}^+$-automorphism of $(YQ)^+$, so by Remark~\ref{SatAut}(a), $((YQ)_Q)^+=((YQ)^+)_{Q^+}$ is normal in $X^+$. Hence, $(YQ)_QC$ and thus $Y_QC_1$ is normal in $X$. By assumption, $O^p(H)\not\leq C_1$ and so $O^p(H)\not\leq Y_QC_1$. Therefore, by (1), $Y_QC_1/Inn(Q)$ is a $p^\prime$-group. Hence, $Y_Q\leq Inn(Q)$ and so $Y\leq Q$. This proves the assertion.
\end{proof}

Applying Lemma~\ref{HelpInQ} with $U^*=1$ we obtain the following corollary.

\begin{corollary}\label{InQ}
Let $1\neq U\ch Q$ such that $U$ is fully normalized. Then $O_p(N_\F(U))\leq Q$.
\end{corollary}

\begin{notation}\label{FminNot}
Let $1\neq U\leq Q$ such that $U\unlhd T$. Then we set
$$\m{D}(Q,U)=\{U_0\;:\;U_0\leq Q,\;U_0\;\mbox{is invariant under}\;N_{A(S)}(U)\;\mbox{and}\;N_{A(Q)}(U)\}.$$
By $U^*(Q)$ we denote the element of $\m{D}(Q,U)$ which is maximal with respect to inclusion.
\end{notation}

Note that here $U^*(Q)$ is well defined since $U\in \m{D}(Q,U)$ and the product of two elements of $\m{D}(Q,U)$ is contained in $\m{D}(Q,U)$. Moreover, if $O_p(N_\F(U))\leq Q$, then $O_p(N_\F(U))\in\m{D}(Q,U)$ and therefore $U\leq O_p(N_\F(U))\leq U^*(Q)$.

\begin{lemma}\label{Fmin1}
Let $\F$ be minimal, let $1\neq U\ch Q$ and assume $U$ is fully normalized. Then $O^p(N_{\Ac(Q)}(U))\leq C_{A(Q)}(Q/U^*(Q))C_{A(Q)}(V(Q))$.
\end{lemma}

\begin{proof}
Set $H:=N_{\Ac(Q)}(U)$, $X:=HS_Q$, $U^*=U^*(Q)$, and  $C:=C_X(Q/U^*)C_X(V(Q))$. Assume $O^p(H)\not\leq C$. Observe that $N_\F(U)$ is saturated and solvable, since $U$ is fully normalized and $\F$ is minimal. Moreover, $U^*\unlhd N_S(U)$ is fully normalized in $N_\F(U)$ and so, by Proposition~\ref{SubSolv}(a), $\m{N}:=N_{N_\F(U)}(U^*)$ is saturated and solvable. Therefore, $O_p(\m{N}/U^*)\neq 1$ and so $U^*< U_0$, where $U_0$ is the full preimage of $O_p(\m{N}/U^*)$ in $N_S(U)$. By Lemma~\ref{HelpInQ}, $U_0\leq Q$. Now $U_0\in \m{D}(Q,U)$ and so $U_0=U^*(Q)=U$, a contradiction.
\end{proof}

\begin{lemma}\label{XU*Char}
Let $\F$ be minimal and let $1\neq U\ch Q$ such that $U$ is fully normalized. Then $U^*(Q)X\ch Q$ for every subgroup $X$ of $Q$ with $X\unlhd T$.
\end{lemma}

\begin{proof}
Note that $U_0:=U^*(Q)X$ is normal in $T$. Moreover, by Lemma~\ref{Fmin1}, we have 
$$\begin{array}{lll}
N_{\Ac(Q)}(U) &\leq T_QO^p(N_{\Ac(Q)}(U))\leq
T_QC_{A(Q)}(V(Q))C_{\Ac(Q)}(Q/U^*(Q))\\
&\leq C_{A(Q)}(V(Q))N_{A(Q)}(U_0).\end{array}$$ 
Hence, $\Ac(Q)=C_{A(Q)}(V(Q))N_{\Ac(Q)}(U)=C_{A(Q)}(V(Q))N_{\Ac(Q)}(U_0)$ and $U_0\ch Q$.
\end{proof}

\begin{lemma}\label{ThRestrConj1}
Let $\phi\in Mor_\F(N_S(Q),S)$. Then $Q\phi$ is Thompson-restricted.
\end{lemma}

\begin{proof}
As $Q$ is centric, $\widetilde{Q}:=Q\phi$ is centric. Observe that 
$$ |N_S(J(Q))|=|N_S(Q)|=|N_S(Q)\phi|\leq |N_S(\widetilde{Q})|\leq |N_S(J(\widetilde{Q}))|.$$
So, as $Q$ and $J(Q)$ are fully normalized, $\widetilde{Q}$ and $J(\widetilde{Q})=J(Q)\phi$ are fully normalized, and $N_S(Q)\phi=N_S(\widetilde{Q})=N_S(J(\widetilde{Q}))$. Observe that $\phi:N_S(Q)\rightarrow N_S(\widetilde{Q})$ is an isomorphism of fusion systems from $N_\F(Q)$ to $N_\F(\widetilde{Q})$. Moreover, for $V\leq \Omega(Z(\widetilde{Q}))$, we have $C_S(V)=\widetilde{Q}$ if and only if $C_{N_S(\widetilde{Q})}(V)=\widetilde{Q}$. So $\widetilde{Q}$ is Thompson-restricted as $Q$ is Thompson-restricted.
\end{proof}

\section{Pushing up in fusion systems}\label{PUFs}

Throughout this section, assume the following hypothesis.

\begin{hyp}\label{HypMaxParUnique0}
Let $\F$ be a saturated fusion system on a finite $p$-group $S$. Suppose $\F$ is minimal. Let $\m{N}$ be a proper saturated subsystem of $\F$ on $S$, and let $\Q$ be the set of all Thompson-maximal members of $\FN$ which are Thompson-restricted. 
\end{hyp}

Recall here the definition of Thompson-maximality and Thompson-restricted subgroups from Definition~\ref{ThMax} and Definition~\ref{ThRestrictDef} in the introduction.
Furthermore, recall from Notation~\ref{ENDef} that $\FN$ is the set of subgroups $P\in \F$ with $Aut_\F(P)\not\leq \m{N}$.
The aim of this section is to prove Theorem~\ref{CasesL}, which then, together with Theorem~\ref{mainSL2qThm} and Theorem~\ref{ClassifyG}, implies Theorem~\ref{Classify}. We restate Theorem~\ref{CasesL} here for the readers convenience. Recall the Definition of a full maximal parabolic from Definition~\ref{ParDef}.

\begin{hyp}\label{HypMaxParUnique}
Assume Hypothesis \ref{HypMaxParUnique0} and suppose $\m{N}$ contains every full maximal parabolic of $\F$.
\end{hyp}

\setcounter{mtheorem}{\value{repeat}}
\begin{mtheorem}
Assume Hypothesis~\ref{HypMaxParUnique}. 
Let $Q\in \Q$, $G:=G(Q)$ and $M:=J(G)$.\footnote{Recall Notation~\ref{ModNot}.} Then $N_S(X)=N_S(Q)$, for every non-trivial normal $p$-subgroup $X$ of $MN_S(Q)$. Moreover, $Q\leq M$, $M/Q\cong SL_2(q)$ and one of the following holds:
\begin{itemize}
\item[(I)] $Q$ is elementary abelian, and $Q/C_Q(M)$ is a natural $SL_2(q)$-module for $M/Q$, or
\item[(II)] $p=3$, $S=N_S(Q)$ and $|Q|= q^5$. Moreover, $Q/Z(Q)$ and $Z(Q)/\Phi(Q)$ are natural $SL_2(q)$-modules for $M/Q$, and $\Phi(Q)=C_Q(M)$ has order $q$.
\end{itemize}
\end{mtheorem}

Note here that Theorem~\ref{mainSL2qThm} yields  $\Q\neq\emptyset$ if Hypothesis~\ref{HypMaxParUnique} holds. In fact, this is already the case if we assume the following more general hypothesis. 

\begin{hyp}\label{HypMaxParUnique1}
Assume Hypothesis \ref{HypMaxParUnique0}, and suppose $N_\F(C)\leq\m{N}$ for every characteristic subgroup $C$ of $S$.
\end{hyp}

Many arguments in the proof of Theorem~\ref{CasesL} require only Hypothesis~\ref{HypMaxParUnique1}. More precisely, we will be able to prove the following Lemma.

\begin{lemma}\label{PUThm} 
Assume Hypothesis~\ref{HypMaxParUnique1}. Let $Q\in\Q$ and $1\neq U\ch Q$. Then $B(N_S(U))=B(N_S(Q))$.
\end{lemma}

Here for a Thompson-restricted subgroup $Q$ of $\F$ recall the definition of $\Ac(Q)$ and of $\F$-characteristic subgroups from Notation~\ref{AcDef} and Definition~\ref{chDef}. For a finite group $H$, recall the Definition of the Baumann subgroup $B(H)$ from Definition~\ref{BaumDef}.

\bigskip

Lemma~\ref{PUThm} is a major step in the proof of Theorem~\ref{CasesL} because, together with Lemma~\ref{BTH}, it enables us to apply the pushing up result by Baumann and Niles in the form stated in Theorem~\ref{PUGrCor}.

\bigskip

In the remainder of this section we use the following notation: For $P\in\F$ set 
$$A(P):=Aut_\F(P)\mbox{ and }V(P):=\Omega(Z(P)).$$ 
Recall from Notation~\ref{SP} that, for subgroups $P$ and $R$ of $S$,
$$R_P:=Aut_R(P):=\{{c_g}_{|P,P}:g\in N_R(P)\}.$$

\subsection{Preliminaries}\label{PuFs1}

Throughout Subsection~\ref{PuFs1} assume Hypothesis~\ref{HypMaxParUnique1}.

\begin{lemma}\label{A(T)}
Let $Q\in\Q$. Then $A(YQ)\leq \m{N}$ and $N_{A(Q)}(Y_Q)\leq \m{N}$, for every subgroup $Y$ of $T$ with $J(QY)\not\leq Q$.
\end{lemma}

\begin{proof}
Set $X:=YQ$. By Corollary~\ref{Cor3}, we have $A(X)\leq \m{N}$. Since $Q$ is fully normalized and $C_S(Q)\leq Q\leq X$, Remark~\ref{SatAut}(b) implies that every element of $N_{A(Q)}(X_Q)$ extends to an element of $A(X)$. As $N_{A(Q)}(Y_Q)\leq N_{A(Q)}(X_Q)$, this shows the assertion.
\end{proof}

\begin{remark}\label{RemAc}
Let $Q\in\Q$. Then $\Ac(Q)\not\leq \m{N}$.
\end{remark}

\begin{proof}
Otherwise, by the Frattini Argument and \ref{A(T)}, $A(Q)=\Ac(Q)N_{A(Q)}(J(T)_Q)\leq \m N$, contradicting $Q\in\FN$.
\end{proof}

\begin{lemma}\label{trivial0}
Let $Q\in\Q$, let $U\in\F$ be $\F$-characteristic in $Q$ and characteristic in $S$. Then $U=1$.
\end{lemma}

\begin{proof}
Assume $U\neq 1$. As $U$ is characteristic in $S$,  Hypothesis \ref{HypMaxParUnique1} implies  $N_{\Ac(Q)}(U)\leq N_\F(U)\leq \m N$ and  $C_{\Ac(Q)}(V(Q))\leq N_\F(\Omega(Z(S)))\leq \m N$. Hence, as $U\ch Q$, we have $\Ac(Q)\leq \m N$. This is a contradiction to Remark~\ref{RemAc}.
\end{proof}

For $Q\in\Q$ and a subgroup $U$ of $Q$ with $1\neq U\unlhd T$ define $U^*(Q)$ as in Notation~\ref{FminNot}.

\begin{notation}\label{CQDef}
Let $Q\in \Q$.
\begin{itemize}
\item Set 
$$\m{C}(Q)=\{U\leq Q\;:\; U\ch Q,\;C_S(V(U))=U=U^*(Q)\}.$$
\item We define $\m{C}^*(Q)$ to be the set of all $1\neq U\ch Q$ such that $U$ is fully normalized and
$$|U|=max\{|U^*|:U^*\ch Q,\;U^*\unlhd N_S(U)\}.$$
\end{itemize}
\end{notation}

Let $Q\in\Q$. Observe that, by the definition of $U^*(Q)$, we have $U^*(Q)\ch Q$ and $U^*(Q)\unlhd N_S(U)$, for every $U\ch Q$. Hence, for every $U\in \m{C}^*(Q)$, $U=U^*(Q)$. Also note $V(Q)\leq C_S(V(U))=U$, for every $U\in \m{C}(Q)$. This implies the following remark.

\begin{remark}\label{V(Q)Rem}
Let $Q\in\Q$ and $U\in \m{C}(Q)$. Then $V(Q)\leq V(U)$.
\end{remark}

\begin{lemma}\label{XleqU}
Let $Q\in \Q$, $U\in \m{C}^*(Q)$ and $X\leq Q$ such that $X\unlhd N_S(U)$. Then $X\leq U$.
\end{lemma}

\begin{proof}
Since $U\ch Q$ and $U=U^*(Q)$, it follows from Lemma~\ref{XU*Char} that $UX\ch Q$. Moreover, $UX\unlhd N_S(U)$. Hence, the maximality of $|U|$ yields $X\leq U$.
\end{proof}

\begin{lemma}\label{Xt}
Let $Q\in \Q$, $U\in \m{C}^*(Q)$ and $X\leq Q$ such that $X\not\leq U$. Then there is $t\in N_S(U)$ such that $X^t\not\leq Q$.
\end{lemma}

\begin{proof}
Otherwise $\<X^{N_S(U)}\>\leq Q$, a contradiction to Lemma~\ref{XleqU} and $X\not\leq U$.
\end{proof}

\begin{lemma}\label{C*C}
Let $Q\in\Q$. Then $\m{C}^*(Q)\subseteq \m{C}(Q)$.
\end{lemma}

\begin{proof}
Set $T:=N_S(Q)$ and let $U\in \m{C}^*(Q)$. As already remarked above, $U=U^*(Q)$. By Lemma~\ref{XleqU}, we have $Z:=\Omega(Z(S))\leq U$ and so $Z\leq V(U)$. Hence, as $U\ch Q$ and $Z\leq V(Q)$, we have $V:=\<Z^{\Ac(Q)}\>=\<Z^{N_{\Ac(Q)}(U)}\>\leq V(U)$. Lemma~\ref{trivial0} implies $Z\not\leq C(\Ac(Q))$. 
Thus, $[V,\Ac(Q)]\neq 1$ and, by Remark~\ref{(IV)}(c), we have $C_S(V)=Q$. Therefore, $C_S(V(U))\leq Q$ and so 
$C_S(V(U))=C_Q(V(U))\ch Q$. At the same time, $C_S(V(U))\unlhd N_S(U)$. Hence, the maximality of $|U|$ yields $U=C_S(V(U))$ and thus $U\in\m{C}(Q)$.
\end{proof}

\begin{notation}\label{Astar}
Let $W\leq S$ be elementary abelian and $W\leq Y\leq N_S(W)$. Then we write $\m{A}_*(Y,W)$ for the set of elements $A\in\m{A}(Y)$ with $[A,W]\neq 1$ for which $AC_Y(W)$ is minimal with respect to inclusion among the groups $BC_Y(W)$ with $B\in\m{A}(Y)$ and $[W,B]\neq 1$. (In particular, $\m{A}_*(Y,W)=\emptyset$ if $[W,J(Y)]=1$.)
\end{notation}

\begin{lemma}\label{minBestOff} 
Let $Q\in \Q$, $U\in \m{C}(Q)$, $W:=V(U)$ and $A\in \m{A}_*(T,W)$.
Assume $A\not\leq Q$. Then 
$$|W/C_W(A)|=|A/C_A(W)|=q\;\mbox{and}\;W=V(Q)C_W(A).$$
\end{lemma}

\begin{proof}
By Lemma~\ref{V(Q)Rem}, $V(Q)\leq W$. As $C_S(V(Q))=Q$, we have $[V(Q),A]\neq 1$. Remark~\ref{(IV)}(c) implies
$$|V(Q)/C_{V(Q)}(A)|=|A/C_A(V(Q))|=q.$$
Hence, the assertion follows from Lemma~\ref{1}.
\end{proof}

\begin{notation}
For $Q\in\Q$ set $$R(Q)=[V(Q),J(N_S(Q))].$$
\end{notation}

\begin{remark}\label{RCent}
Let $Q\in\Q$ and set $T:=N_S(Q)$. Then $[R(Q),J(T)Q]=1$ and $R(Q)=[V(Q),A]$, for every $A\in \m{A}(T)$ with $A\not\leq Q$.
\end{remark}

\begin{proof}
This follows from Remark~\ref{(IV)}(a),(b).
\end{proof}

\begin{lemma}\label{R=}
Let $Q\in\Q$, $U\in \m{C}(Q)$ and $A\in\m{A}_*(T,V(U))$ such that $A\not\leq Q$. Then $R(Q)=[V(U),A]$.
\end{lemma}

\begin{proof}
By Remark~\ref{RCent}, $R(Q)=[V(Q),A]$, and by Lemma~\ref{minBestOff}, $V(U)=V(Q)C_{V(U)}(A)$. This implies the assertion.
\end{proof}

\begin{lemma}\label{ThRestrConj2a}
Let $Q\in\Q$ and $\phi\in Mor_\F(N_S(Q),S)$. Then $Q\phi\in \Q$, $N_S(Q)\phi=N_S(Q\phi)$, $\Ac(Q)\phi^*=\Ac(Q\phi)$,\footnote{Recall Notation~\ref{phi*}} $V(Q)\phi=V(Q\phi)$ and $R(Q)\phi=R(Q\phi)$. Moreover, for every $U\ch Q$, we have $U\phi\ch Q\phi$.
\end{lemma}

\begin{proof}
By Lemma~\ref{ThRestrConj1}, $Q\phi$ is Thompson-restricted. As $J(N_S(Q))\not\leq Q$ and $Q$ is Thompson-maximal in $\FN$, it follows from Corollary~\ref{Cor2} that $\phi$ is a morphism in $\m{N}$. Hence, $A(Q\phi)=A(Q)\phi^*\not\leq \m{N}$ as $A(Q)\not\leq \m{N}$. Thus, $Q\phi\in \FN$ and Thompson-maximal in $\FN$, since $Q$ is Thompson-maximal in $\FN$. Hence, $Q\phi\in \Q$. Now the assertion is easy to check as the map $\phi^*:A(Q)\rightarrow A(Q\phi)$ is an isomorphism of groups with $J(N_S(Q))_Q\phi^*=J(N_S(Q\phi))_{Q\phi}$. 
\end{proof}

\begin{corollary}\label{ThRestrConj2}
Let $Q\in\Q$ and $1\neq U\ch Q$. Then there is $\phi\in Mor_\F(N_S(U),S)$ such that $U\phi$ is fully normalized. For each such $\phi$ we have $Q\phi\in \Q$, $U\phi\ch Q\phi$, $N_S(Q)\phi=N_S(Q\phi)$, $\Ac(Q)\phi^*=\Ac(Q\phi)$, $V(Q)\phi=V(Q\phi)$ and $R(Q)\phi=R(Q\phi)$.
\end{corollary}

\begin{proof}
By Lemma~\ref{fullNormChar}, there is $\phi\in Mor_\F(N_S(U),S)$ such that $U\phi$ is fully normalized. As $U\ch Q$, $U\unlhd T=N_S(Q)$. Hence, $\phi_{|N_S(Q)}\in Mor_\F(N_S(Q),S)$. Now the result follows from Lemma~\ref{ThRestrConj2a}.
\end{proof}

\subsection{The proof of Lemma~\ref{PUThm}} \label{PUThmP}

Throughout Subsection~\ref{PUThmP} assume Hypothesis~\ref{HypMaxParUnique1}.

\begin{lemma}\label{PUFsl}
Let $Q\in \Q$, let $U\in \m{C}(Q)$ be fully normalized, and let $R_0\leq R(Q)$ such that $[R_0,\Ac(Q)]\neq 1$ and 
$$N_S(U)\cap N_S(R_0)\cap N_S(\<\m{A}_*(Q,V(U))\>)\leq N_S(Q).$$ 
Then $N_S(U)\cap N_S(R_0)\leq N_S(Q)$.
\end{lemma}

\begin{proof}
Set $W:=V(U)$, $T:=N_S(Q)$, $T_0:=N_T(R_0)$, $R:=R(Q)$ and $\m{A}_*(Y):=\m{A}_*(Y,W)$ for $Y\leq T$. Assume the assertion is wrong. Then $T_0<N_S(U)\cap N_S(R_0)$. In particular, $T<N_S(U)$ and so $J(Q)\not\leq U$ since $N_S(J(Q))=T$. Hence, $\m{A}_*(Q)\neq\emptyset$. Moreover, $T_0 < N_S(U)\cap N_S(R_0)\cap N_S(T_0)$, i.e. there is $t\in N_S(U)\cap N_S(R_0)\cap N_S(T_0)$ such that $t\not\in T$. Then, by assumption, there is $A\in \m{A}_*(Q)$ such that $A^t\not\in \m{A}_*(Q)$. Note that $A^t\leq T_0\leq T$. Remark~\ref{(IV)}(c) implies $\m{A}(Q)\subseteq \m{A}(T)$. Hence, as $C_S(W)=U\leq Q$, we get 
$$\m{A}_*(Q)\subseteq \m{A}_*(T).$$
Therefore, $A^t\in\m{A}_*(T)$ and $A^t\not\leq Q$.
Now Lemma~\ref{R=} yields $R=[W,A^t]=[W,A]^t $. Hence, $R^{t^{-1}}= [W,A]=[W,AU]$. So, by Lemma~\ref{Fmin1}, $R^{t^{-1}}\cap V(Q)=[W,AU]\cap V(Q)$ is $O^p(N_{\Ac(Q)}(U))$-invariant. As $U\ch Q$, this implies $R^{t^{-1}}\cap V(Q)\leq C_{V(Q)}(\Ac(Q))$ or $V(Q)=(R^{t^{-1}}\cap V(Q))C_{V(Q)}(\Ac(Q))$. By Remark \ref{(IV)}(a), we have $J(T)\leq A^tQ\leq T_0$, so $J(T)=J(T_0)$ and $J(T)^t=J(T)$. Remark~\ref{RCent} implies $[R,J(T)]=1$. Therefore, we get $[R^{t^{-1}},J(T)]=1$. As $[V(Q),J(T)]\not\leq C_{V(Q)}(\Ac(Q))$, it follows now $R^{t^{-1}}\cap V(Q)\leq C_{V(Q)}(\Ac(Q))$. Hence,
$$R_0=R_0^{t^{-1}}\leq R^{t^{-1}}\cap V(Q)\leq C(\Ac(Q)),$$
a contradiction. This proves the assertion.
\end{proof}

\begin{lemma}\label{NT0R}
Let $Q\in \Q$ and $1\neq U\ch Q$. Then $N_S(R_0)\cap N_S(U)\leq N_S(Q)$ for every $R_0\leq R(Q)$ with $[R_0,\Ac(Q)]\neq 1$. In particular,
$N_S(R(Q))\cap N_S(U)=N_S(Q)$.
\end{lemma}

\begin{proof}
Assume the assertion is wrong. Choose $Q,U\in\F$ such that $Q\in\Q$, $1\neq U\ch Q$, and there exists $R_0\leq R(Q)$ with $[R_0,\Ac(Q)]\neq 1$ and $N_S(R_0)\cap N_S(U)\not\leq N_S(Q)$. We may choose this pair $(Q,U)$ such that $|U|$ is maximal.
By Corollary~\ref{ThRestrConj2}, there is $\phi\in Mor_\F(N_S(U),S)$ such that $Q\phi\in \Q$, $U\phi \ch Q\phi$ and $U\phi$ is fully normalized. Moreover, then $R_0\phi\leq R(Q)\phi=R(Q\phi)$, $[R_0\phi, \Ac(Q\phi)]\neq 1$ and $N_S(R_0\phi)\geq N_{N_S(U)}(R_0)\phi\not\leq N_S(Q)\phi=N_S(Q\phi)$. So, replacing $(Q,U)$ by $(Q\phi,U\phi)$, we may assume without loss of generality that $U$ is fully normalized. Observe that then $U\in \m{C}^*(Q)$ and thus, by Lemma~\ref{C*C}, $U\in \m{C}(Q)$.

\bigskip

Set $Q_*=\<\m{A}_*(Q,V(U))\>$ and note that $UQ_*\ch Q$, as $Q_*$ is $N_{A(Q)}(U)$-invariant and $U\ch Q$. Since $N_S(J(Q))=N_S(Q)$ and $N_S(Q)<N_S(U)$ by assumption, we have $Q^*\not\leq U$. Hence, the maximality of $|U|$ yields $N_S(R_0)\cap N_S(UQ_*)\leq N_S(Q)$ and thus $N_S(R_0)\cap N_S(U)\cap N_S(Q_*)\leq N_S(Q)$. Now Lemma~\ref{PUFsl} yields $N_S(U)\cap N_S(R_0)\leq N_S(Q)$, contradicting the choice of $U$.
\end{proof}

\begin{lemma}\label{AbleqQ}
Let $Q\in\Q$, $1\neq U\in \m{C}(Q)$, $A\in \m{A}_*(N_S(Q),V(U))$ and $b\in N_S(U)\backslash N_S(Q)$ such that $A\not\leq Q$ and $A^b\leq N_S(Q)$. Then $A^b\leq Q$.
\end{lemma}

\begin{proof}
Assume $A^b\not\leq Q$. Then Lemma~\ref{R=} implies $$R(Q)=[V(U),A^b]=[V(U),A]^b=R(Q)^b.$$ This is a contradiction to Lemma~\ref{NT0R}.
\end{proof}

\begin{lemma}\label{JT0T}
Let $Q\in\Q$ and $U\in \m{C}^*(Q)$. Then we have $\m{A}_*(N_S(Q),V(U))=\m{A}_*(Q,V(U))$ or $J(N_S(U))\leq N_S(Q)$.
\end{lemma}

\begin{proof}
\setcounter{equation}{0}
Set $T=N_S(Q)$, $T_0=N_S(U)$, $R:=R(Q)$, $W=V(U)$, and  $\m{A}_*(Y)=\m{A}_*(Y,W)$ for every $Y\leq T_0$.
We will use frequently and without reference that, by Lemma~\ref{C*C}, $U\in \m{C}(Q)$ and, in particular, by Remark~\ref{V(Q)Rem}, $V(Q)\leq W$. Assume $J(T_0)\not\leq T$ and $\m{A}_*(T)\neq\m{A}_*(Q)$. We show first: 

\bigskip

(1)\;\; $\<\m{A}_*(T_0)\>\not\leq T.$

\bigskip

By assumption, there is $B_*\in \m{A}(T_0)$ with $B_*\not\leq T$. We may choose $B_*$ such that $|B_*U|$ is minimal. Let $B\in \m{A}_*(B_*U)$. Then $B\in \m{A}_*(T_0)$. Let $t\in T_0$ and observe that $B^t\in \m{A}_*(T_0)$. Assume (1) does not hold. Then $B$ and $B^t$ are contained in  $T$.

\bigskip

Suppose $B^t\not\leq Q$. Since $B_*^tU/U$ is elementary abelian, $B^tU$ is normalized by $B_*^t$. Hence, for every $x\in B_*^t$, $(B^t)^x\leq T$ and $(B^t)^x\not\leq Q$. Hence, by Lemma~\ref{AbleqQ}, $B_*^t\leq T$. In particular, $\m{A}(T)\subseteq\m{A}(T_0)$ and, by Remark~\ref{(IV)}(a),  $B_*^t\leq J(T)\leq B^tQ$. Since $B\leq B_*U$, this gives  $B_*^tU=B^t(B_*^tU\cap Q)=B^tU(B_*^t\cap Q)$ and $B_*U=BU(B_*\cap Q^{t^{-1}})$. By Remark~\ref{(IV)}(c), we have  $C_*=(B_*^t\cap Q)V(Q)\in\m{A}(T)\subseteq \m{A}(T_0)$. Therefore, $C_*^{t^{-1}}=(B_*\cap Q^{t^{-1}})V(Q)^{t^{-1}}\in\m{A}(T_0)$. 
Note that, by Remark \ref{V(Q)Rem}, $V(Q)^{t^{-1}}\leq U^{t^{-1}}=U$. In particular, $B_*U=BU(B_*\cap Q^{t^{-1}})=BUC_*^{t^{-1}}$. As $BU\leq T$ and $B_*\not\leq T$, we get $C_*^{t^{-1}}\not\leq T$. On the other hand, $C_*^{t^-1}\leq B_*U$, so the minimality of $|B_*U|$ gives $C_*^{t^{-1}}U=B_*U$. Then $B_*\leq Q^{t^{-1}}$, i.e. $B^t\leq B_*^t\leq Q$ contradicting our assumption. Hence, $B^t\leq Q$. Since $t\in T_0$ was arbitrary we have shown that $X:=\<B^{T_0}\>\leq Q$. Therefore, it follows from Lemma~\ref{XleqU} that $B\leq X\leq U$, a contradiction to the choice of $B$. Thus, (1) holds. We show next:

\bigskip

(2)\;\; There is $T\leq T_1\leq N_{T_0}(\<\m{A}_*(T)\>)$ such that $\<\m{A}_*(T_1)\>\not\leq T.$

\bigskip

For the proof let $T\leq Y\leq T_0$ be maximal with respect to inclusion such that $\<\m{A}_*(Y)\>\leq T$. Then, by (1), $Y\neq T_0$ and hence $Y<T_1:=N_{T_0}(Y)$. So the maximality of $Y$ implies $\<\m{A}_*(T_1)\>\not\leq T$. Since $\<\m{A}_*(Y)\>=\<\m{A}_*(T)\>$, we have $T_1\leq N_{T_0}(\<\m{A}_*(T)\>)$. This shows (2).

\bigskip

So we can choose now $T_1$ with the properties as in (2). We fix $B_*\in\m{A}_*(T_1)$ such that $B_*\not\leq T$. It follows from \cite[9.2.1]{KS} and \cite[9.2.3]{KS} that there is $C\in \m{A}(B_*C_{T_1}(W))$ such that $[W,C]\neq 1$ and $[W,C,C]=1$. Hence, $B_*\in \m{A}_*(T_1)$ implies

\bigskip

(3)\;\;$B_*$ acts quadratically on $W$.

\bigskip

We show next:

\bigskip

(4)\;\; $|B_*/N_{B_*}(R)|=|B_*/B_*\cap T|=2=p.$

\bigskip

By Lemma~\ref{NT0R}, $N_{B_*}(R)=B_*\cap T$. Let $b\in B_*\backslash T$ and assume there is $c\in B_*\backslash ((B_*\cap T)\cup b(B_*\cap T))$. Note that $b,c$ and $cb^{-1}$ are not elements of $T$. By assumption, $\m{A}_*(T)\neq \m{A}_*(Q)$, i.e. there is $A\in \m{A}_*(T)$ with $A\not\leq Q$. Then by the choice of $B_*\leq T_1$ and Lemma~\ref{AbleqQ},
$$A^b\leq Q,\;A^c\leq Q\;\mbox{and}\;A^{cb^{-1}}\leq Q.$$
This gives
$$A^c\leq Q\cap Q^b\cap Q^{b^{-1}c}.$$
Hence, $A^c$ centralizes $W_0:=V(Q)V(Q)^bV(Q)^{b^{-1}c}$. Note that $W_0\leq W$ and 
$$W_1:=V(Q)[V(Q),b][V(Q),b^{-1}c]\leq W_0.$$ 
By (3), $W_1$ is invariant under $b$ and $b^{-1}c$, so $W_0=W_1$ and $W_0=W_0^b=W_0^{b^{-1}c}$. Hence, $W_0^c=(W_0^{b^{-1}})^c=W_0^{b^{-1}c}=W_0$. As shown above, $[W_0,A^c]=1$. So we get $[W_0^c,A^c]=1$ and thus $[W_0,A]=1$. 
In particular, $[V(Q),A]=1$, a contradiction to $A\not\leq Q$. Hence, (4) holds. We show now

\bigskip

(5)\;\;$W=RC_W(B_*).$

\bigskip

It follows from (4) and Remark~\ref{RCent} that $$|B_*/C_{B_*}(R)|=|B_*/B_*\cap T|\cdot |(B_*\cap T)/C_{B_*}(R)|\leq 2\cdot |(B_*\cap~T)/(B_*\cap J(T))|.$$
Since $Q$ is Thompson-restricted, we have for $\ov{A(Q)}=A(Q)/C_{A(Q)}(V(Q))$ that $\ov{\Ac(Q)}\cong SL_2(q)$, $\ov{J(T)_Q}\in Syl_p(\ov{\Ac(Q)})$ and $T/Q\cong \ov{T_Q}$ embeds into $Aut(\ov{\Ac(Q)})$. Hence, $T/J(T)Q\cong Aut(GF(q))$ is cyclic, and $q=2$ implies $T=J(T)$. Therefore, $|(B_*\cap T)/(B_*\cap J(T))|\leq 2$ and

\bigskip

(*)\;\;$|B_*/C_{B_*}(R)|\leq q.$

\bigskip

The module structure of $V(Q)$ implies $|R/C_R(\Ac(Q))|=q$. By Lemma~\ref{NT0R}, $R\cap B_*\leq C_R(B_*)\leq C_R(\Ac(Q))$ and hence
$$|R/R\cap B_*|\geq |R/C_R(B_*)|\geq |R/C_R(\Ac(Q))|=q.$$
Thus, by (*), $|RC_{B_*}(R)|=|R/R\cap B_*|\cdot |C_{B_*}(R)|\geq |B_*|$.
Observe that $RC_{B_*}(R)$ is elementary abelian, so $RC_{B_*}(R)\in\m{A}(T_1)$. Since $(RC_{B_*}(R))U=C_{B_*}(R)U$ is a proper subset of $B_*U$, it follows from the minimality of $B_*U$ that $C_{B_*}(R)\leq U$ and $C_{B_*}(R)=C_{B_*}(W)$. Therefore, by Lemma~\ref{Baum} and (*), 
$$|W/C_W(B_*)|\leq |B_*/C_{B_*}(W)|=|B_*/C_{B_*}(R)|\leq q. $$
As seen above, $|R/C_R(B_*)|\geq q$. This implies $|RC_W(B_*)|\geq |W|$ and thus (5).

\bigskip

Now choose $t\in \Ac(Q)\backslash N_{\Ac(Q)}(T_Q)C_{\Ac(Q)}(V(Q))$ and $b\in
B_*\backslash C_{B_*}(W)$. Set
$Y=RR^tR^b$. Note that $Y \leq W$, since $RR^t\leq V(Q)\leq W$. 
Using (5), we get $[W,b]=[RC_W(B_*),b]=[R,b]\leq RR^b\leq Y$. Hence,
$$Y^b=Y.$$ 
As before let $A\in \m{A}_*(T)$ with $A\not\leq Q$. Then, by the choice of $B_*\leq T_1$ and Lemma~\ref{AbleqQ}, we have $A^b\leq Q$. Hence,
$[RR^t,A^b]\leq [V(Q),A^b]=1$. By Remark~\ref{RCent}, $[R,A]=[R,J(T)]=1$, so $[R^b,A^b]=1$ and
$[Y,A^b]=1$. As we have shown above, $Y=Y^b$. So we get $[Y,A]^b =[Y^b,A^b]=[Y,A^b]=1$ and hence, $[Y,A]=1$. In particular, $[R^t,A]=1$ which is a contradiction to the module structure of $V(Q)$. This completes the proof of Lemma~\ref{JT0T}.
\end{proof}

\begin{lemma}\label{JNSUT}
Let $Q\in\Q$ and $1\neq U\ch Q$. Then $J(N_S(U))\leq N_S(Q)$.
\end{lemma}

\begin{proof}
Assume the assertion is wrong. Then there is $Q\in\Q$ and $U\ch Q$ such that $J(N_S(U))\not\leq N_S(Q)$. We can choose the pair $(Q,U)$ such that $|U|$ has maximal order. By Corollary~\ref{ThRestrConj2} we can furthermore choose it such that $U$ is fully normalized. Set $T=N_S(Q)$ and $T_0:=N_S(U)$. Note that $U\in
\m{C}^*(Q)$. Thus, by Lemma~\ref{JT0T}, $\m{A}_*(T,V(U))=\m{A}_*(Q,V(U))$. Set
$X:=\<\m{A}_*(Q,V(U))\>$. Observe that $T=N_S(J(Q))<T_0$, so $J(Q)\not\leq U$ and  $X\not\leq U$ as $U=C_S(V(U))$. Also note $U_1:=XU\ch Q$. Therefore, by
the choice of $U$, $J(N_S(U_1))\leq T \leq T_0$. 
In particular, $J(T) = J(N_S(U_1)) = J(N_{T_0}(U_1))$ and 
$$\m{A}_*(Q,V(U)) = \m{A}_*(T,V(U)) = \m{A}_*(N_{T_0}(U_1),V(U)).$$
Hence, $N_{T_0}(N_{T_0}(U_1))$ normalizes $XU = U_1$. It follows that $T_0 \leq N_S(U_1)$, which contradicts $J(N_S(U_1))\leq T$ and $J(T_0)\not\leq T$.
\end{proof}

\textit{The proof of Lemma~\ref{PUThm}.} 
Let $Q\in \Q$ and $1\neq U\ch Q$.
Set $T:=N_S(Q)$ and $T_0:=N_S(U)$. By Lemma~\ref{JNSUT}, $J(T_0)=J(T)$ and so $B(T_0)=C_{T_0}(\Omega(Z(J(T))))$. By Remark~\ref{RCent} $R(Q)\leq \Omega(Z(J(T)))$. Hence, by Lemma~\ref{NT0R}, $B(T_0)\leq N_{T_0}(R(Q))=T$. This shows $B(T_0)=B(T)$ and completes the proof.

\subsection{The proofs of Theorem~\ref{CasesL} and Theorem~\ref{Classify}}

From now on assume Hypothesis~\ref{HypMaxParUnique}. Observe that this implies Hypothesis~\ref{HypMaxParUnique1}. In particular, we can use Lemma~\ref{PUThm} and the other results from the previous subsections.

\begin{lemma}\label{trivial}
Let $Q\in\Q$, let $U\in\F$ be $\F$-characteristic in $Q$ and $A(S)$-invariant. Then $U=1$.
\end{lemma}

\begin{proof}
Assume $U\neq 1$. Since $U$ is $A(S)$-invariant, $N_\F(U)$ is full parabolic. Hence, $N_{\Ac(Q)}(U)\leq N_\F(U)\leq \m N$. Observe also that $C_{\Ac(Q)}(V(Q))\leq N_\F(\Omega(Z(S)))\leq \m N$ and hence, as $U\ch Q$, $\Ac(Q)\leq \m N$. This is a contradiction to Remark~\ref{RemAc}.
\end{proof}

\textit{The proof of Theorem~\ref{CasesL}.}
Choose a pair $(Q,U)$ such that $Q\in \m{Q}$, $1\neq U\ch Q$ and $|N_S(U)|$ is maximal. Moreover, choose $U$ so that $|U|\geq |U_0|$ for all $1\neq U_0\ch Q$ with $U_0\unlhd N_S(U)$. Note that $U$ is fully normalized by Corollary~\ref{ThRestrConj2}. So the maximal choice of $|U|$ yields $U\in \m{C}^*(Q)$. Hence, by Lemma~\ref{C*C}, $Q\in \m{C}(Q)$. Set 
$$G:=G(Q),\;T:=N_S(Q) \mbox{ and } {M^*}:=C_G(V(Q))J(G).$$ 
Observe that it is sufficient to show the following properties.
\begin{itemize}
\item[(a)] $N_S(Q)=N_S(U)$.
\item[(b)] ${M^*}/Q\cong SL_2(q)$ and one of the following hold:
\begin{itemize}
\item[(I)] $Q$ is elementary abelian, $|Q|\leq q^3$ and $Q/C_Q({M^*})$ is a natural $SL_2(q)$-module for ${M^*}/Q$.
\item[(II)] $p=3$, $T=S$, $|Q|=q^5$, $\Phi(Q)=C_Q({M^*})$, and $Q/V(Q)$ and $V(Q)/\Phi(Q)$ are natural $SL_2(q)$-modules for ${M^*}/Q$.
\end{itemize}
\end{itemize}
For the proof of (a) and (b) set 
$$T_1:=B(T),\; T_0:=N_S(U) \mbox{ and } Q_1:=Q\cap T_1.$$
The maximal choice of $T_0=N_S(U)$ together with Lemma~\ref{trivial} yields the following property.

\bigskip

(1)\;\; Let $1\neq C\leq T_1$ such that $C\ch Q$. Then $C$ is not $A(T_0)$-invariant and, if $S\neq T_0$, $C$ is not normal in $N_S(T_0)$. 

\bigskip

By Lemma~\ref{BTH}, we can now choose $H\leq N_{M^*}(U)$ such that $T_1\in Syl_p(H)$, $H$ is normalized by $T$, ${M^*}=C_G(V(Q))H$ and $(H/O_p(H))/\Phi(H/O_p(H))\cong L_2(q)$. Observe that $Q_1=O_p(H)$.
Note that, by Lemma~\ref{PUThm}, $T_1=B(T_0)$ and so every characteristic subgroup of $T_1$ is $A(T_0)$-invariant. Therefore, (1) implies that $H$ fulfills Hypothesis~\ref{ModHyp} with $V(Q)$ in place of $W$. Thus, by Theorem~\ref{PUGrCor} one of the following holds for $V:=[Q_1,O^p(H)]$.
\begin{itemize}
\item[(I')] $V\leq \Omega(Z(Q_1))$ and $V/C_V(H)$ is a natural $SL_2(q)$-module for $H/C_H(V(Q))$.
\item[(II')] $Z(V)\leq Z(Q_1)$, $p=3$, $\Phi(V)=C_V(H)$ has order $q$, $V/Z(V)$ and $Z(V)/\Phi(V)$ are natural $SL_2(q)$-modules for $H/C_H(V(Q))$.
\end{itemize}
Furthermore, the following hold for every $\phi\in Aut(T_1)$ with $V\phi\not\leq Q_1$.
\begin{itemize}
\item[(i)] $Q_1=VC_{Q_1}(L)$ for some subgroup $L$ of $H$ with $O^p(H)\leq L$ and $H=LQ_1$.
\item[(ii)] $\Phi(C_{Q_1}(O^p(H)))\phi=\Phi(C_{Q_1}(O^p(H)))$.
\item[(iii)] If (II') holds then $V\leq V(Q)\<(V(Q)\phi)^H\>\leq Q_1$.
\item[(iv)] If (II') holds then $T_1$ does not act quadratically on $V/\Phi(V)$.
\item[(v)] $V\not\leq Q\phi$.
\item[(vi)] If (II') holds then $Q_1\phi^2=Q_1$.
\end{itemize}
If $U_0\leq Q_1$ for $U_0=\<V^{A(T_0)}\>$ or for $U_0=\<V^{N_S(T_0)}\>$, then $[U_0,O^p(H)]\leq V\leq U_0$ and $U_0\ch Q$. Together with (1) this gives the following property.

\bigskip

(2)\;\; There is $\phi \in A(T_0)$ such that $V\phi\not\leq Q_1$. If $S\neq T_0$ then we may choose $\phi$ such that $\phi\in S_{T_0}$.

\bigskip

Let now $\phi\in A(T_0)$ such that $V\phi\not\leq Q_1$.
Recall that, by Lemma~\ref{PUThm}, $T_1=B(T_0)$ and hence $T_1\phi=T_1$.
Set $$D:=\Phi(C_{Q_1}(O^p(H))).$$
Note that, as $Q_1$ and $H$ are $T$-invariant, $D$ is normal in $T$ and so $\F$-characteristic in $Q$. By Lemma~\ref{ThRestrConj2a} and (ii), $Q\phi\in \Q$ and $D=D\phi \ch Q\phi$. Assume $D\neq 1$. By Corollary~\ref{ThRestrConj2}, there is $\psi\in Mor_\F(N_S(D),S)$ such that $D\psi$ is fully normalized, $Q\psi, Q\phi\psi\in \Q$, and $D\psi$ is $\F$-characteristic in $Q\psi$ and $Q\phi\psi$. Hence, by Corollary~\ref{InQ}, $D^*:=O_p(N_\F(D\psi))\leq Q\psi\cap Q\phi\psi$. As $\F$ is minimal, $N_\F(D)$ is solvable and thus constrained. Hence, 
$$V(Q\psi)V(Q\phi\psi)\leq C_{N_S(D\psi)}(D^*)\leq D^*\leq (Q\psi)\cap (Q\phi\psi).$$ 
In particular, $V(Q)\phi\leq Q$. If (I') holds then $V\leq V(Q)C_V(X)=V(Q)Z(T_1)$ and so $V\phi\leq Q$, contradicting the choice of $\phi$. Therefore (II') holds. 
Observe that $H_Q\psi^*\leq N_\F(D\psi)$ as $H_Q\leq N_\F(D)$.\footnote{Recall Notation~\ref{phi*}.} Hence, $H_Q\psi^*$ normalizes $D^*$, and $V_0:=V(Q\psi)\<V(Q\phi\psi)^{H_Q\psi^*}\>\leq D^*\leq Q\phi\psi$. This implies $V(Q)\<V(Q\phi)^H\>=V_0\psi^{-1}\leq Q\phi$. Then by (iii), $V\leq Q\phi\cap T_1=Q_1\phi$, a contradiction to (v). 
This proves $D=1$ and so we have shown that

\bigskip

(3)\;\;$C_{Q_1}(O^p(H))$ is elementary abelian.

\bigskip

We show next that (a) holds. For the proof assume $T<N_S(U)$. Then there is $x\in (N_S(U)\cap N_S(T))\backslash T$. Since $J(Q_1)=J(Q)$ and $N_S(J(Q))=T$, we have $Q_1^x\neq Q_1$. By (3) and (i), $Q_1=VC_{Q_1}(H)=VZ(T_1)$ and so $V^x\not\leq Q_1$. On the other hand, $U\in \m{C}(Q)$ and so, by Corollary~\ref{InQ}, $O_p(N_\F(U))=U^*(Q)=U$. As $U=U^x\leq Q\cap Q^x$ and $N_\F(U)$ is constrained, we get $V(Q)V(Q)^x\leq C_{N_S(U)}(U)\cap T_1\leq U\cap T_1\leq Q_1$. If (I') holds then $V\leq V(Q)Z(H)\leq V(Q)Z(T_1)$ and so $V^x\leq Q_1$, a contradiction. Hence (II') holds. Then, by (iii), we have $V\leq V(Q)\<V(Q^x)^H\>\leq U$. So $V\leq U\cap T_1\leq Q^x$, a contradiction to (v). This proves (a). The choice of $(Q,U)$ together with Corollary~\ref{ThRestrConj2} and Lemma~\ref{trivial} gives now the following property.

\bigskip

(4)\;\;For every $1\neq U_0\ch Q$, we have $T=N_S(U_0)$ and $U_0$ is fully normalized. In particular, $U_0$ is not $A(T)$-invariant.

\bigskip

We show next:

\bigskip

(5)\;\;Let $\alpha\in A(T)$ such that $Q_1\alpha\neq Q_1$. If (I') holds then $\Phi(Q)\alpha=\Phi(Q)$.

\bigskip

For the proof of (5) assume that (I') holds and $\alpha$ is as in (5). Then $Q_1=V(Q)Z(T_1)$, so we have $V(Q)\alpha\not\leq Q_1$. By \ref{natSL2q}(b), $V(Q\alpha)$ is not an over-offender on $V(Q)$ and vice versa, so $|V(Q)/C_{V(Q)}(V(Q\alpha))|=|V(Q\alpha)/C_{V(Q\alpha)}(V(Q))|$. Hence, $V(Q\alpha)$ is an offender on $V(Q)$ and vice versa. So, again by Lemma~\ref{natSL2q}(b), $J(T)Q=V(Q\alpha)Q$ and $J(T)(Q\alpha)=V(Q)(Q\alpha)$. In particular, $[J(T),Q\alpha]\leq Q$ and, by Remark~\ref{(IV)}(d), $Q\alpha\leq J(T)Q$. Hence, $J(T)Q=J(T)(Q\alpha)=V(Q\alpha)Q=V(Q)(Q\alpha)$. In particular, $Q\alpha=V(Q\alpha)(Q\cap Q\alpha)$ and $Q=V(Q)(Q\cap Q\alpha)$. This yields
$\Phi(Q\alpha)=\Phi(Q\cap Q\alpha)=\Phi(Q)$ and proves (5). 

\bigskip

From now on let $\alpha\in A(T)$ such that $Q_1\alpha\neq Q_1$. Note that $\alpha$ exists by (4). We show now:

\bigskip

(6)\;\;If (I') holds then $Q$ is elementary abelian.

\bigskip

Let $\beta\in A(T)$ such that $Q_1\beta=Q_1$. Then $Q_1\beta\alpha\neq Q_1$. If (I') holds then (5) yields $\Phi(Q)\alpha=\Phi(Q)=\Phi(Q)\beta\alpha$ and hence $\Phi(Q)=\Phi(Q)\beta$. Thus, by (5), $\Phi(Q)$ is $A(T)$-invariant. Now (4) implies $\Phi(Q)=1$, so (6) holds. We show now:

\bigskip

(7)\;\;$C_Q(H)\cap (C_Q(H)\alpha)=1$.

\bigskip
For the proof of (7) assume $U_1:=C_Q(H)\cap (C_Q(H)\alpha)\neq 1$. 
By Lemma~\ref{ThRestrConj2a}, we have $Q\alpha\in \Q$. Note that $U_1\ch Q$ and $U_1\ch Q\alpha$. In particular, by (4), $U_1$ is fully normalized. Moreover, Corollary~\ref{InQ} implies $U_1^*:=O_p(N_\F(U_1))\leq Q\cap Q\alpha$. By Corollary~\ref{MinCharpType}, $N_\F(U)$ is constrained. Hence, $$V(Q)V(Q\alpha)\leq C_{N_S(U_1)}(U_1^*)\leq U_1^*\leq Q\cap Q\alpha.$$ 
If (I') holds then, by (6), $Q=V(Q)$ and so $Q=Q\alpha$, contradicting the choice of $\alpha$. By (i) and (3), $Q_1=VZ(T_1)$, so $V\alpha\not\leq Q_1$. Hence, if (II') holds then, by (iii),  $V\leq V(Q)\<V(Q\alpha)^H\>\leq U_1^*\leq Q\alpha$, contradicting (v). This shows (7).

\bigskip

It remains to show that (I') implies (I) and (II') implies (II). Assume first (I') holds. By (6), $Q$ is elementary abelian. Hence, $Q=V(Q)$, $C_G(V(Q))=Q$, ${M^*}/Q\cong SL_2(q)$ and $Q/C_Q({M^*})$ is a natural $SL_2(q)$-module for ${M^*}/Q$. By (7), $C_Q({M^*})\cap (C_Q({M^*})\alpha)=1$. This implies
$$|C_Q({M^*})|=|(C_Q({M^*})\alpha)C_Q({M^*})/C_Q({M^*})|
\leq |Z(J(T))/C_Q({M^*})|\leq q.$$
Hence, $|Q|\leq q^3$ and (I) holds.

\bigskip

Assume from now on that (II') holds. Note that, for $W:=Z(Q_1)$, $W/C_W(H)$ is a natural $SL_2(q)$-module for $H/Q_1$. Hence, $|Z(T_1)/C_{Q_1}(H)|=|C_W(T_1)/C_W(H)|\leq q$. Now (7) yields
$$|C_{Q_1}(H)|=|(C_{Q_1}(H)\alpha)C_{Q_1}(H)/C_{Q_1}(H)|\leq |Z(T_1)/C_{Q_1}(H)|\leq q$$
and so $C_{Q_1}(H)=C_V(H)$. Now by (i) and (3), $Q_1=VC_{Q_1}(H)=V$. In particular, by (iii), $Q_1=V=V(Q)\<(V(Q)\phi)^H\>$. So $Q_1=V$ is generated by elements of order $p$ and $[Q_1,Q_1]=\Phi(Q_1)=C_{Q_1}(H)$. As $Q_1/Z(Q_1)$ is an irreducible module for $H$, $[Q_1,Q]\leq Z(Q_1)$ and so $[Q_1,Q,Q_1]=1=[Q,Q_1,Q_1]$. Now the Three-Subgroups Lemma implies $[C_{Q_1}(H),Q]=[Q_1,Q_1,Q]=1$. Observe that $Z(Q_1)=V(Q)C_{Q_1}(H)$ and so $[Z(T_1),Q]\leq [Z(Q_1),Q]=1$. The definition of $T_1$ gives now $[\Omega(Z(J(T))),Q]=1$ and $Q=Q_1=V$. In particular, by (vi), every automorphism of $T$ of odd order normalizes $Q$. Hence, $Q$ is normal in $N_S(T)$ and so $T=S$. If $[\ov{Q},C_G(V(Q))]\neq 1$ for $\ov{Q}=Q/C_Q(H)$, then $\ov{Q}$ is the direct sum of two natural $SL_2(q)$-modules for $H/C_H(V(Q))$ and so $[\ov{Q},T_1,T_1]=1$, a contradiction to (iv). Thus, $[\ov{Q},C_G(V(Q))]=1$ and, if $x\in C_G(V(Q))$ has order prime to $p$, then $[Q,x]=[Q,x,x]\leq [C_Q(H),x]\leq [V(Q),x]=1$. Hence, $C_G(V(Q))=Q$. This shows ${M^*}/Q\cong SL_2(q)$ and (II) holds. Thus, the proof of Theorem~\ref{CasesL} is complete.

\bigskip

\textit{The proof of Theorem~\ref{Classify}.} Theorem~\ref{Classify} follows from Theorem \ref{mainSL2qThm}, Theorem~\ref{CasesL}, Corollary~\ref{MinCharpType} and Theorem~\ref{ClassifyG}.

\bigskip

\textbf{\Large Acknowledgements}

\bigskip

I would like to thank Sergey Shpectorov for many useful discussions. The concept of a minimal fusion system is a generalization of a concept suggested to me by him. I thank Michael Aschbacher for pointing out that my original definition of a minimal fusion system is equivalent to the one given in this paper. Also, I am very thankful to Bernd Stellmacher who suggested applying results from \cite{BHS}. Moreover, some arguments which appear in Sections 8, 10 and 11 go back to ideas from him. Apart from that, I thank him for his thorough reading of, and suggestions regarding Sections 9, 10 and 11.

\end{document}